\documentclass[moor]{informs1}

\usepackage{endnotes}
\let\footnote=\endnote

\usepackage{listings}
\usepackage{tabularray}
\usepackage{natbib}
\bibpunct[, ]{(}{)}{,}{a}{}{,}%
\def\bibfont{\fontsize{8}{9.5}\selectfont}%

\usepackage[small, margin=1cm]{caption}

\usepackage{multirow}
\usepackage{subfig}
 \NatBibNumeric
 \def\bibfont{\small}%
 \bibpunct[, ]{[}{]}{,}{n}{}{,}%

\usepackage[colorlinks=true,breaklinks=true,bookmarks=true,urlcolor=blue,
     citecolor=blue,linkcolor=blue,bookmarksopen=false,draft=false]{hyperref}

\usepackage{enumitem}
\def\EMAIL#1{\href{mailto:#1}{#1}}

\def\blot{\quad \mbox{$\vcenter{ \vbox{ \hrule height.4pt
				\hbox{\vrule width.4pt height.9ex \kern.9ex \vrule width.4pt}
				\hrule height.4pt}}$}}
\TheoremsNumberedThrough     
\EquationsNumberedThrough    

\begin{document}


\RUNAUTHOR{Huang et~al.}

\RUNTITLE{Propagation of Input Tail Uncertainty in Rare-Event Estimation}

\TITLE{Propagation of Input Tail Uncertainty in Rare-Event Estimation: A Light versus Heavy Tail Dichotomy}

\ARTICLEAUTHORS{%
\AUTHOR{Zhiyuan Huang}
\AFF{Department of Management Science and Engineering, Tongji University, \EMAIL{huangzy@tongji.edu.cn}
}
\AUTHOR{Henry Lam}
\AFF{Department of Industrial Engineering and Operations Research, Columbia University, \EMAIL{henry.lam@columbia.edu}
}
\AUTHOR{Zhenyuan Liu}
\AFF{Department of Industrial Engineering and Operations Research, Columbia University, \EMAIL{zl2817@columbia.edu}
}
} 



\ABSTRACT{%
We consider the estimation of small probabilities or other risk quantities associated with rare but catastrophic events. In the model-based literature, much of the focus has been devoted to efficient Monte Carlo computation or analytical approximation assuming the model is accurately specified. In this paper, we study a distinct direction on the propagation of model uncertainty and how it impacts the reliability of rare-event estimates. Specifically, we consider the basic setup of the exceedance of i.i.d. sum, and investigate how the lack of tail information of each input summand can affect the output probability. We argue that heavy-tailed problems are much more vulnerable to input uncertainty than light-tailed problems, reasoned through their large deviations behaviors and numerical evidence. We also investigate some approaches to quantify model errors in this problem using a combination of the bootstrap and extreme value theory, showing some positive outcomes but also uncovering some statistical challenges.

}%







\maketitle

\section{Introduction.}
\label{sec:intro}

Estimating small probabilities and risk quantities associated with rare but catastrophic events is ubiquitous in risk analysis and management. Examples include the prediction of large asset losses in finance \cite{Glasserman04,chan2010efficient}, cash flow imbalances in insurance \cite{ASM00Ruin,asmussen2006improved}, system overloads in service operations \cite{kroese1999efficient,blanchet2009rare,blanchet2014rare,dupuis2009importance} and breakdowns in communication systems \cite{Heidelberger95,RT09,nicola1993fast,nicola2001techniques}. In model-based approaches, namely when a model is built to capture the internal dynamics of the system, much of the focus in the literature has been devoted to efficient Monte Carlo computation or analytical approximation, assuming the model is accurately specified (e.g., \cite{bucklew2013introduction,blanchet2012state,juneja2006rare}). However, in virtually all cases, the target rare-event quantity computed from the model is affected by input misspecifications, where any inaccuracies in calibrating the model can be propagated to the outputs and make the target estimate erroneous or even meaningless.



The study of the impacts from input misspecifications has gathered growing attention in recent years, generally known as the problem of model uncertainty or input uncertainty. Its main focus is to develop methodologies that can quantify the errors in the output estimation or decision attributed to model misspecifications, measured in terms of bounds or variance. See, e.g., \cite{barton2002panel,henderson2003input,chick2006bayesian,barton2012input,song2014advanced,lam2016advanced,corlu2020stochastic,barton2022input} in the stochastic simulation literature, and \cite{petersen2000minimax,hansen2008robustness,glasserman2014robust,lim2007relative} in finance, economics, control and operations management applications. In extremal estimation, this problem is intimately related to extreme value theory, in which one attempts to extrapolate the tail beyond the scope of data in a  statistically justified fashion, along with uncertainty quantification \cite{embrechts2013modelling,embrechts2005quantitative,resnick2008extreme,haan2006extreme}. Recently, the framework of so-called distributionally robust optimization \cite{delage2010distributionally,wiesemann2014distributionally,ben2013robust,ghaoui2003worst} has been studied to construct bounds on extremal measures with additional robustness properties beyond statistical asymptotics. This approach utilizes postulations such as the acknowledgement of the true distribution within a neighborhood of a baseline model \cite{atar2015robust,blanchet2020distributionally,engelke2017robust}, marginal information and extremal coefficients \cite{wang2013bounds,bernard2017risk,embrechts2006bounds,puccetti2013sharp,wang2011complete,dhara2021worst,puccetti2012computation,embrechts2013model,yuen2020distributionally,lam2018sensitivity}, moments and shape assumptions on the tail such as monotonicity or convexity \cite{lam2017tail,van2019distributionally,li2019ambiguous,li2018worst,roos2022tight,bai2023distributionally}. In simulation-based rare-event analysis, \cite{nakayama1995asymptotics,nakayama1998derivative} studied methods to efficiently compute sensitivities of rare-event probabilities with respect to model parameters, and \cite{nelson2021reducing} proposed an averaging of distributions to fit input models to enhance tail performances. 

In contrast to the past literature above that focused on techniques to quantify the impact of model uncertainty, in this paper we aim to understand the significance of the impact itself in relation to the level of input information and propagation mechanics in rare-event problems. Specifically, we consider the basic problem of the large exceedance probability of an i.i.d. sum, where the input model refers to the probability distribution governing each summand. Here, suppose we have input data that informs the summand's distribution. Our first question would be: 
\\

\begin{enumerate}
\item[1.] \emph{By simply using the empirical distribution as my input model fit, would the rare-event probability be reasonably close to the truth (assuming computational or Monte Carlo error is negligible)?}
\\
\end{enumerate}


To address Question 1, our viewpoint is that the main uncertainty in the input model that undermines the accuracy of rare-event estimation is the lack of knowledge of its tail. The central, non-tail part of the input distribution can typically be fit by parametric or nonparametric techniques, where there are adequate data to perform such fit (and Question 1 above simply uses the empirical distribution). In contrast, the tail portion is often ill-informed due to inadequate data, yet it can have a high impact on the aggregated tail behavior. 
Thus, to address Question 1, we first focus on:
\\

\begin{enumerate}
\item[0.] \emph{How does truncating the tail of the input model affect the rare-event estimate?}
\\
\end{enumerate}

Our main contention is that heavy-tailed problems suffers from a much larger impact from input model truncation than light-tailed counterparts. This truncation level represents the amount of tail knowledge obtained from the data, such as the maximum of all data or remaining data after we discard a fixed number of largest data points. As a consequence, using the empirical distribution with 
even a moderate data size, as asked in Question 1, fails to estimate a heavy-tailed rare-event quantity reliably. On the other hand, the effect of missing tails on light-tailed estimation is relatively milder and hence reliable estimation in this case is achievable with much less data requirement. 

The above different effects from truncating the input tail can be explained from the different large deviations behaviors in heavy- versus light-tailed systems. Specifically, the former pertains to the one or several ``big jumps" \cite{embrechts2013modelling,denisov2008large,rhee2019sample,nair2022fundamentals,blanchet2008efficient}, i.e., to invoke a rare event, one or several input components exhibit huge values, while the latter invokes rare events by having each component contribute a small shift and adding these contributions \cite{dembo2009large,bucklew2004introduction}. Thus, to accurately estimate a heavy-tailed rare event, one needs to accurately estimate the far tail of each input component, whereas this is not necessary in light-tailed systems. The mathematical analysis to rigorize these intuitions relies on Berry-Ess\'{e}en expansions and exact large deviations asymptotics, which allow us to compare the rare-event probabilities under the true and the truncated distributions. In particular, these expansions and asymptotics are derived for triangular arrays with growing truncation levels, which requires elaborate analyses and are new to our best knowledge. 

 

With the above, we go further to ask whether we can statistically assess the error arising from input uncertainty, specifically:
\\

\begin{enumerate}
\item[2.] \emph{With the point estimate in Question 1, could the bootstrap give rise to valid confidence intervals that account for the input error?}
\item[3.] \emph{Would incorporating extreme value theory in fitting the input tail lead to more reliable uncertainty quantification?}
\\
\end{enumerate}

To address Questions 2 and 3, we ran extensive experiments to evaluate the performances of bootstrap and extreme-value-theory-enhanced bootstrap techniques across various scenarios involving different tail distributions, sample sizes, and target probabilities. We find that, particularly in situations exhibiting heavy-tailed behaviors, even when the sample size is relatively large, i.e., approaching the inverse of the target probability, using the empirical bootstrap in Question 2 that significantly ignores the tail content would fail to estimate the rare-event quantity and vastly under-estimate the uncertainty. This uncertainty under-estimation could come from a power-law tail or non-power-law subexponential tail. Using extreme value theory in Question 3 to extrapolate tail (such as the peak-over-threshold method, e.g., \cite{leadbetter1991basis}) helps to an extent, but could introduce extra bias, at least using our fitting methods. The introduction of bias is primarily due to the presence of the unknown slowly varying function in the true data distribution, making the generalized Pareto distribution a misspecified tail model. On the positive side, our numerical experiment demonstrates that the extreme index estimator effectively identifies cases prone to uncertainty under-estimation, irrespective of whether it is caused by a power-law tail or non-power-law subexponential tail. This identification serves as an indicator that additional data collection is necessary to improve the precision and reliability of the estimation and uncertainty quantification.


Finally, we contrast our study of Question 0 to some related works \cite{dupuis2020sensitivity,blanchet2014robust,lam2015simulating} and especially \cite{olveracravioto2006,jelenkovic1999network}. \cite{dupuis2020sensitivity} investigated the sensitivities on the large deviations rate when the input model deviates within a R\'{e}nyi divergence ball. \cite{blanchet2014robust} showed in a similar context that imposing a single ball over all inputs, thus allowing the distortion of dependency structure among the inputs, can lead to a substantially heavier tail than the original model when the Kullback-Leibler divergence is used. \cite{lam2015simulating} studied robust rare-event simulation when the input tail is unknown but subject to geometric assumptions. \cite{olveracravioto2006} studied the impacts on the waiting times when the tail of service times is misspecified or truncated. Relating to \cite{jelenkovic1999network}, they investigated the truncation threshold needed to retain the heavy-tailed characteristic of a system. They also contrasted it with the light-tailed case and observed that the required threshold is higher for heavy tail. Our observation in this regard is thus similar to \cite{olveracravioto2006}, but with a different setting (aggregation of i.i.d. variables), which requires us to derive elaborately the Berry-Ess\'{e}en expansion and exact large deviations asymptotics for truncated distributions in order to compare the rare-event probabilities under the truncated and the true distributions. Moreover, we focus on the statistical implications asked in Questions 1-3, validate our theory numerically and investigate error assessment schemes.

The remainder of this paper is as follows. Section \ref{sec:setting} describes our estimation target and explains the impacts of tail truncation in light- versus heavy-tailed cases. Section \ref{sec:data driven scenario} discusses the data requirement thresholds for reliable and unreliable estimations. Section \ref{sec:experiments} shows numerical results in the use of truncated distributions, empirical distributions, bootstrapping, and the detection of heavy tails. 
The appendix contains all the proofs and supplemental results. 


Throughout this paper, for any sequences $a_n,b_n\in\mathbb R$ we write  $a_n= o(b_n)$ if $a_n/b_n\to0$ as $n\to\infty$, $a_n= \omega(b_n)$ if $a_n/b_n\to\infty$ as $n\to\infty$, $a_n=O(b_n)$ if there exists an integer $n_0$ such that $|a_n/b_n|\leq \overline M$ for $n\geq n_0$ and $0<\overline M<\infty$, and $a_n=\Theta(b_n)$ if there exists an integer $n_0$ such that $\underline M\leq|a_n/b_n|\leq \overline M$ for $n\geq n_0$ and $0<\underline M\leq\overline M<\infty$.

\section{Setting and Theory.}\label{sec:setting}
We consider estimating the overshoot of an aggregation of $n$ i.i.d. variables, i.e., consider $p=P(S_n>\gamma)$ where $S_n=X_1+\cdots+X_n$ and $X_i\in\mathbb R$ are i.i.d. variables drawn from the distribution $F$. We denote $X$ as a generic copy of $X_i$ for convenience. We assume the density of $X$ exists and denote as $f$. Correspondingly, we let $\bar F(x)=1-F(x)$ be the tail distribution function. We let $\mu=E[X]<\infty$. Suppose $\gamma=\gamma(n)$ is a high level that grows to $\infty$ as $n\to\infty$. 


Our investigation pertaining to Question 0 is the following. Suppose we truncate the distribution $F(x)$ at the point $u$ so that the density becomes 0 for $x>u$, i.e., consider the truncated distribution function given by 
$$\tilde F_u(x)=\left\{\begin{array}{ll}
F(x)/F(u)&\text{for\ \ }x\leq u\\
1&\text{for\ \ }x>u\end{array}\right. $$
and correspondingly the truncated density $\tilde f_u(x)=(f(x)/F(u))I(x\leq u)$, where $I(\cdot)$ denotes the indicator function. For convenience, denote $p(G)$ as the probability $P_G(S_n>\gamma)$ where $X_i$'s are governed by an arbitrary distribution $G$, and we simply denote $P(S_n>\gamma)$ if $X_i$'s are governed by $F$. We would like to investigate the approximation error $p(\tilde F_u)-p(F)$, which is by definition
\begin{equation}
p(\tilde F_u)-p(F)=\frac{P(S_n>\gamma,X_i\leq u\ \forall i=1,\ldots,n)}{F(u)^n}-P(S_n>\gamma). \label{error}
\end{equation}

Note that, roughly speaking, the above situation captures the case where we use the empirical distribution to plug into our input model, so that the probability mass is zero for regions outside the scope of data or close to zero at the very tail of the data. The proportional constant $F(u)$ is introduced to ensure a proper truncated distribution and has little effect on the mass below $u$ when $u$ is reasonably big.

\subsection{Heavy-Tailed Case.}\label{sec:theory-heavytail}
We first consider the Pareto-tail case. Suppose that $\bar F$ has a regularly varying tail in the form
\begin{equation}
\bar F(x)=L(x)x^{-\alpha}\label{rv}
\end{equation}
for some slowly varying function $L(\cdot)$ and $\alpha>2$, and $E|X|^{2+\delta}<\infty$.

Suppose $n\to\infty$ and $\gamma=n\mu+\Theta(n)$ (or more generally $\gamma=n\mu+\omega(\sqrt{n\log n}))$. In this case, it is known that $P(S_n>\gamma)$ is approximately $P(\text{at least one\ }X_i>\gamma-n\mu)$, or probabilistically, that the rare event $S_n>\gamma$ happens most likely due to a big jump from one of the $X_i$'s (e.g., \cite{embrechts2013modelling,nair2022fundamentals}). Thus, if the truncation level $u$ is too small compared to $\gamma-n\mu$, then the big jump that contributes to the dominating mass of the rare event is barred, making $P(S_n>\gamma,X_i\leq u\ \forall i=1,\ldots,n)$ substantially smaller than $P(S_n>\gamma)$. In this situation, $p(\tilde F_u)$ becomes negligible compared to $P(S_n>\gamma)$, and the approximation error \eqref{error} is effectively $-P(S_n>\gamma)$. In other words, using a truncated input distribution leads to a substantial under-estimate with a bias almost equal to the magnitude of the rare-event probability itself.

Alternately, we can write the approximation error \eqref{error} as 
\begin{equation}
\frac{-P(S_n>\gamma,\text{\ at least one\ }X_i>u)+P(S_n>\gamma)(1-F(u)^n)}{F(u)^n}.\label{interim}
\end{equation}
Again, when $u$ is relatively small compared to $\gamma-n\mu$, then the event $\{\text{at least one\ }X_i>u\}$ inside the probability $P(S_n>\gamma,\text{\ at least one\ }X_i>u)$ is redundant, making this probability asymptotically equivalent to $P(S_n>\gamma)$ and that \eqref{interim} is asymptotically equivalent to $-P(S_n>\gamma)$.

We summarize the above as:
\begin{theorem}[Unreliable approximation of $p(F)$ for heavy-tailed distributions]
Suppose $X_i$'s are i.i.d. random variables with regularly varying tail distribution $\bar F$ in the form \eqref{rv} with $\alpha>2$ and $E|X|^{2+\delta}<\infty$. Let $n\to\infty$ and $\gamma=n\mu+\omega(\sqrt{n\log n})$. Assume $u\leq \mu+O((\gamma-n\mu)/\sqrt{\log n})$. The discrepancy between using a truncated distribution $\tilde F_u$ and the original distribution $F$ in evaluating the probability $p(F)=P(S_n>\gamma)$ as $n\to\infty$ is given by
$$ p(\tilde F_u)-p(F)=-p(F)(1+o(1)).$$
\label{heavy}
\end{theorem}

On the other hand, when the truncation level is large enough, the probability $P(S_n>\gamma,\text{\ at least one\ }X_i>u)$ will be negligible compared to $p(F)=P(S_n>\gamma)$, and $F(u)^n\rightarrow1$. From (\ref{interim}), the error would be asymptotically negligible. This is summarized as the following theorem:
\begin{theorem}[Reliable approximation of $p(F)$ for heavy-tailed distributions]
Suppose $X_{i}$'s are i.i.d. random variables with regularly varying tail distribution $\bar{F}$ in the form \eqref{rv} with $\alpha>2$ and $E|X|^{2+\delta}<\infty$. Let $n\rightarrow\infty$ and $\gamma=n\mu+\omega(\sqrt{n\log n})$. Assume $u=\omega((\gamma-n\mu)^{\beta})$ for some $\beta>1$. The discrepancy between using a truncated distribution $\tilde{F}_{u}$ and the original distribution $F$ in evaluating the probability $p(F)=P(S_{n}>\gamma)$ as $n\rightarrow\infty$ is asymptotically negligible, i.e.,
\[
p(\tilde{F}_{u})-p(F)=o(p(F)).
\]
\label{heavy_large_u}
\end{theorem}

Consider the case $\gamma=bn$ for some $b>\mu$. Theorem \ref{heavy} states that when $u$ is below $(b-\mu)n/\sqrt{\log n}$, the rare-event estimation is essentially void, at least asymptotically. On the other hand, Theorem \ref{heavy_large_u} ensures that the estimation is reliable when $u$ is larger than $n^\beta(b-\mu)^\beta$. In other words, the truncation level should grow faster than a linear function of $n$ to get a reliable estimation of the rare-event probability. When the number of input components $n$ is large, it could be difficult to sustain an accuracy level given a finite set of input data.

\subsection{Light-Tailed Case.}
We now consider $X$ that possesses finite exponential moment, i.e., the logarithmic moment generating function $\psi(\theta)=\log E[e^{\theta X}]<\infty$ for $\theta$ in a neighborhood of 0. Define $\mathcal{D=}\{\theta:\psi\left(  \theta\right)  <\infty\}$. Consider $\gamma=bn$ for some constant $b>\mu$. Suppose that there exists a unique solution $\theta^*\in\mathcal{D}^{\circ}$ to the equation $b=\psi'(\theta)$. Then $p(F)=P(S_n>\gamma)$ exhibits exponential decay as $n\to\infty$, i.e., $-(1/n)\log p(F)\to I$ where $I$ is the rate function given by the Legendre transform or the convex conjugate of $\psi(\theta)$ 
$$I=\sup_{\theta}\left\{b\theta-\psi(\theta)\right\}=b\theta^*-\psi\left(  \theta^*\right).$$
In fact, if $X$ is further assumed non-lattice, we have the following more accurate asymptotic from Theorem 3.7.4 in \cite{dembozeitouni1998} (although the asymptotic there is for $P(S_n\geq\gamma)$, the proof can be applied to $P(S_n>\gamma)$) 
\begin{equation}
p(F)=\frac{1}{\theta^*\sqrt{\psi''(\theta^*)2\pi n}}e^{-nI}(1+o(1)),\label{asymptotic_original_nonlattice} 
\end{equation}
and if $X$ is lattice, i.e., $(X-x_{0})/h$ is a.s. an integer for some $x_{0}$ and $h$, and $h$ is the largest number with this property ($h$ is called the span of $X$) and $0<P(X=b)<1$, we can also get an accurate asymptotic from Remark 8.1.5 and Remark 8.1.6 in \cite{bucklew2013introduction}%
\begin{equation}
p(F)=\frac{he^{-\theta^*h}}{(1-e^{-\theta^*h})\sqrt{\psi^{\prime\prime}(\theta^*)2\pi n}}e^{-nI}(1+o(1)).\label{asymptotic_original_lattice}
\end{equation}
For the truncated distribution, we similarly define the logarithmic moment generating function by $\psi_{u}\left(  \theta\right)  =\log E\left[
\left.  e^{\theta X}\right\vert X\leq u\right]  $. Let us consider the non-lattice case as an example. Just like the asymptotic (\ref{asymptotic_original_nonlattice}) of the original distribution, we can approximate $p(\tilde{F}_u)$ by $p(\tilde{F}_u) \approx e^{n\psi_{u}(  \theta_{u}^*)  -nb\theta_{u}^*} /\sqrt{\psi_{u}^{\prime\prime}\left(  \theta_{u}^*\right)2\pi n  } \theta_{u}^*$, where $\theta_{u}^*$ is the solution to the equation $\psi_{u}^{\prime}\left(  \theta\right)  =b$. If the truncation level $u$ is large enough, $\psi_{u}\left(  \theta\right)  $ will converge to $\psi\left(  \theta\right)  $ so fast that the error caused by replacing $\theta_{u}^*$ and $\psi_{u}$ with $\theta^*$ and $\psi$ will be negligible, i.e.,
\begin{equation}
p(\tilde{F}_u) \approx \frac{e^{n\psi_{u}(  \theta_{u}^*)  -nb\theta_{u}^*}} {\sqrt{\psi_{u}^{\prime\prime}\left(  \theta_{u}^*\right)2\pi n  }\theta_{u}^*}  \approx \frac{e^{n\psi(  \theta^*)  -nb\theta^*}}{\sqrt{\psi^{\prime\prime}\left( \theta^*\right)2\pi n  }\theta^*}.
\label{approx_truncated}
\end{equation}
In fact, if we assume that the choice of $u$ satisfies that $\exists~\theta^{\prime}>0$ s.t. $\theta^*+\theta^{\prime}
\in\mathcal{D}^{\circ}$ (such $\theta^{\prime}$ exists since $\theta^*\in\mathcal{D}^{\circ}$) and
\begin{equation}
\lim_{n\rightarrow\infty}\frac{n}{e^{\theta^{\prime} u}}=0,\label{restrictions_on_un}
\end{equation}
then the approximation in (\ref{approx_truncated}) can be made precise as follows.

\begin{lemma}[Exact asymptotics for truncated distributions]
\label{asymptotic_truncated} Consider $\gamma=bn$ for some constant $b>\mu$. Suppose $0\in\mathcal{D}^{\circ}$ where $\mathcal{D=}\{\theta:\psi\left(\theta\right)  <\infty\}$, and there exists a unique solution $\theta^{\ast}\in\mathcal{D}^{\circ}$ to the equation $b=\psi^{\prime}(\theta)$. Further, suppose the choice of $u$ satisfies (\ref{restrictions_on_un}).

\begin{description}
\item[(1)] If the distribution of $X$ is non-lattice, then we have%
\[
\lim_{n\rightarrow\infty}\theta^*\sqrt{\psi^{\prime\prime}(\theta^*)2\pi n}e^{nb\theta^*-n\psi(\theta^*)}p(\tilde{F}_{u})=1.
\]

\item[(2)] Suppose $X$ has a lattice distribution, i.e. $(X-x_{0})/h$ is a.s. an integer for some $x_{0}$ and span $h$. Assume further that $0<P(X=b)<1$, i.e., $b$ is in the lattice of $X$. Then we have%
\[
\lim_{n\rightarrow\infty}\theta^*\sqrt{\psi^{\prime\prime}(\theta^*)2\pi n}e^{nb\theta^*-n\psi(\theta^*)}p(\tilde{F}_{u})=\frac{\theta^*he^{-\theta^*h}}{1-e^{-\theta^*h}}.
\]

\end{description}
\end{lemma}


The proof of Lemma \ref{asymptotic_truncated} is quite lengthy because we need to establish a Berry-Ess\'{e}en expansion for a specific triangular array in order to obtain the exact large deviations asymptotic (\cite{dembozeitouni1998} Theorem 3.7.4). In particular, it requires many precise estimates on the characteristic function that generalize the case of i.i.d. sequence (\cite{feller2008introduction} Section XVI.4 Theorem 1; \cite{gnedenko1968limit} Section 43 Theorem 1). Although \cite{garcia1998edgeworth} establishes Edgeworth expansions for general triangular arrays, it is unclear if its technical assumption (2.1) holds for our characteristic function in (\ref{varphi_n}) (in the Appendix) and thus we resort to a self-contained analysis of the triangular array Berry-Ess\'{e}en expansion. 


Based on Lemma \ref{asymptotic_truncated}, (\ref{asymptotic_original_nonlattice}) and (\ref{asymptotic_original_lattice}), we can see $p(F)$ and $p(\tilde{F}_u)$ are asymptotically equal. Therefore, the relative error of using $p(\tilde{F}_u)$ to estimate $p(F)$ is asymptotically negligible as summarized in the following theorem:

\begin{theorem}[Reliable approximation of $p(F)$ for light-tailed distributions]
Consider $\gamma=bn$ for some constant $b>\mu$. Suppose that either $F$ is lattice with $0<P(X=b)<1$ or $F$ is non-lattice. Moreover, suppose $0\in\mathcal{D}^{\circ}$ where $\mathcal{D=}\{\theta:\psi\left(\theta\right)  <\infty\}$, and there exists a unique solution $\theta^{\ast}\in\mathcal{D}^{\circ}$ to the equation $b=\psi^{\prime}(\theta)$. Then, as long as the truncation level $u$ is chosen such that (\ref{restrictions_on_un}) is satisfied, the discrepancy between using a truncated distribution $\tilde{F}_{u}$ and the original distribution $F$ in evaluating the probability $p(F)=P(S_{n}>\gamma)$ is asymptotically negligible, i.e.,
\[
p(\tilde{F}_{u})-p(F)=o(p(F)).
\]
\label{light}
\end{theorem}

Theorem \ref{light} postulates that as long as the truncation level $u$ is chosen high enough such that (\ref{restrictions_on_un}) is satisfied, then the model error in using the truncated input is negligible. In contrast to the heavy-tailed case, this condition on $u$ dictates typically a logarithmic requirement on $n$. To illustrate the choices of $u$ for different distributions, consider the following Examples \ref{normal ex}-\ref{ex generalization}, the first of which show increasingly stringent requirements of $u$ as the tail heavinesss increases.

\begin{example}
Suppose $F$ is the normal distribution $N\left(  \mu,\sigma^{2}\right)  $. Direct calculation shows that%
\[
\psi\left(  \theta\right)  =\mu\theta+\frac{1}{2}\sigma^{2}\theta^{2},\text{ }\mathcal{D}=\mathbb{R},\text{ }\theta^*=\frac{b-\mu}{\sigma^{2}}.
\]
The choice of $u$ can be any number satisfying $u\geq c\log n$ for some $c>0$.\label{normal ex}
\end{example}

\begin{example}
Suppose $F$ is the exponential distribution $\text{Exp}\left(  \lambda\right)  $. Direct calculation shows that%
\[
\psi\left(  \theta\right)  =\log\lambda-\log\left(  \lambda-\theta\right) \text{ for } \theta<\lambda,\text{ }\mathcal{D}=\left(  -\infty,\lambda\right)  ,\text{ }\theta^*=\lambda-\frac{1}{b}.
\]
The choice of $u$ can be any number satisfying $u\geq c\log n$ for some $c>b$.\label{exp ex}
\end{example}

\begin{example}
Suppose $F$ is the gamma distribution $\Gamma\left(  \alpha,\beta\right)  $. Direct calculation shows that%
\[
\psi\left(  \theta\right)  =-\alpha\log\left(  1-\frac{\theta}{\beta}\right)\text{ for }\theta<\beta,\text{ }\mathcal{D}=\left(  -\infty,\beta\right)  ,\text{ }\theta^*=\beta-\frac{\alpha}{b}.
\]
The choice of $u$ can be any number satisfying $u\geq c\log n$ for some $c>b/\alpha$.\label{gamma ex}
\end{example}

\begin{example}\label{ex generalization}
In this example, let us generalize the three examples above. Suppose $F$ admits a density function $f\left(  x\right)  $ satisfying $f\left(  x\right) =O(e^{-\lambda x})$ as $x\rightarrow\infty$ for some $\lambda>0$. Further, suppose $\theta^*<\lambda$. Then the choice of $u$ can be any number satisfying $u\geq c\log n$ for some $c>1/(\lambda-\theta^*)$.
\end{example}

In contrast to the heavy tail setting, we do not establish results about unreliable estimation in the light tail case since the conditions are subtle. For example, consider an extreme case where $X$ is bounded. Then the estimation $p(\tilde{F}_u)$ would eventually have zero error as long as the truncation level $u$ goes to infinity. Moreover, Theorem \ref{light} is enough to support our point that the estimation problem is harder for the heavy tail case (see the discussion in Section \ref{sec:data driven scenario}).

Finally, besides the quantity $P(S_{n}>\gamma)$ we have been targeting, we can also consider $P(S_{n}\ge\gamma)$. Although the reliability results for probability estimation under truncated distributions do not change, the asymptotics of the probabilities themselves change in the lattice case. For completeness, we present all results related to $P(S_{n}\ge\gamma)$ in Appendix \ref{sec:another target}.

\section{From Theoretical to Data-Driven Information Level.}\label{sec:data driven scenario}
In Section \ref{sec:setting}, we assume we fully know the truncated distribution $\tilde{F}_{u}$. In this section, we connect our developed theory there to data by combining our theorems and extreme value theory, aiming to give guidance on the reliability of rare-event probability estimation in terms of the data size.

Consider estimating $P(S_{n}>nb)$ where $S_{n}$ is the sum of $n$ i.i.d. random variables from an unknown distribution $F$ and $b>\mu=E_{F}[X]$. Suppose now we have $N$ i.i.d. samples $X_{1},X_{2},\ldots,X_{N}$ from $F$ so that we can construct an empirical distribution $\hat{F}_{N}$. We want to answer the question whether the estimation $P_{\hat{F}_{N}}(S_{n}>nb)$ of the true probability $P(S_{n}>nb)$ is reliable. A closely related question is how many samples we need to get a reliable approximation.

\subsection{Empirical Truncation Level.}
A natural approach to connect Theorems \ref{heavy} and \ref{light} to data is to set the truncation level $u$ there as the threshold informable by the data, namely we use a properly defined empirical truncation level. If we know the asymptotic order of the empirical truncation level with respect to the data size $N$, then we can transform the conditions on truncation level in Theorems \ref{heavy} and \ref{light} to conditions on $N$, and we can judge estimation reliability in terms of $N$. Note that a natural definition of the empirical truncation level should converge to infinity as $N\rightarrow\infty$, otherwise it will never approximate the true distribution even as $N\rightarrow\infty$ and lead to an inconsistent estimator. Due to this restriction, it appears not a good choice to define the empirical truncation level as, e.g., 0.95 empirical quantile since the empirical quantile converges to the true quantile instead of infinity.

For simplicity, we set the empirical truncation level as the sample maximum $M_{N}=\max(X_{1},X_{2},\ldots,X_{N})$ so that the empirical truncated distribution is exactly the empirical distribution $\hat{F}_{N}$. Alternatively, we can discard a fixed number, e.g., $20$ largest samples from the data and take the maximum of the remaining data as the truncation level. The latter does not affect our analysis below since the asymptotic orders of the truncation levels remain the same (Corollary 4.2.4 in \cite{embrechts2013modelling}). The remaining task is to analyze the asymptotic order of $M_{N}$ where we can borrow tools from the extreme value theory. In the following, we consider heavy- and light-tailed cases separately because $M_{N}$ behaves differently in these two cases. A summary of our analyses (with selected light-tailed distributions) is displayed in Table \ref{empirical_level_table}.

\begin{table}
\centering
\caption{Required data size for reliably estimating $p(F)=P(S_n > nb)$ for $S_n=X_1+\cdots+X_n,X_i \overset{i.i.d.}{\sim}F$ with mean $\mu$ and $b>\mu$.}
\label{empirical_level_table}
\resizebox{\columnwidth}{!}{
\begin{tblr}{
  cells = {c},
  hline{1,5} = {-}{0.08em},
  hline{2} = {-}{},
}
Distribution                                      & Reliable $u$         & Order of $M_N$             & Minimum $N$                           & Minimum $N$ (based on~$p(F)$)                             \\
{Heavy-tailed\\($\bar{F}(x)=x^{-\alpha}L(x)$)}    & $(n(b-\mu))^{\beta}$ & $N^{1/\alpha}$             & $(n(b-\mu))^{\alpha\beta}$            & $n^{\beta}/(p(F))^{\beta}$                                \\
{Light-tailed\\(normal $N(\mu,\sigma^2)$)}        & $c\log n, c>0$       & $\sqrt{2\sigma^{2}\log N}$ & $n^{\frac{c^{2}}{2\sigma^{2}}\log n}$ & $n^{-\frac{c^{2}\log p(F)}{(b-\mu)^{2}}\frac{\log n}{n}}$ \\
{Light-tailed\\(exponential with rate $\lambda$)} & $b\log n$            & $\log N/\lambda$           & $n^{\lambda b}$                       & $n^{1+\sqrt{-\frac{2}{n}\log(p(F))}}$                     
\end{tblr}
}
\end{table}

\subsection{Heavy-Tailed Case.}\label{sec:samplesize_heavy}
In the heavy-tailed case, suppose the unknown distribution $F$ has a tail $\bar{F}(x)=x^{-\alpha}L(x)$ for $\alpha>2$ where $L(x)$ is a slowly varying function. The index parameter $\alpha$ can be estimated by the tail index estimator. By Theorem 3.3.7 in \cite{embrechts2013modelling}, we know that $c_{N}^{-1}M_{N}\overset{d}{\rightarrow}\Phi_{\alpha}$ where $\Phi_{\alpha}$ is the Fr\'{e}chet distribution with parameter $\alpha$ and the normalizing constant $c_{N}$ satisfies $c_{N}=N^{1/\alpha}L_{1}(N)$ for some slowly varying function $L_{1}$. Thus we have a rough approximation $M_{N}\approx N^{1/\alpha}$. According to Theorem \ref{heavy}, when the truncation level $u\leq \mu+\overline{M}n(b-\mu)/\sqrt{\log n}$ for a fixed number $\overline{M}>0$, the approximation $p(\tilde{F}_{u})$ is not reliable. So when $M_{N}\approx N^{1/\alpha}\leq \mu+\overline{M}n(b-\mu)/\sqrt{\log n}$, i.e., $N\leq(\mu+\overline{M}n(b-\mu)/\sqrt{\log n})^{\alpha}\approx(\overline{M}n(b-\mu)/\sqrt{\log n})^{\alpha}$, the estimate $p(\hat{F}_{N})$ is not reliable. By Theorem \ref{heavy_large_u} and a similar argument, $p(\hat{F}_{N})$ is reliable when $N>(n(b-\mu))^{\alpha\beta}$ for some $\beta>1$. Moreover, by (1.25b) in \cite{nagaev1979large}, $b-\mu$ can be approximately represented by the true probability $p(F)$:
\begin{align}
& p(F)=P(S_{n}\geq nb)=P(S_{n}-n\mu\geq n(b-\mu))\approx n\bar{F}(n(b-\mu))\approx n(n(b-\mu))^{-\alpha}\nonumber\\
& \Rightarrow b-\mu\approx(p(F))^{-1/\alpha}n^{1/\alpha-1}.\label{relation_p_b_mu}
\end{align}
Plugging (\ref{relation_p_b_mu}) into the above conditions on $N$, we know that $p(\hat{F}_{N})$ is unreliable for $N\le\overline{M}^{\alpha}n/(p(F)(\sqrt{\log n})^{\alpha})$ while $p(\hat{F}_{N})$ is reliable for $N>n^{\beta}/(p(F))^{\beta}$.

\subsection{Light-Tailed Case.}\label{sec:samplesize_light}
The light-tailed case is complicated since the asymptotic order of $M_{N}$ varies for different distributions. So the conclusions are also different for different distributions. To illustrate, we consider two examples. 

First, suppose the true distribution $F$ is the exponential distribution $\text{Exp}(\lambda)$. By Example 3.5.6 in \cite{embrechts2013modelling}, $M_{N}/(\lambda^{-1}\log N)\rightarrow1$ a.s.. Thus we have a rough approximation $M_{N}\approx(\log N)/\lambda$. According to Theorem \ref{light}, when the truncation level $u\geq c\log n$ for some $c>b$, the approximation $p(\tilde{F}_{u})$ is reliable. So when $M_{N}\approx(\log N)/\lambda\geq c\log n>b\log n$, i.e., $N>n^{\lambda b}$, the estimate $p(\hat{F}_{N})$ is reliable. Like the heavy-tailed case, we can approximately represent $\lambda b$ by the true probability $p(F)$. By the large deviations principle and the second order Taylor expansion of $\log(1+x)$, we have
\begin{align}
& p(F)=P(S_{n}\geq nb)\approx\exp(-n(\lambda b-1-\log(\lambda b)))\approx\exp\left(  -\frac{n}{2}(\lambda b-1)^{2}\right)  \nonumber\\
& \Rightarrow\lambda b=1+\sqrt{-\frac{2}{n}\log(p(F))}.\label{relation_p_b_lambda}
\end{align}
Plugging (\ref{relation_p_b_lambda}) into $N>n^{\lambda b}$, we know that the estimate $p(\hat{F}_{N})$ is reliable for $N>n^{1+\sqrt{-2\log(p(F))/n}}$. 

Second, suppose the true distribution $F$ is the normal distribution $N(\mu,\sigma^{2})$. By Example 3.5.4 in \cite{embrechts2013modelling}, $(M_{N}-\mu)/\sqrt{2\sigma^{2}\log N}\rightarrow1$ a.s., which leads to an approximation $M_{N}\approx\mu+\sqrt{2\sigma^{2}\log N}\approx\sqrt{2\sigma^{2}\log N}$. By Theorem \ref{light}, when the truncation level $u\geq c\log n$ for arbitrarily small $c>0$, the approximation $p(\tilde{F}_{u})$ is reliable. So when $M_{N}\approx\sqrt{2\sigma^{2}\log N}\geq c\log n$, i.e.,
\begin{equation}
N\geq\exp\left(  \frac{c^{2}}{2\sigma^{2}}(\log n)^{2}\right)  =n^{\frac{c^{2}}{2\sigma^{2}}\log n}\label{normal_N}   
\end{equation}
for arbitrarily small $c>0$, the estimate $p(\hat{F}_{N})$ is reliable. Now we derive the required data size for a target probability level $p(F)$. According to the large deviations of $N(\mu,\sigma^{2})$, we have%
\[
-\frac{1}{n}\log p(F)\rightarrow\frac{1}{2}\frac{(b-\mu)^{2}}{\sigma^{2}}%
\]
as $n\rightarrow\infty$, which implies the following approximation%
\[
\frac{1}{2\sigma^{2}}\approx-\frac{\log p(F)}{n(b-\mu)^{2}}.
\]
Plugging it into (\ref{normal_N}), we obtain
\[
N>n^{-\frac{c^{2}\log p(F)}{(b-\mu)^{2}}\frac{\log n}{n}}%
\]
for reliably estimating a target probability level $p(F)$.%

The discussion above also explains why the estimation problem is usually harder for heavy-tailed distributions than light-tailed distributions. Suppose the target probability level $p(F)$ is fixed. For the heavy-tailed case, $p(\hat{F}_{N})$ is reliable when $N>n^{\beta}/(p(F))^{\beta}$ for some $\beta>1$. For light-tailed distributions such as exponential, $p(\hat{F}_{N})$ is reliable when $N>n^{1+\sqrt{-2\log(p(F))/n}}$. Since $\beta>1+\sqrt{-2\log(p(F))/n}$ for large $n$, we can see heavy tail requires more data than light tail to estimate the same level of rare-event probability.

We caution that the above analyses are intuitive but not fully rigorous. However, they give us guidelines on the (in)sufficiency of the data size in reliably estimating rare-event probabilities. The next section report numerical results to validate our analyses herein.

\section{Numerical Experiments.} \label{sec:experiments}


We investigate numerically the quality in estimating $p=P(S_{n}>nb)$ with finite input data. We consider different numbers of summands $n$, and vary the magnitude of $p$ by changing the value of $b$. Under this setting, given a certain pair of $n$ and $p$, the tail distribution can be viewed as the only variable in the experiments. We examine three groups of parametric distributions each representing a different type of tail:





\begin{itemize}
\item Regularly varying tails, including the generalized Pareto distribution and Student's $t$-distribution (positive half), whose densities are
\begin{equation*}
f_P(x)=(\xi x)^{-1-\frac{1}{\xi}}, x\geq \frac{1}{\xi}, \quad f_t(x) = \frac{2\Gamma((\nu+1)/2)}{\Gamma(\nu/2)\sqrt{\nu \pi}}\left(1+\frac{x^2}{\nu}\right)^{-\frac{\nu+1}{2}},x\geq0.
\end{equation*}
The tail index $\alpha$ (in \eqref{rv}), which reflects the tail heaviness, equals the inverse of the shape parameter $\xi$ in the generalized Pareto distribution and the degree of freedom $\nu$ in the $t$-distribution (according to the Karamata’s
theorem \cite{embrechts2013modelling}). These distributions satisfy the assumptions in Theorem~\ref{heavy}.  

\item Light-tailed distributions, including standard exponential, standard normal (positive half), and Weibull with the shape parameter $k=2.5$ whose density is given by
\begin{equation*}
f_W(x)=k x^{k-1} e^{-x^k},x\geq0.
\end{equation*}
Since Weibull with shape parameter $k=2.5$ has a lighter tail than normal, we call it ``light-tailed" Weibull distribution in our experiments. The above three distributions satisfy the assumptions in Theorem~\ref{light}. 

\item Several distributions that are out of the scope of our analyses, namely log-normal whose logarithm follows $N(-1/2,1)$, and Weibull with shape parameter $k=0.5$ (we call it ``heavy-tailed" Weibull distribution to distinguish from the light-tailed Weibull above). These distributions do not have a power-law tail, but they are subexponential and can exhibit heavy tail behaviors, e.g. the one-big-jump behavior \cite{asmussen1998subexponential}.

\end{itemize}

Our numerical investigations are divided into four parts. Section~\ref{sec:num_validation} provides empirical validation for the theoretical findings discussed in Section~\ref{sec:setting} by considering estimations using tail-truncated distributions as the input model. Section \ref{sec:tail_heaviness} evaluates the validity of the relations between reliable estimation and data size in Section \ref{sec:data driven scenario}. Section \ref{sec:bootstrap_exp} evaluates the performance of the bootstrap and its combination with extreme value theory tools, in particular the generalized Pareto distribution, in quantifying estimation uncertainty. Lastly, given the observed challenges in both estimation and uncertainty quantification, it appears that the most dependable strategy for ensuring accurate estimations is to have a sample size significantly larger than $1/p$ in cases where inadequate sample size could result in uncertainty under-estimation. In Section \ref{sec:tail_detection}, then, we discuss and recommend some diagnostic methods for detecting these cases. 




\subsection{Validating Theories on the Impacts of Tail Truncation.}\label{sec:num_validation}
We validate the theoretical findings discussed in Section~\ref{sec:setting} by demonstrating that heavy-tailed distributions are more vulnerable to truncated tails. In this set of experiment, we use the true distribution, conditional on being under a truncation point, to generate samples to evaluate the target rare-event probabilities. We use the 0.001 tail quantile as the truncation point, defined as $u$ such that $P(X>u)=0.001$. We compare the rare-event probabilities driven by the truncated distribution, denoted $p_u$, and those driven by the untruncated distribution, denoted $p$. We use the Gaussian distribution as a representative for light tail and the $t$-distribution for heavy tail. To measure the effects of tail heaviness, we consider the $t$-distribution with three different degrees of freedom $\nu=2.5,4,10$. We also consider various numbers of summands $n$. To enhance computational efficiency, we apply various variance reduction techniques including importance sampling for light-tailed cases and conditional Monte Carlo for heavy-tailed cases (for further details, see \cite{asmussen2007stochastic}).

Our results are displayed in Figure \ref{fig:figure_t_group2}. We maintain a consistent true probability level at approximately $10^{-5}$ across our settings, as shown in Figure \ref{fig:validate_prob}. Figure \ref{fig:validate_error} shows the relative errors, which appear to increase with the heaviness of the tails. In particular, for heavier-tailed $t$-distributions, specifically with $\nu=2.5,4$, the relative errors are close to 1, which indicates very poor estimation relative to the true magnitude of the probability. These observations align with the analysis presented in Section~\ref{sec:setting}, which asserts that the absence of tail information is particularly harmful for problems involving heavy-tailed distributions. 



\begin{figure}[!t]
     \centering
     \subfloat[ \label{fig:validate_prob} Probability estimates with untruncated distributions.    ]{\includegraphics[width=0.45\textwidth]{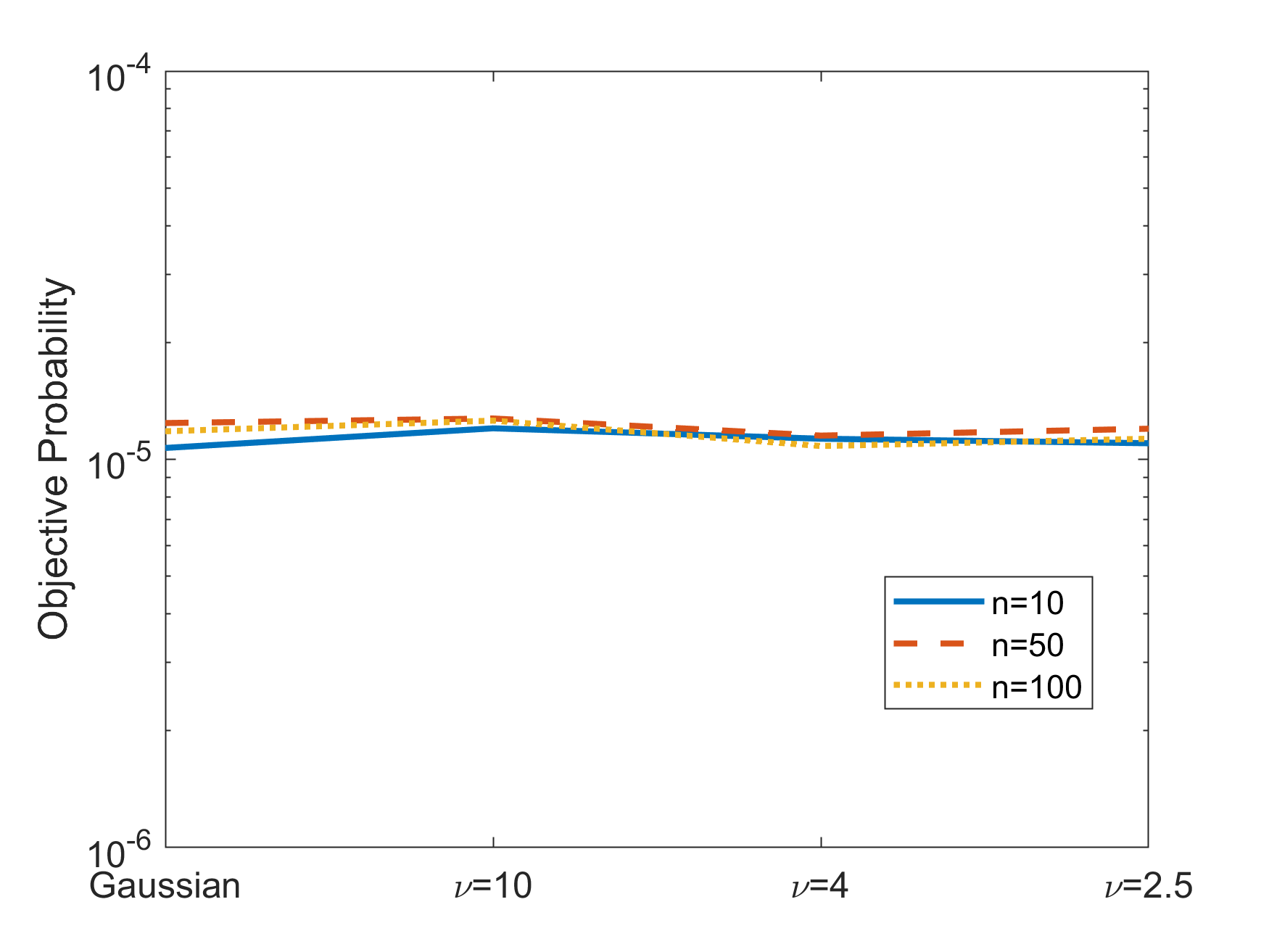}} \quad 
     \subfloat[\label{fig:validate_error} Relative errors given by $|p-p_u|/p$.]{\includegraphics[width=0.45\textwidth]{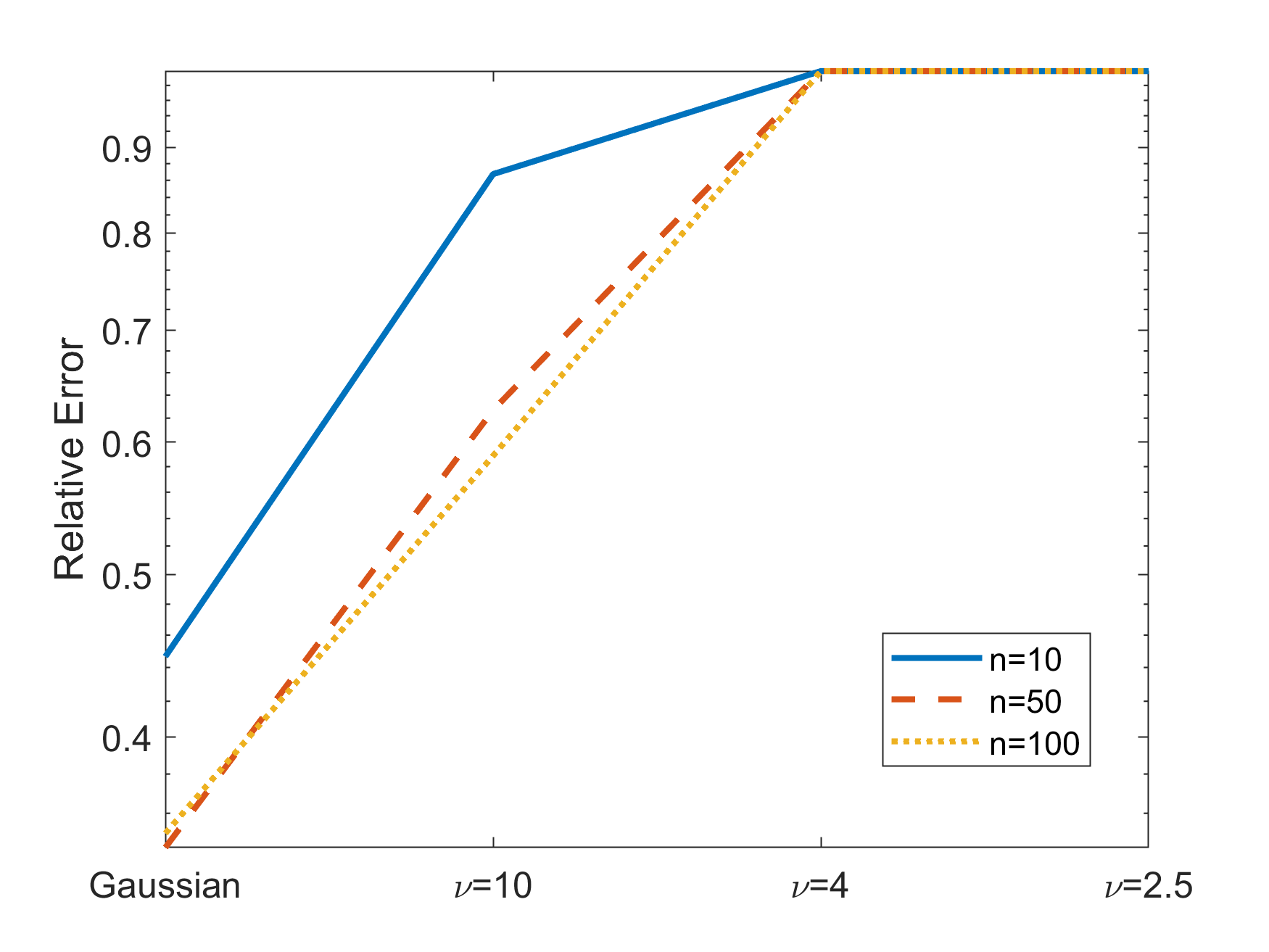}}
     \caption{Probability estimation with untruncated and truncated distributions. Each $p$ approximately has the same magnitude.  }
     \label{fig:figure_t_group2}
 \end{figure}
 
 

\subsection{Impacts of Tail Heaviness in Data-Driven Estimation.} \label{sec:tail_heaviness}
We investigate errors in rare-event probability estimation in relation to data size, for various tails and rare-event probability levels. We consider $n=10,50,100$ and $p\approx10^{-5},10^{-6},10^{-7}$. For each problem with a pair of $n$ and $p$,  we independently generate data sets of $X_{i}$'s with a sample size $N$. We use $p(\hat F_N)$ to approximate $p$, where $\hat F_N$ is the empirical distribution constructed from the size-$N$ data. For each setup, we repeat the experiment 100 times and report the relative errors (the ratio of approximation error to the ground truth) of these 100 experimental outcomes by using box plots. Each plot will show the median as a central red mark, and the top and bottom edges of the boxes will represent the 75\% and 25\% percentiles. The whiskers will extend outside the box to the extreme data within the range of 1.5 times of the box height. Data outside this range will be considered as outliers represented as a red cross. Reliability of the estimates will be determined by the coverage of the box. When the box covers 0, a smaller height suggests more concentrated estimates, which is a sign that the approximation is relatively accurate. If the box fails to cover 0 (the box would mostly be close to $-1$ in this case), then it means the approximation is poor due to insufficiency of using the empirical distribution with the given data size. 




Like in Section \ref{sec:num_validation}, variance reduction is needed to speed up computation. For heavy-tailed problems, we apply conditional Monte Carlo \cite{asmussen2006improved} when the sample size for constructing the empirical distribution is sufficient (more than $10^6$ samples). Since this method relies on the equation $P(S_n > \gamma)=n P(S_n > \gamma, X_n= \max\{X_1,...,X_n\})$, which does not hold when $X_i$ is discrete, the estimator is biased when simulations are based on the empirical distribution. However, as demonstrated in Appendix~\ref{sec:append_bias_cond_MC}, this bias is negligible. When we run crude Monte Carlo, we estimate $p$ from $10^{8}$ samples drawn from the empirical distribution; for the conditional Monte Carlo cases, we use $10^7$ samples. For problems involving light-tailed distributions, the conditional Monte Carlo method is also employed when there is a sufficient sample size. Note that,  despite the availability of more efficient alternatives, we select the latter method to exclude the influence of variations among different estimators. This decision does not affect the essential findings of our experiments, as the conclusions would remain consistent regardless of the use of a more efficient approach.


We unfold our experiment results in three directions. Section \ref{sec:exp_light_vs_heavy} compares heavy- with light-tailed distributions. Section \ref{sec:exp_heavy_heavier} compares distributions within the same family but with different tail heaviness. Section \ref{sec:exp_slow_varying} investigates the impacts of slowly varying functions and distributions in between light and power-law tails.

\subsubsection{Light versus Heavy Tail.}\label{sec:exp_light_vs_heavy}
\begin{figure}[t]
      \centering
     \subfloat[$n=10$, $p\approx10^{-5}$, $N=10^3$\label{fig:lvh_n10p5N3}
     ]{\includegraphics[width=0.45\textwidth]{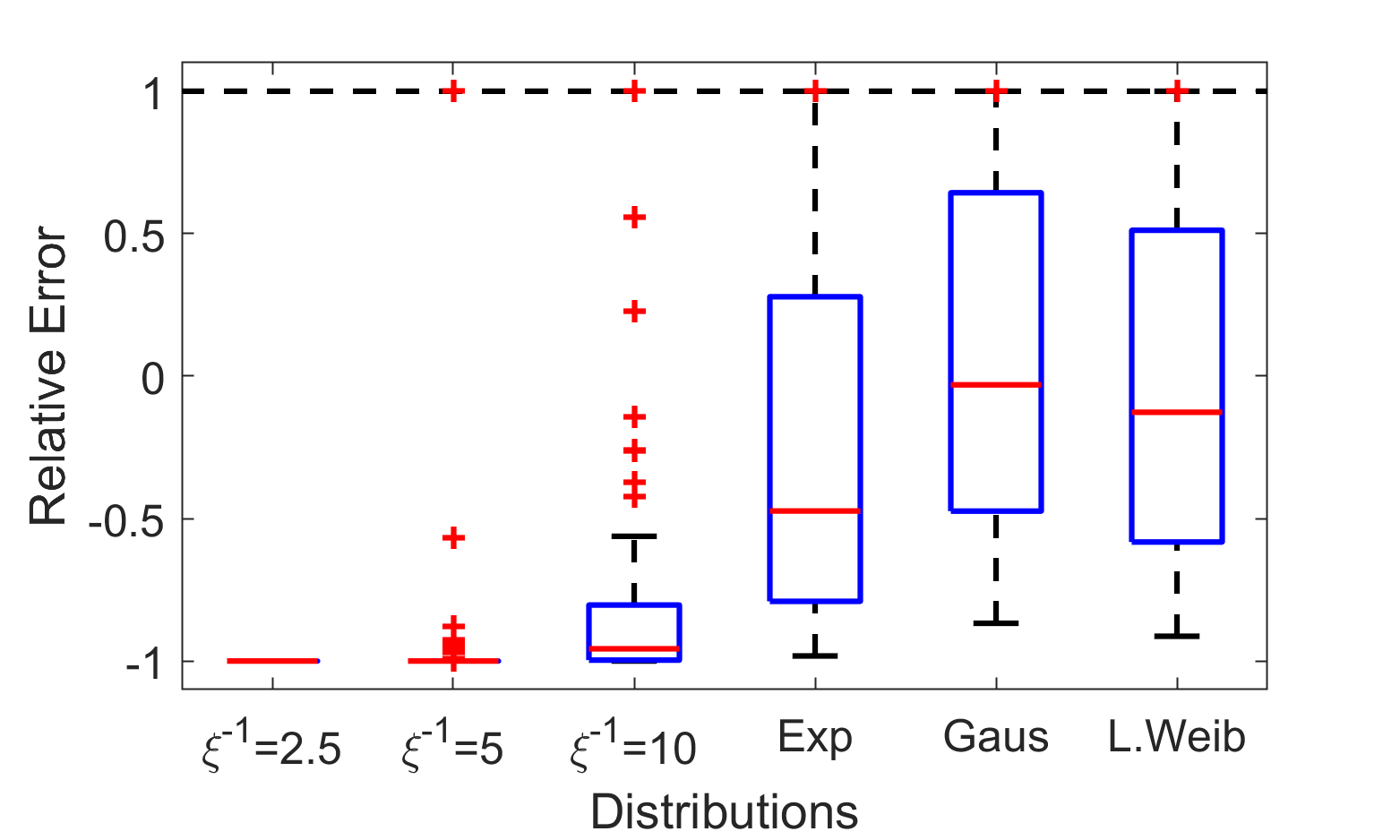}} \quad 
     \subfloat[$n=10$, $p\approx10^{-5}$, $N=10^5$\label{fig:lvh_n10p5N5}]{\includegraphics[width=0.45\textwidth]{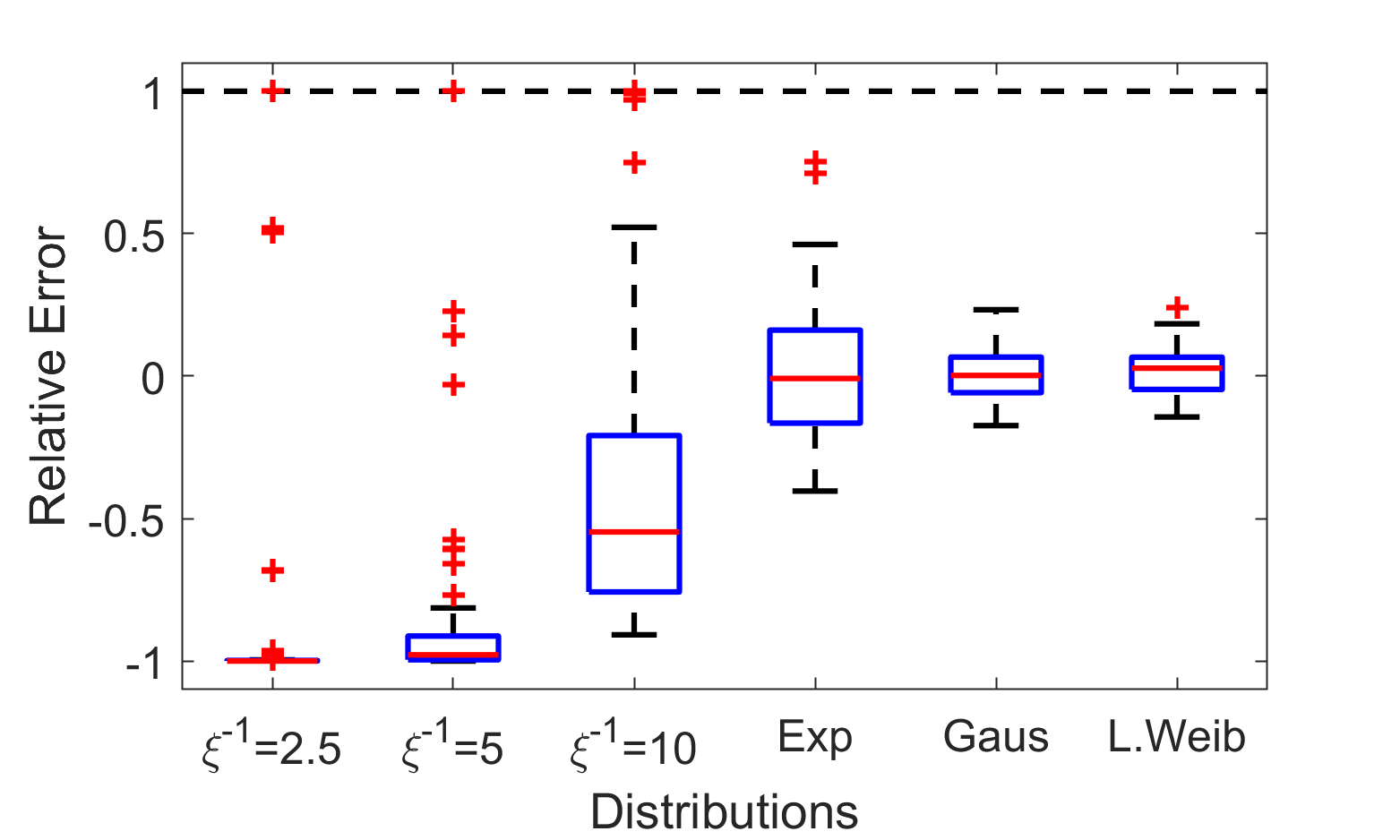}} \quad
      \subfloat[$n=10$, $p\approx10^{-5}$, $N=10^6$ \label{fig:lvh_n10p5N6}
     ]{\includegraphics[width=0.45\textwidth]{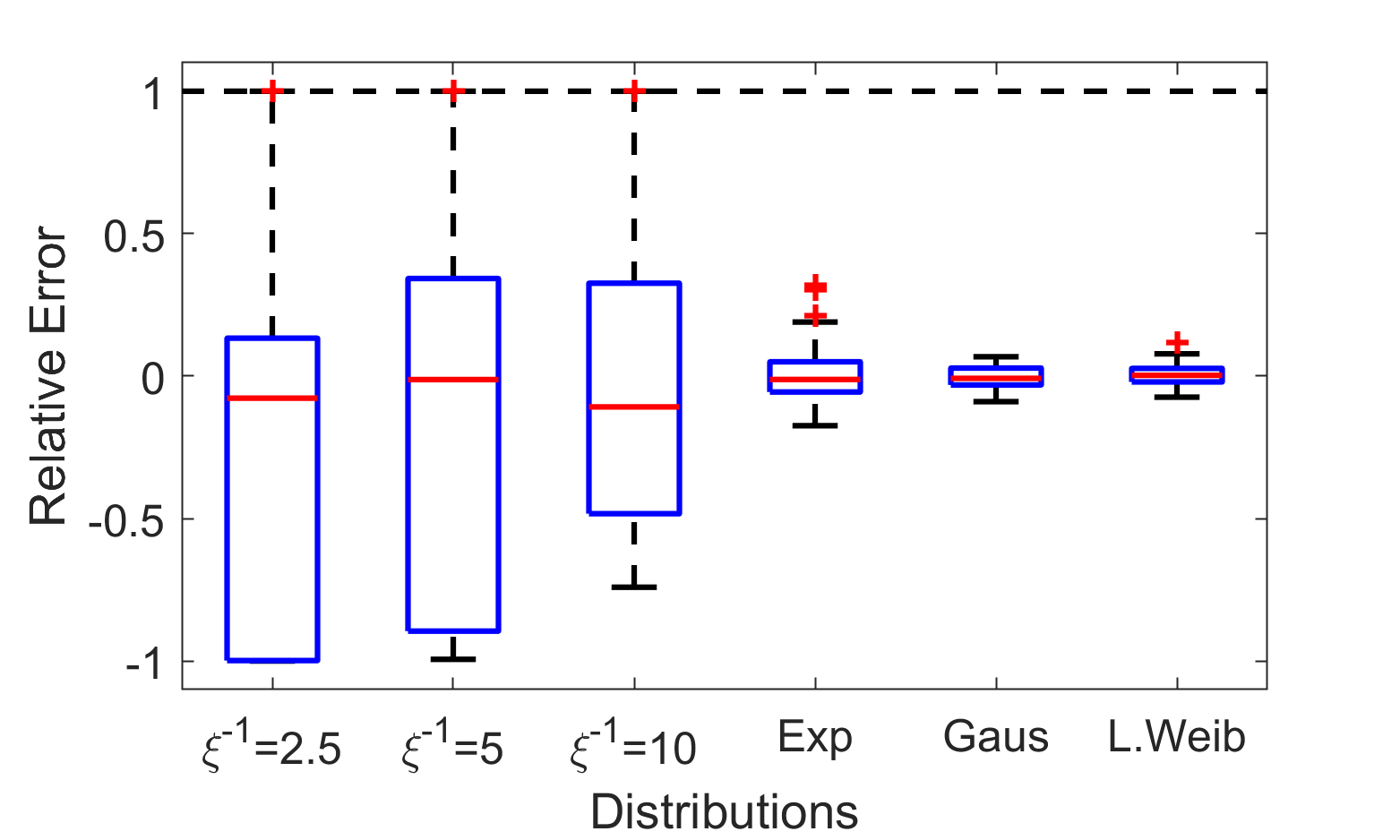}} \quad 
     \subfloat[$n=100$, $p\approx10^{-7}$, $N=10^3$\label{fig:lvh_n100p7N3}]{\includegraphics[width=0.45\textwidth]{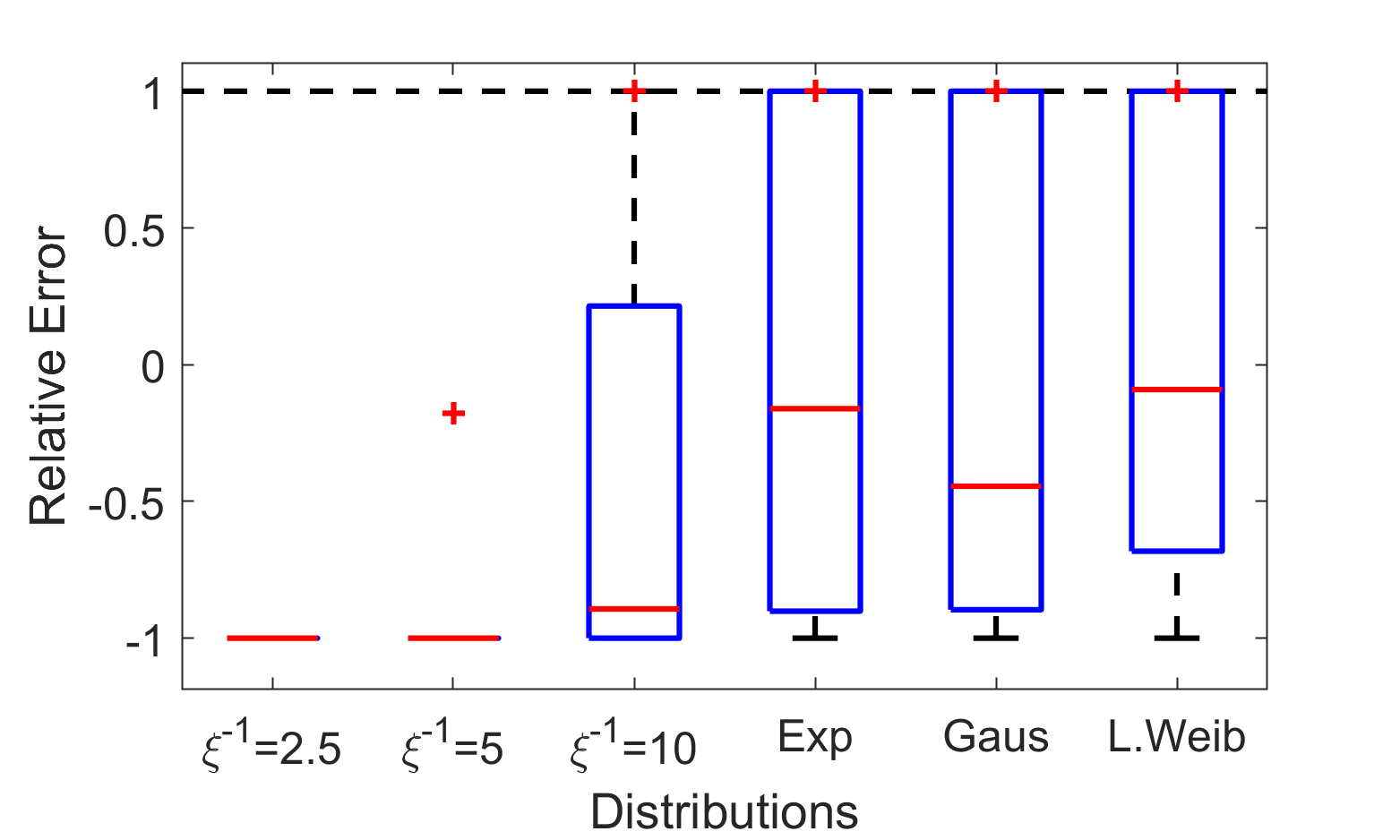}} \quad
     \subfloat[$n=100$, $p\approx10^{-7}$, $N=10^4$ \label{fig:lvh_n100p7N4}
     ]{\includegraphics[width=0.45\textwidth]{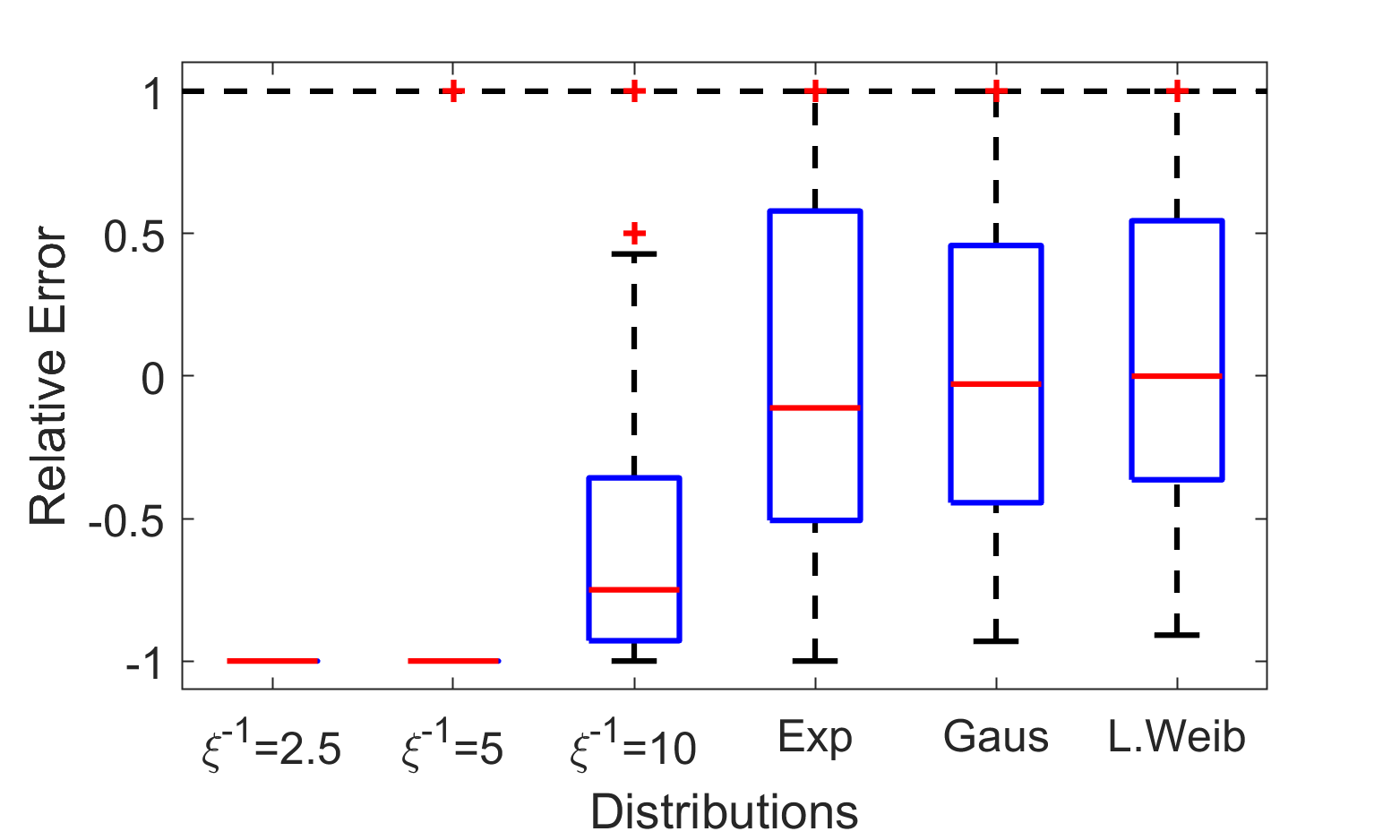}} \quad 
     \subfloat[$n=100$, $p\approx10^{-7}$, $N=10^8$\label{fig:lvh_n100p7N8}]{\includegraphics[width=0.45\textwidth]{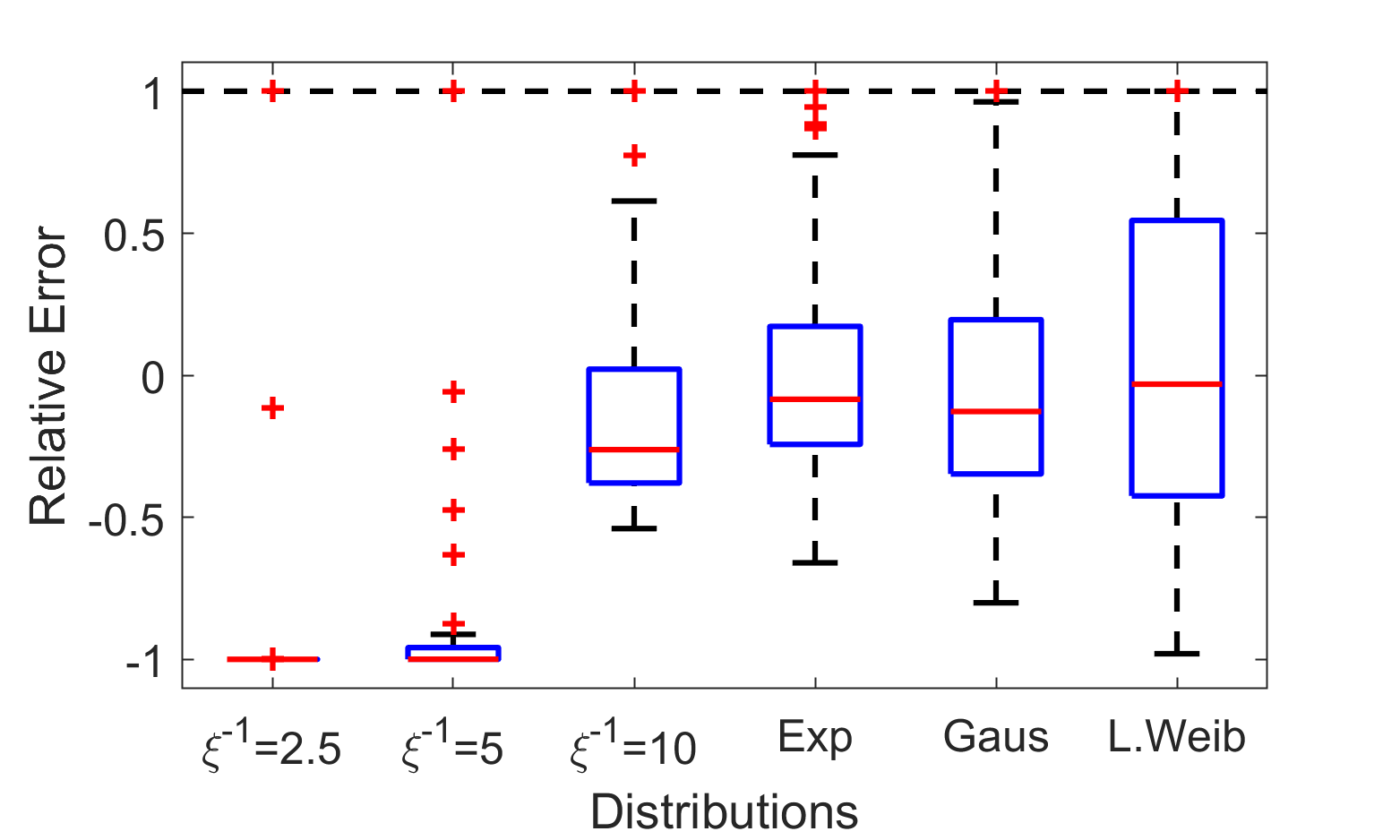}}
     \caption{Relative errors (the ratios of error to the ground truth). The compared tail distributions include generalized Pareto distributions with various tail indices (i.e. $\xi^{-1}$) and light-tailed distributions, including exponential (Exp), Gaussian (Gaus), and light-tailed Weibull (L.Weib).}
     \label{fig:light_vs_heavy}
 \end{figure}
 


Figures \ref{fig:lvh_n10p5N3}, \ref{fig:lvh_n10p5N5}, and \ref{fig:lvh_n10p5N6} show the comparisons among generalized Pareto distributions with tail indices $\xi^{-1}=2.5,5,10$ (heavy tail), and exponential, Gaussian and light-tailed Weibull (light tail). We set $n=10$ and $p\approx 10^{-5}$. In Figure \ref{fig:lvh_n10p5N3}, when the sample size is relatively small ($N=10^3$), estimates in the light-tailed cases are accurate, evidenced by the coverage of 0 by the boxes. However, heavy-tailed estimates significantly under-estimate the rare-event probability and their boxes fail to cover the truth in all three cases. As we increase the sample size $N$ to $10^5$ in Figure \ref{fig:lvh_n10p5N5}, the light-tailed cases start to have highly reliable estimates, which concentrate within $\pm 0.5$ relative error bounds. Meanwhile, the heavy-tailed boxes still cannot cover the truth. In Figure \ref{fig:lvh_n10p5N6}, we further increase the sample size to $10^6$ and observe that the heavy-tailed cases have similar performance as the light-tailed cases with $N=10^3$ in Figure \ref{fig:lvh_n10p5N3}, where the boxes barely cover the truth. On the other hand, with $N=10^6$ samples, the light-tailed estimates are already highly accurate with mostly $\pm 0.2$ relative errors.

Similar observations are found in Figures \ref{fig:lvh_n100p7N3}, \ref{fig:lvh_n100p7N4}, and \ref{fig:lvh_n100p7N8}, where we set $n=100$ and $p\approx 10^{-7}$. As we start with a small sample size $N=10^3$ in Figure \ref{fig:lvh_n100p7N3}, the light-tailed estimates are inaccurate (the box edges are close to $\pm 1$) but their boxes are capable of covering the truth, while the heavy-tailed estimates tend to under-estimate the probability. Even though when $\xi^{-1}=10$, the heavy-tailed box covers the truth, the median of the estimates is close to $-1$, which indicates more than half of the estimates have severe under-estimation. When the sample size is increased to $N=10^4$ in Figure \ref{fig:lvh_n100p7N4}, the light-tailed estimates become more concentrated. The box edges for the light-tailed cases are around $\pm 0.5$. For the heavy tail cases, all estimates have close to $-1$ relative error when $\xi^{-1}=2$ or $5$, while the $\xi^{-1}=10$ case has more concentrated box edges but still tend to under-estimate (median close to $-0.7$ and top edge lower than 0). Even when we increase the sample size to $N=10^8$ in Figure \ref{fig:lvh_n100p7N8}, we observe that the heavy-tailed cases with $\xi^{-1}=2$ or $5$ still cannot obtain reasonable estimates, where more than 75\% of the estimates have relative errors close to $-1$.

These results indicate that problems associated with heavy tails are more prone to under-estimation than light tails with the same amount of data, consistent with our discussions in Sections~\ref{sec:setting} and \ref{sec:data driven scenario}. Note that, in many heavy-tailed cases (e.g., Figure~\ref{fig:lvh_n100p7N4}), the relative errors approach $-1$, suggesting that the estimates are significantly smaller than the actual values which coincide with Theorem~\ref{heavy}. 

\subsubsection{Impacts of Tail Index in Heavy Tail.}\label{sec:exp_heavy_heavier}
\begin{figure}[t]
     \centering
     \subfloat[Pareto, $n=50$, $p\approx10^{-5}$, $N=10^6$\label{fig:hh_n50N6p5_p} 
     ]{\includegraphics[width=0.45\textwidth]{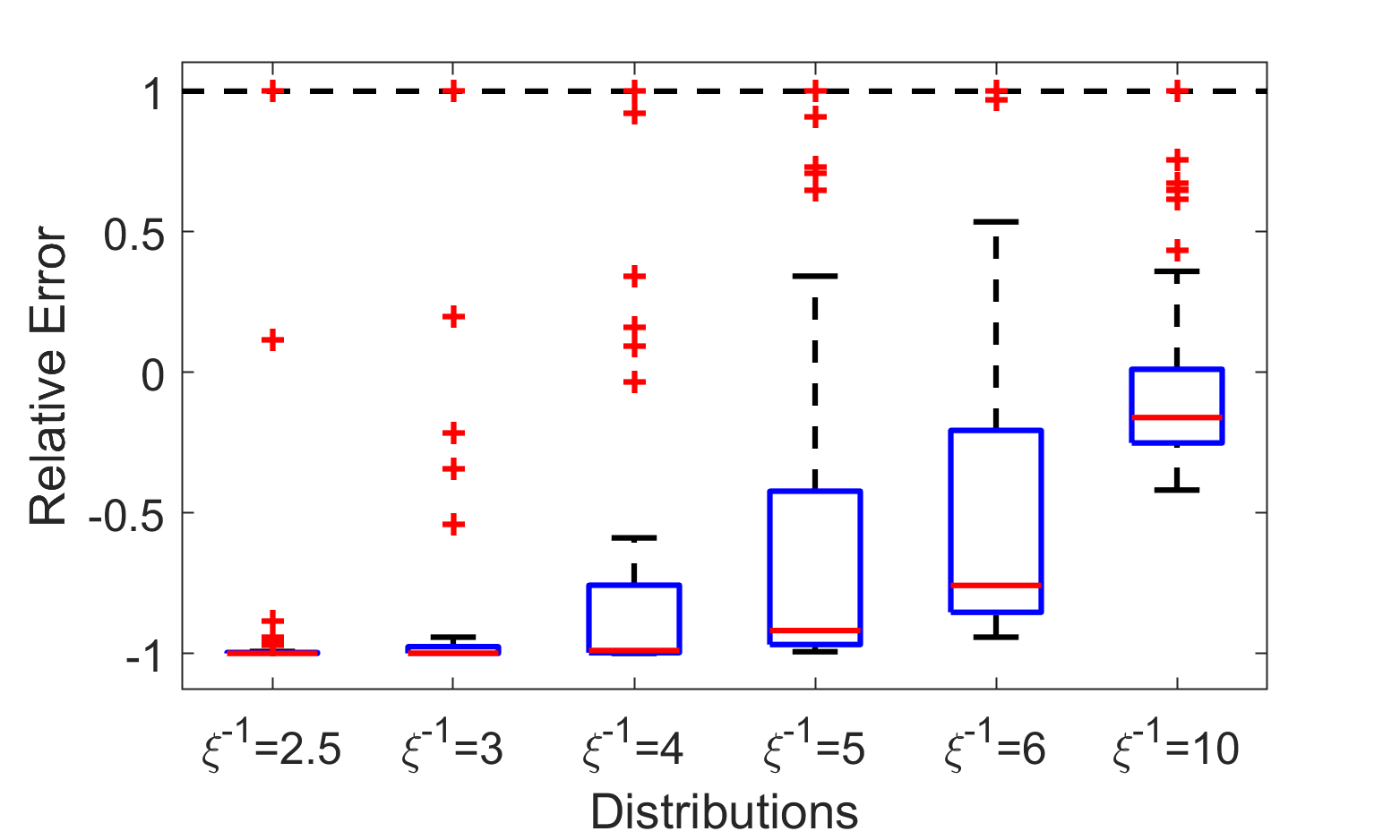}} \quad    
     \subfloat[Pareto, $n=50$, $p\approx10^{-6}$, $N=10^7$\label{fig:hh_n50N7p6_p}]{\includegraphics[width=0.45\textwidth]{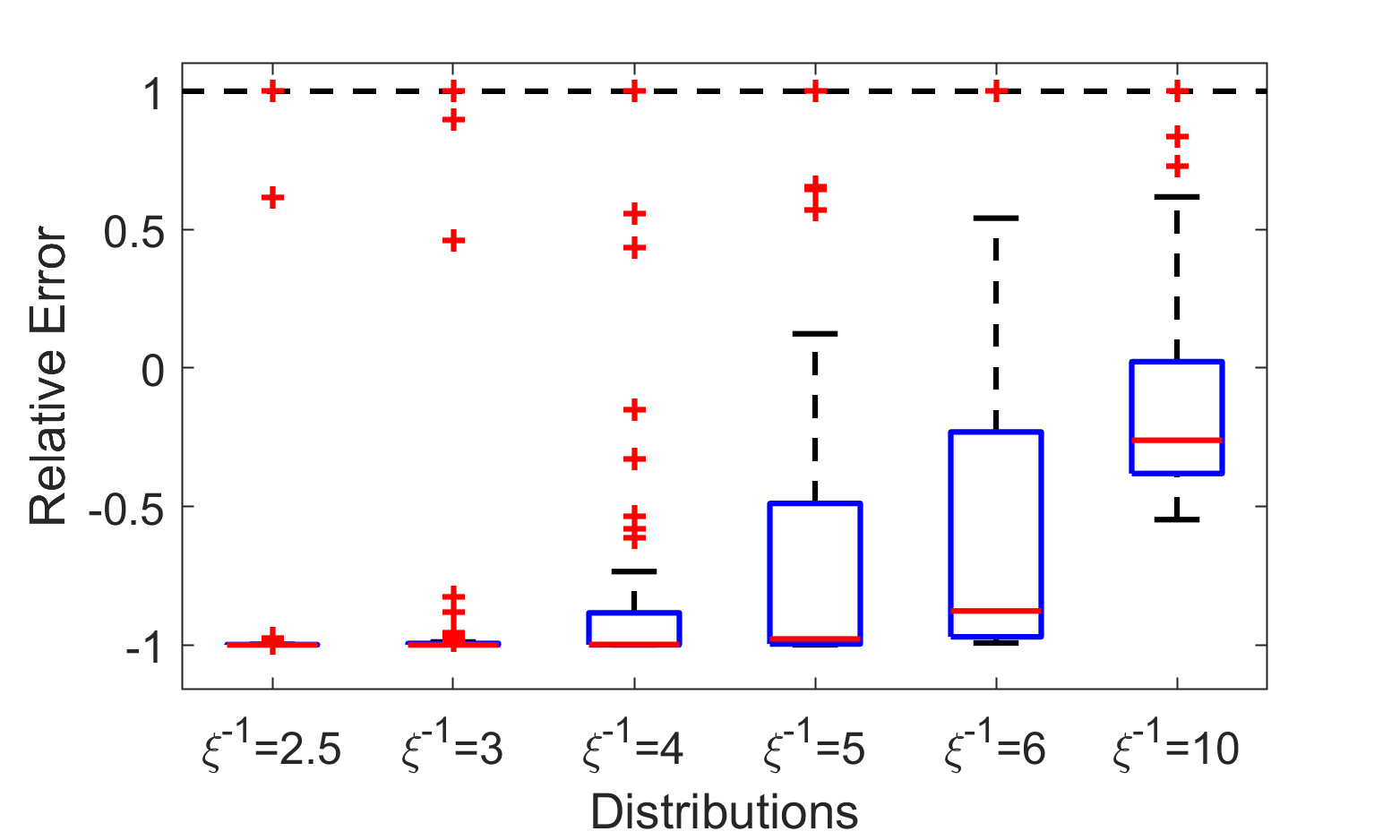}} \quad
     \subfloat[Student's $t$, $n=50$, $p\approx10^{-5}$, $N=10^6$\label{fig:hh_n50N6p5_t}
     ]{\includegraphics[width=0.45\textwidth]{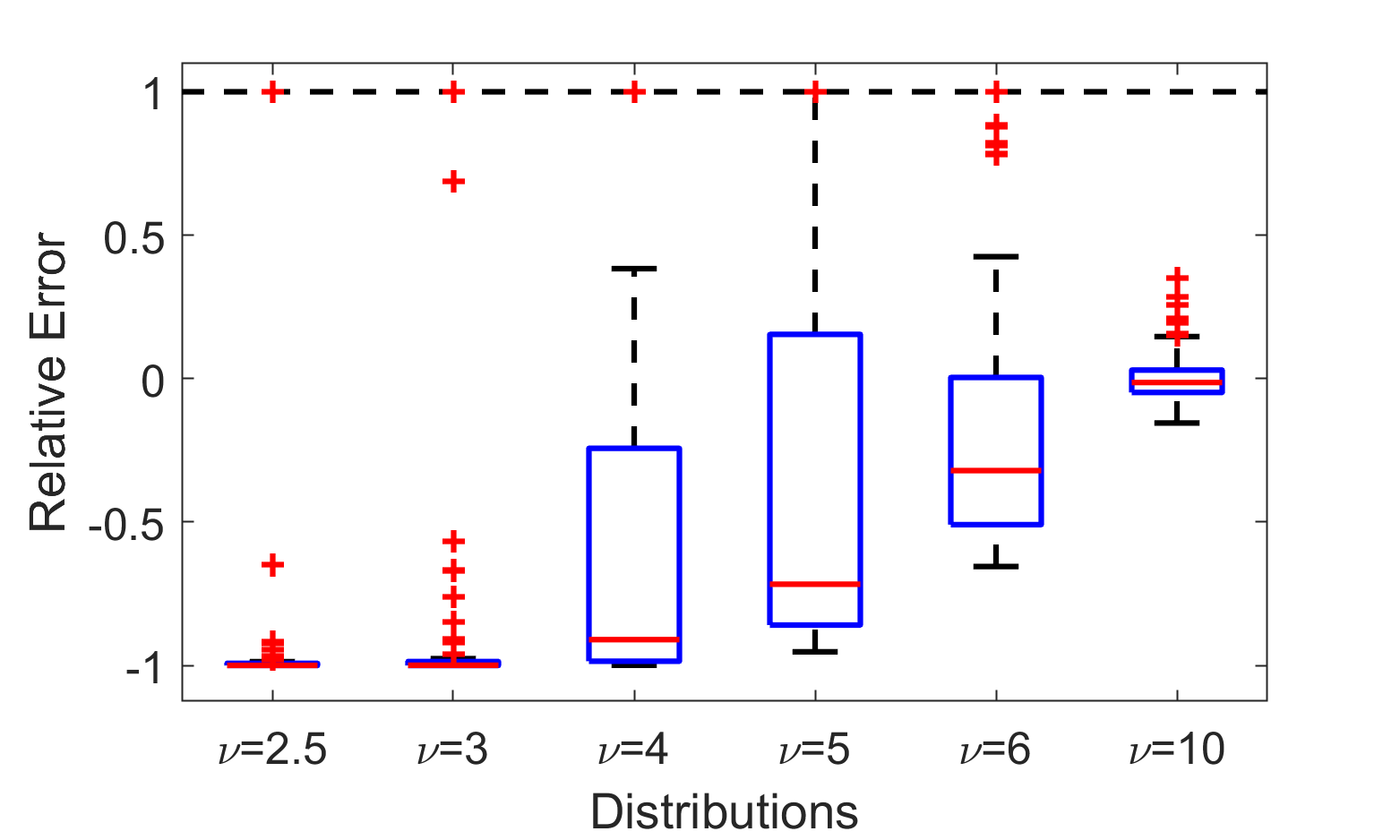}} \quad 
     \subfloat[Student's $t$, $n=50$, $p\approx10^{-6}$, $N=10^7$\label{fig:hh_n50N7p6_t}]{\includegraphics[width=0.45\textwidth]{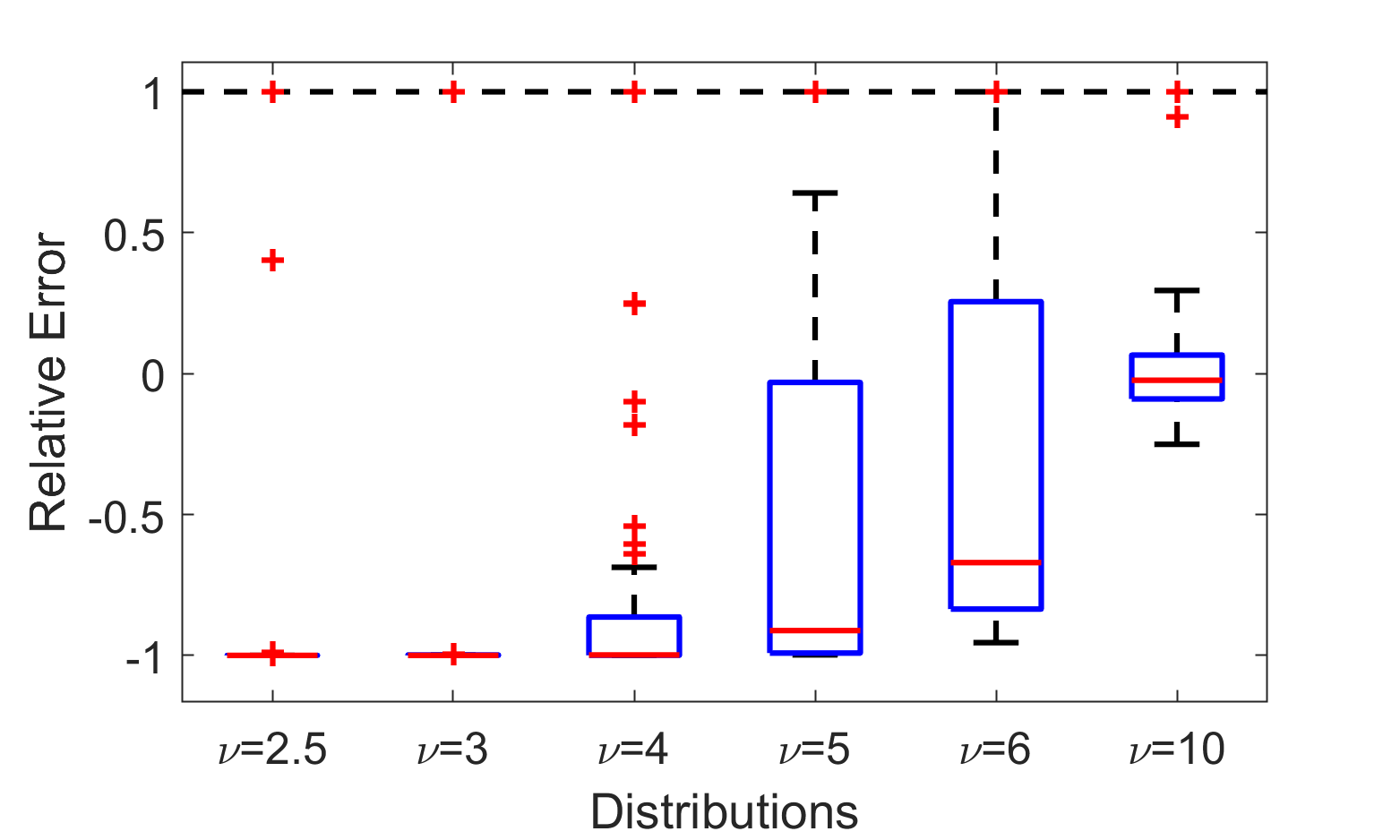}}
     \caption{Relative errors (the ratios of error to the ground truth) in simulation results with empirical distribution. The tail distributions include generalized Pareto distributions and $t$-distributions with various tail indices (i.e. $\xi^{-1}$ for Pareto and $\nu$ for Student's $t$).}
     \label{fig:heavy_heavier}
 \end{figure}
 
 In addition to the comparison between heavy-tailed and light-tailed distributions, we find that heavy-tailed distributions with different tail heaviness exhibit different estimation errors. For instance, in Figures \ref{fig:lvh_n100p7N4} and \ref{fig:lvh_n100p7N8}, we observe that when the tail is lighter with $\xi^{-1}=10$, the estimates have better accuracy than those with a heavier tail (e.g. $\xi^{-1}=2$ or $5$). In this subsection, we drill on this phenomenon further and investigate the impacts of tail heaviness within parametric families. 
 
 We consider heavy-tailed distributions of Pareto and Student's $t$, with different tail indices. We set $n=50$ and consider two settings of $p\approx10^{-5}$ and $p\approx10^{-6}$. Figure \ref{fig:heavy_heavier} shows that a heavier tail would increase the data requirement for obtaining accurate estimates. In particular, in the figure, the distributions are ordered from left to right with decreasing tail heaviness. With $n=50$, $p\approx10^{-5}$ and $N=10^6$ in Figures \ref{fig:hh_n50N6p5_p} and \ref{fig:hh_n50N6p5_t}, the Pareto and $t$-distribution cases show very similar patterns: when the tail is relatively heavy (tail index equals 2.5 or 3), more than 75\% of the estimates have approximately $-1$ relative errors, which represent significant under-estimation. When the tail index increases from 4 (to 5) to 6, we observe that the median of the estimates gradually grows from near $-1$ to $-0.7$ in the Pareto case (Figure \ref{fig:hh_n50N6p5_p}) and $-0.3$ in the Student's $t$ case (Figure \ref{fig:hh_n50N6p5_t}). Finally, when tail is relatively light (tail index equals 10), more than half of the estimates are closely concentrated around 0, and therefore are much more reliable.
 
 Similar observations are found in Figures \ref{fig:hh_n50N7p6_p} and \ref{fig:hh_n50N7p6_t} as well, where we have $n=50$, $p\approx10^{-6}$ and $N=10^7$. Under this setting, the estimates fail badly when the tail indices are smaller than (or equal to) 4 in both the Pareto and $t$-distribution cases, shown by the near $-1$ median of the relative errors. As we increase the tail index, which leads to a lighter tail, the estimates improve and finally the box edges concentrate with $\pm 0.5$ relative errors in the Pareto case (Figure \ref{fig:hh_n50N7p6_p}) and $\pm 0.1$ in the Student's $t$ case (Figure \ref{fig:hh_n50N7p6_t}). These comparisons within the same parametric families further support the findings in Section \ref{sec:exp_light_vs_heavy}, demonstrating that problems with heavier tails lead to greater under-estimation. 

 
\subsubsection{Between Light and Power-Law Tails.}\label{sec:exp_slow_varying}

Next we investigate the performance of non-power-law tails. Note that in Figures \ref{fig:hh_n50N7p6_p} and \ref{fig:hh_n50N7p6_t}, with tail index equal to 6, the Pareto box fails to cover the truth, whereas the $t$-distribution box has a larger range and contains 0. Since they have the same tail index, it appears that the slowly varying function inside the distribution affects the coverage of the estimates. This subsection studies this issue even further, by considering additionally subexponential distributions without power-law tails, namely log-normal and heavy-tailed Weibull distributions.




\begin{figure}[t]
      \centering
     \subfloat[$n=50$, $p\approx10^{-5}$, $N=10^5$\label{fig:slow_n50N5p5}
     ]{\includegraphics[width=0.45\textwidth]{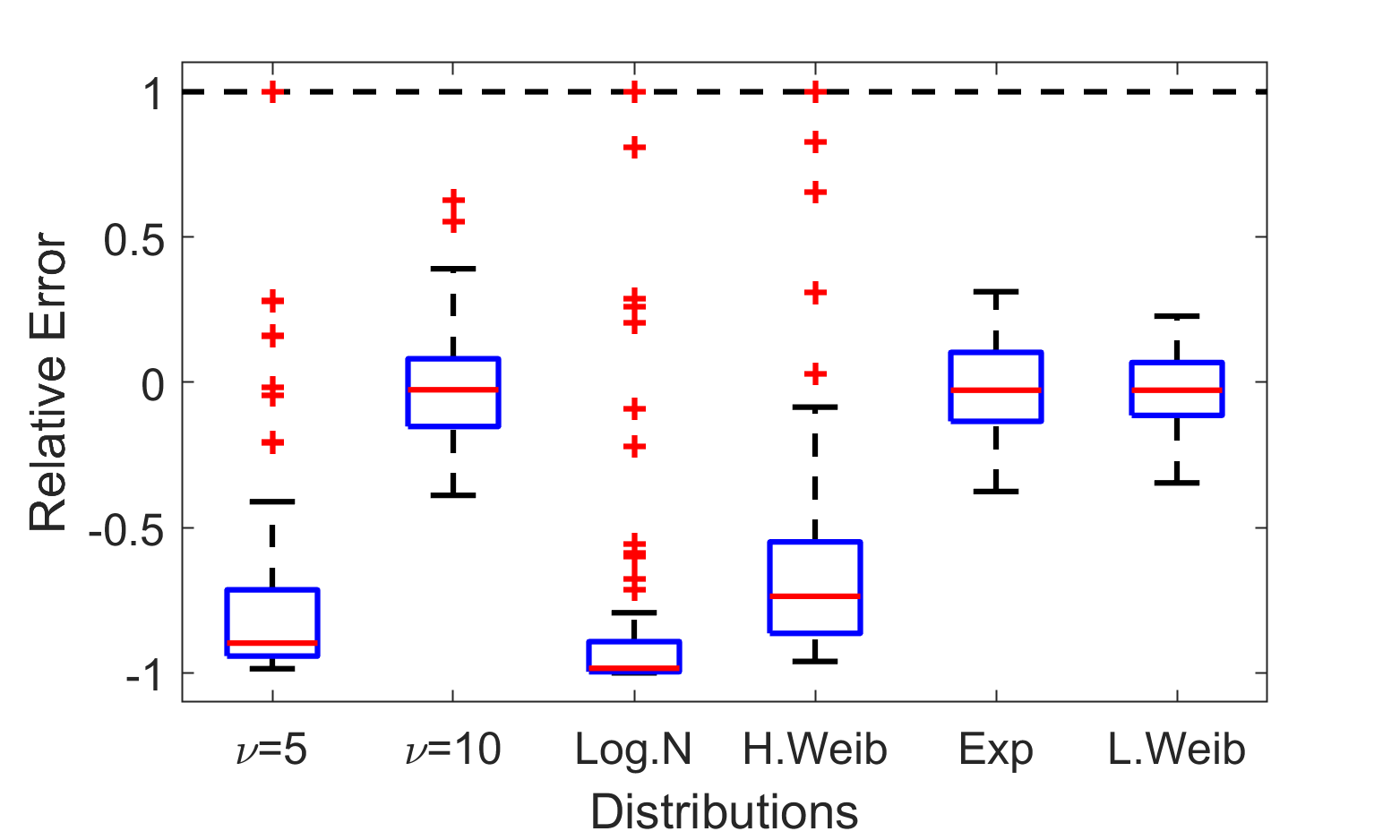}} \quad 
     \subfloat[$n=50$, $p\approx10^{-7}$, $N=10^7$\label{fig:slow_n50N7p7}]{\includegraphics[width=0.45\textwidth]{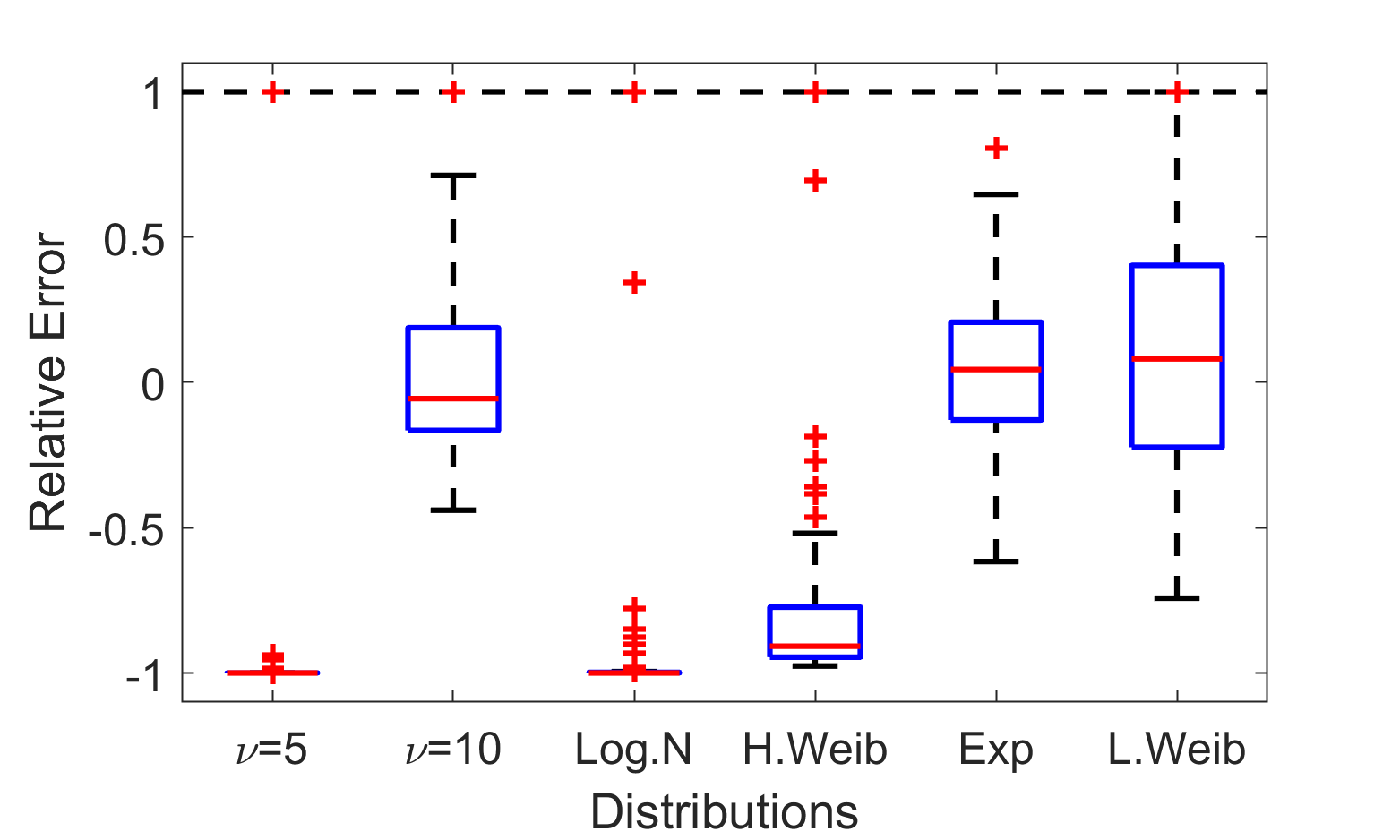}} \quad
     \subfloat[$n=100$, $p\approx10^{-5}$, $N=10^5$\label{fig:slow_n100N5p5}
     ]{\includegraphics[width=0.45\textwidth]{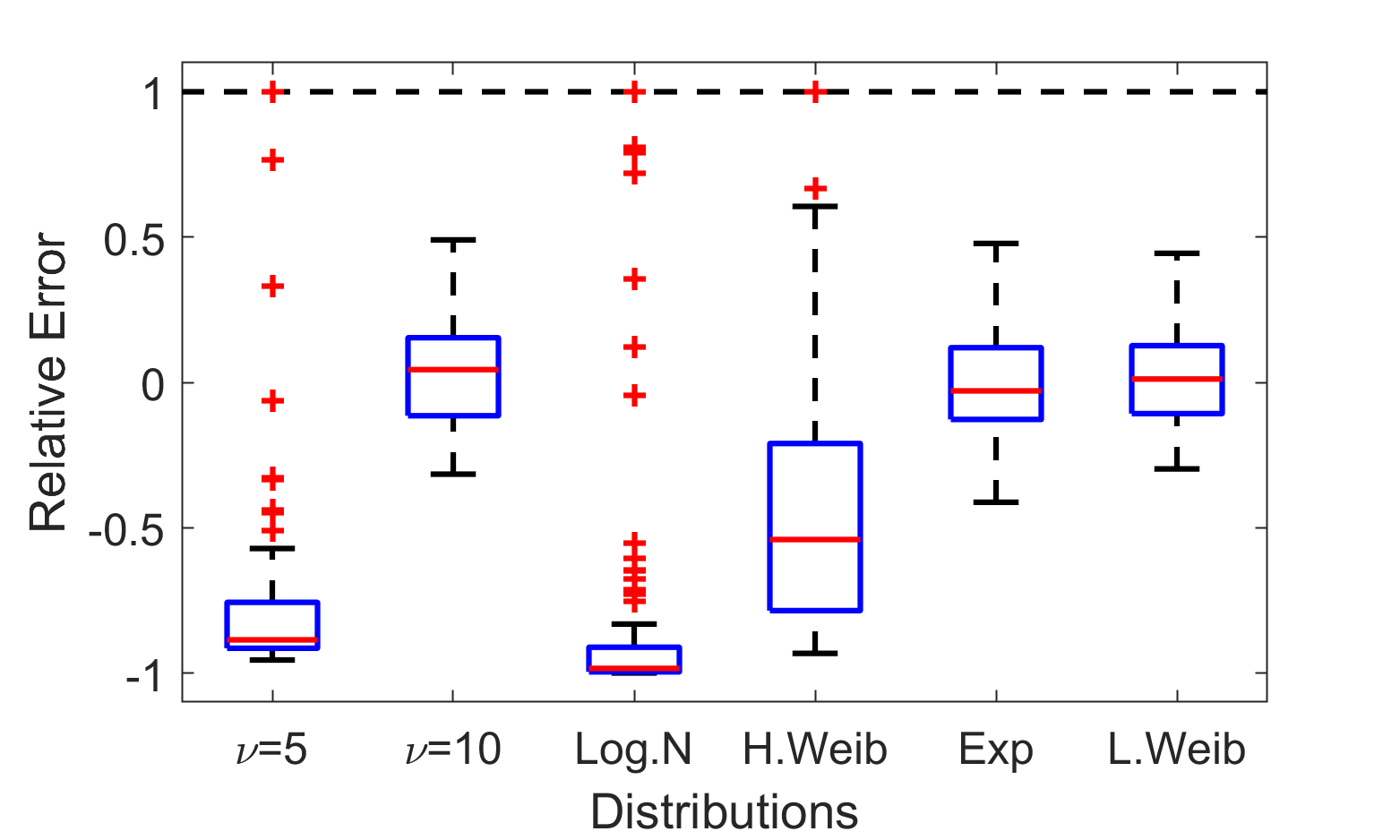}} \quad 
     \subfloat[$n=100$, $p\approx10^{-7}$, $N=10^7$\label{fig:slow_n100N7p7}]{\includegraphics[width=0.45\textwidth]{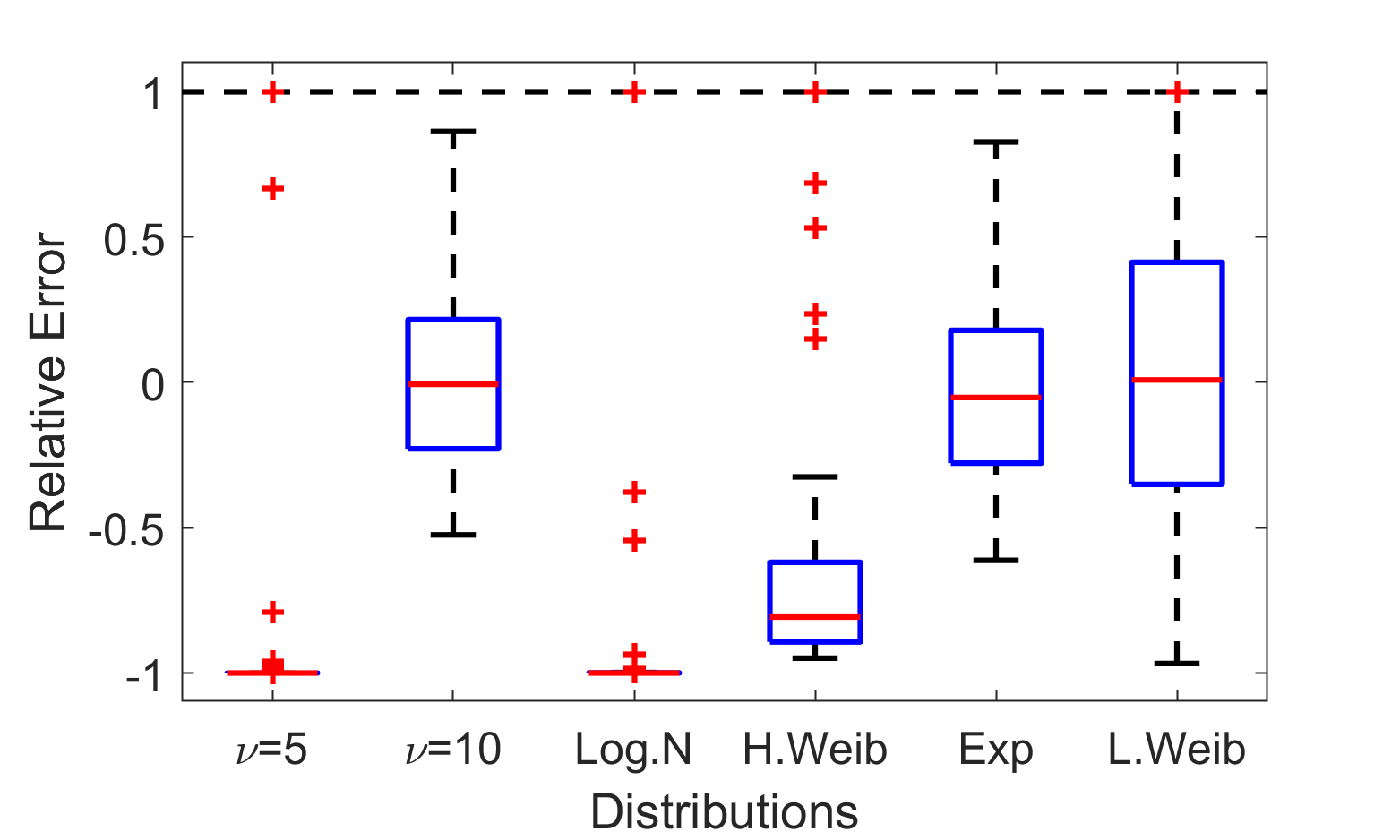}}
     \caption{Relative errors (the ratios of error to the ground truth) in simulation results with empirical distribution. The tail distributions include $t$-distributions with various choices of $\nu$, log-normal distribution (Log.N), heavy-tailed Weibull (H.Weib) distribution and light-tailed distributions, including exponential (Exp) and light-tailed Weibull (L.Weib).}
     \label{fig:slow_varying}
 \end{figure}

In Figure \ref{fig:slow_varying}, we order the distributions from left to right with decreasing tail heaviness (in the asymptotic sense). In Figure \ref{fig:slow_n50N5p5}, we use $n=50$, $p\approx 10^{-5}$ and $N=10^5$. The light-tailed cases (Exp and L.Weib) exhibit accurate estimates that are distributed within approximately $\pm 0.4$ relative errors. For heavy-tailed cases, with $\nu=5$ the box fail to cover the truth and more than 75\% estimates have smaller than $-0.7$ relative errors, while with $\nu=10$ the estimates are well concentrated within $\pm 0.5$ relative errors. We note that for the subexponential distributions, whose tails are asymptotically lighter than $\nu=10$, the estimates are in fact much worse. In particular, we observe that the estimates for the log-normal case are even worse than the $\nu=5$ estimates, as the top edge of box is much closer to $-1$. The estimates for heavy-tailed Weibull are better than the log-normal ones, but their boxes still fail to cover the truth. We present similar results in Figures \ref{fig:slow_n50N7p7}, \ref{fig:slow_n100N5p5}, and \ref{fig:slow_n100N7p7}. In these figures, we observe that the estimates for log-normal and heavy-tailed Weibull distributions have consistently worse estimates than the $t$-distributions with $\nu=10$, even though their tails do not follow power-law and hence are lighter.

These experiments indicate that tails falling between Pareto and light tails can lead to significant under-estimation issues like Pareto tails. Moreover, the comparative estimation performances relative to data size for these tails appear subtle and intuitive notion of tail heaviness alone might not adequately inform estimation reliability.


\subsection{Quantifying Tail Uncertainty Using Bootstraps.} \label{sec:bootstrap_exp}

Next we investigate the use of bootstrapping to assess input uncertainty, in particular whether bootstrap confidence intervals (CIs) can provide valid coverage. In the following, we consider two bootstrap schemes: a) nonparametric bootstrap, and b) bootstrap assisted with generalized Pareto tail fitting.


\subsubsection{Nonparametric Bootstrap.} \label{sec:bootstrap_empirical}


 We create bootstrapped empirical distributions through repeated resampling with replacement, with resample size equal to the original data size. The resampling is repeated $B=100$ times, where each resample is used as the input model to drive the rare-event estimation. The bootstrap CI is then constructed using the empirical quantiles of the resample estimates. 
 
 We present our experimental results on two problems, one representing a heavy-tailed scenario and another a light-tailed scenario. In particular, we consider $X_i$ with a Gaussian distribution for the light-tailed case and a $t$-distribution with $\nu=4$ for the heavy-tailed case. In both scenarios, we set the rare-event probabilities around $10^{-5}$. We focus on evaluating the efficacy of the bootstrap CIs, specifically whether they achieve the desired coverage level, which is set at 95\% for our experiments. To obtain the coverage levels, we make 100 experimental replications, where in each replication we construct a bootstrap CI and assess whether it covers the true probability. 



\begin{figure}[t]
     \centering
     \subfloat[\label{fig:fig_boot_coverage}     ]{\includegraphics[width=0.45\textwidth]{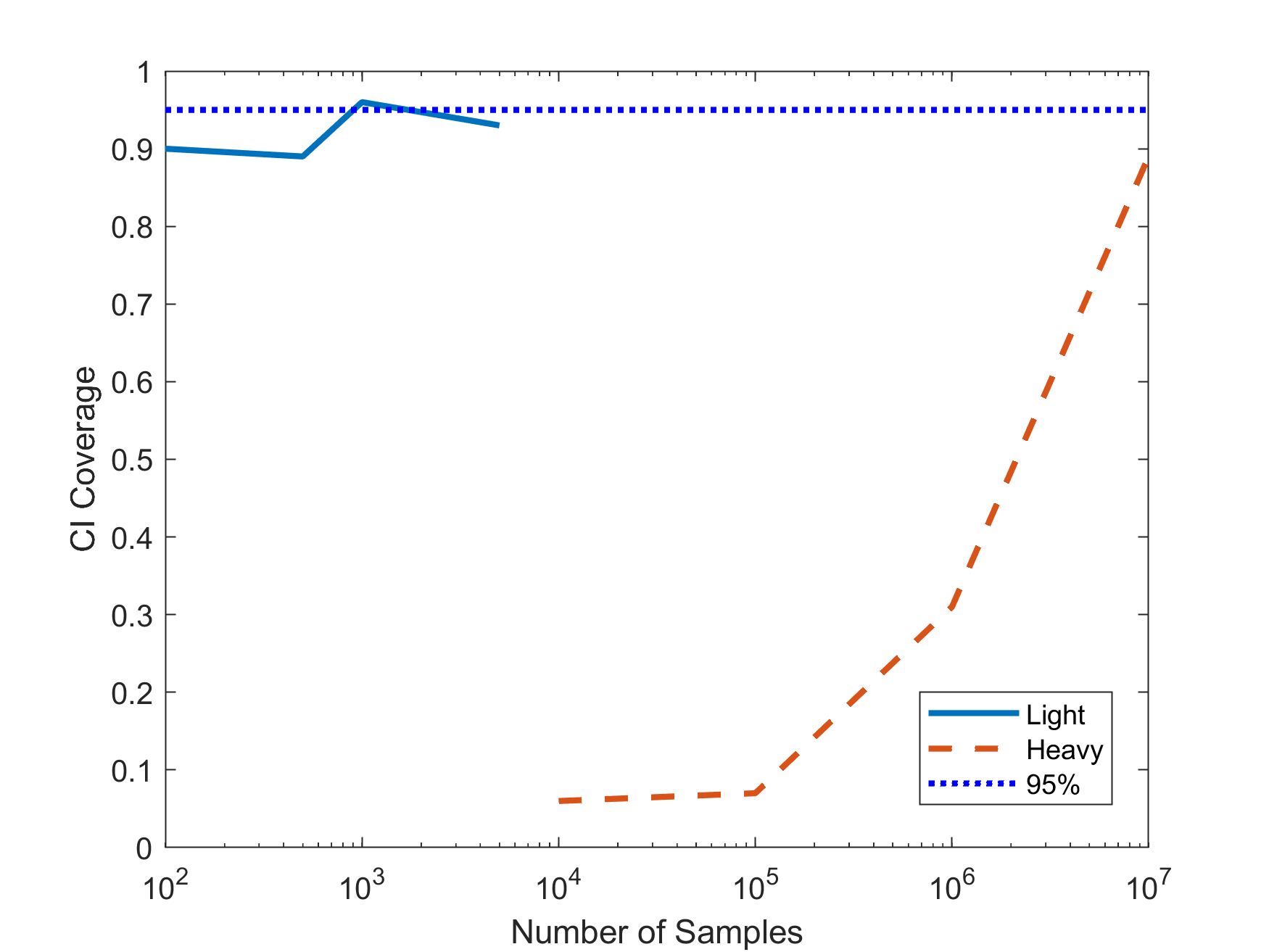}} \quad 
     \subfloat[\label{fig:fig_boot_width}]{\includegraphics[width=0.45\textwidth]{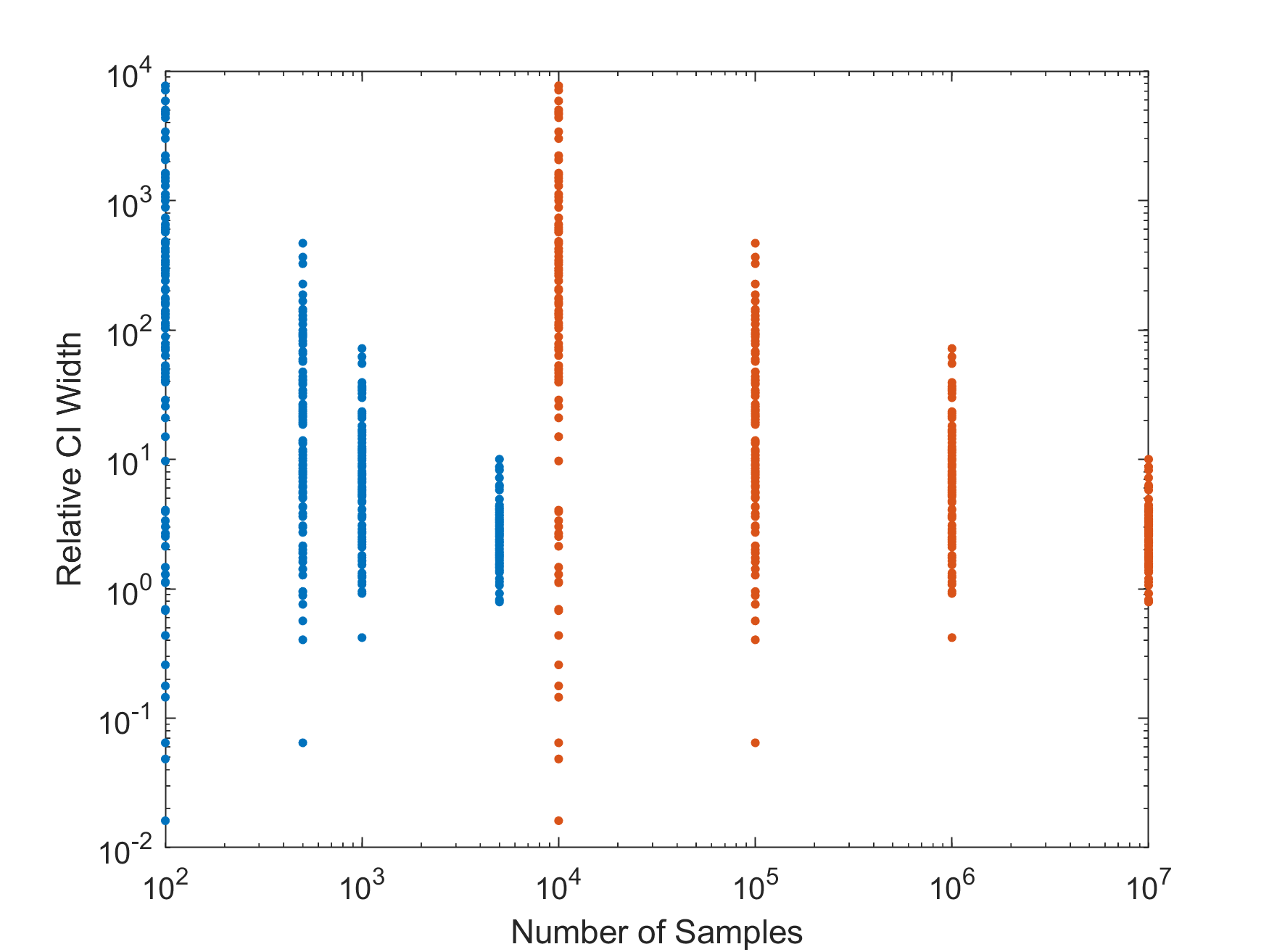}}
     \caption{Results of plain bootstrap with varying sample sizes based on 100 replications for each sample size under two settings: (1) Light-tailed case (Gaussian tail) with sample sizes ranging from $10^2$ to $10^4$ and (2) Heavy-tailed case (Student's $t$ tails, $\nu=4$) with sample sizes ranging from $10^4$ to $10^7$. (a) Coverages of the 95\% bootstrap CIs. (b) Distributions of CI width/target probability.}
     \label{fig:boot_figs_a}
 \end{figure}
 
Figure \ref{fig:boot_figs_a} presents our results. Figure \ref{fig:fig_boot_coverage} shows that, for light-tailed problems, the coverage of the CI is above 90\% in each considered case (the blue solid line). The results indicate that the bootstrap CIs are valid, as they provide coverages close to the target level even when the number of samples is relatively small compared to the probability in question (for example, 100 samples versus a probability of $10^{-5}$). Moreover, we observe in Figure \ref{fig:fig_boot_width} that the CI width is relatively big compared to the estimated probability (for example, $10^4$ times larger than the target probability with 100 samples). While these wider intervals might seem overly cautious, they are important for effectively identifying instances where the probability estimate is unreliable. On the other hand, for the heavy-tailed case, Figure \ref{fig:fig_boot_coverage} shows that the bootstrap CI badly under-covers the truth when the number of samples is smaller than $10^7$. Compared to the light-tailed case, a sample size of $10^6$ is not small, given that the target probability is only around $10^{-5}$. 


The experimental results suggest that the bootstrap works well in light-tailed problems, but fails in heavy-tailed problems. This ties to our explanation in Section \ref{sec:setting} that the impact from tail uncertainty is more profound in the heavy-tailed case. The experiment suggests that in heavy-tailed problems, the lack of tail information not only causes problems in estimating the probability itself, but can also deem the assessment of input uncertainty very challenging.



\subsubsection{Bootstrap Using Generalized Pareto Distribution.} \label{sec:bootstrap_gpd}

We attempt to overcome the challenges in Section~\ref{sec:bootstrap_empirical} by using generalized Pareto distributions to inject additional information about the tail into the estimation. Specifically, the bootstrap is conducted similarly to Section \ref{sec:bootstrap_empirical}, but with a key variation that for each resample, we fit the tail using a generalized Pareto distribution. The latter is conducted via various estimation techniques including the maximum likelihood estimation (MLE), method-of-moments (MOM) and probability-weighted moments (PWM), as detailed in \cite{castillo1997fitting,hosking1987parameter}. For identifying the tail data required for fitting, we experiment with different truncation points, specifically using the 0.05, 0.01, and 0.005 empirical tail quantiles of data.

In our experiment, we focus on the heavy-tailed case in Section~\ref{sec:bootstrap_empirical}, where $X_i$ follows a $t$-distribution with degree of freedom $\nu=4$. We consider sample sizes ranging from $10^4$ to $10^6$ which, as demonstrated in Figure \ref{fig:boot_figs_a}, are insufficient for the standard bootstrap method to function effectively. Similar to the experiments in Section \ref{sec:bootstrap_empirical}, each CI is calculated from 100 number of resamples. We run 30 experimental replications to obtain statistics on coverages and interval widths. 




  \begin{figure}[!t]
     \centering
     \subfloat[ \label{fig4}
     ]{\includegraphics[width=0.45\textwidth]{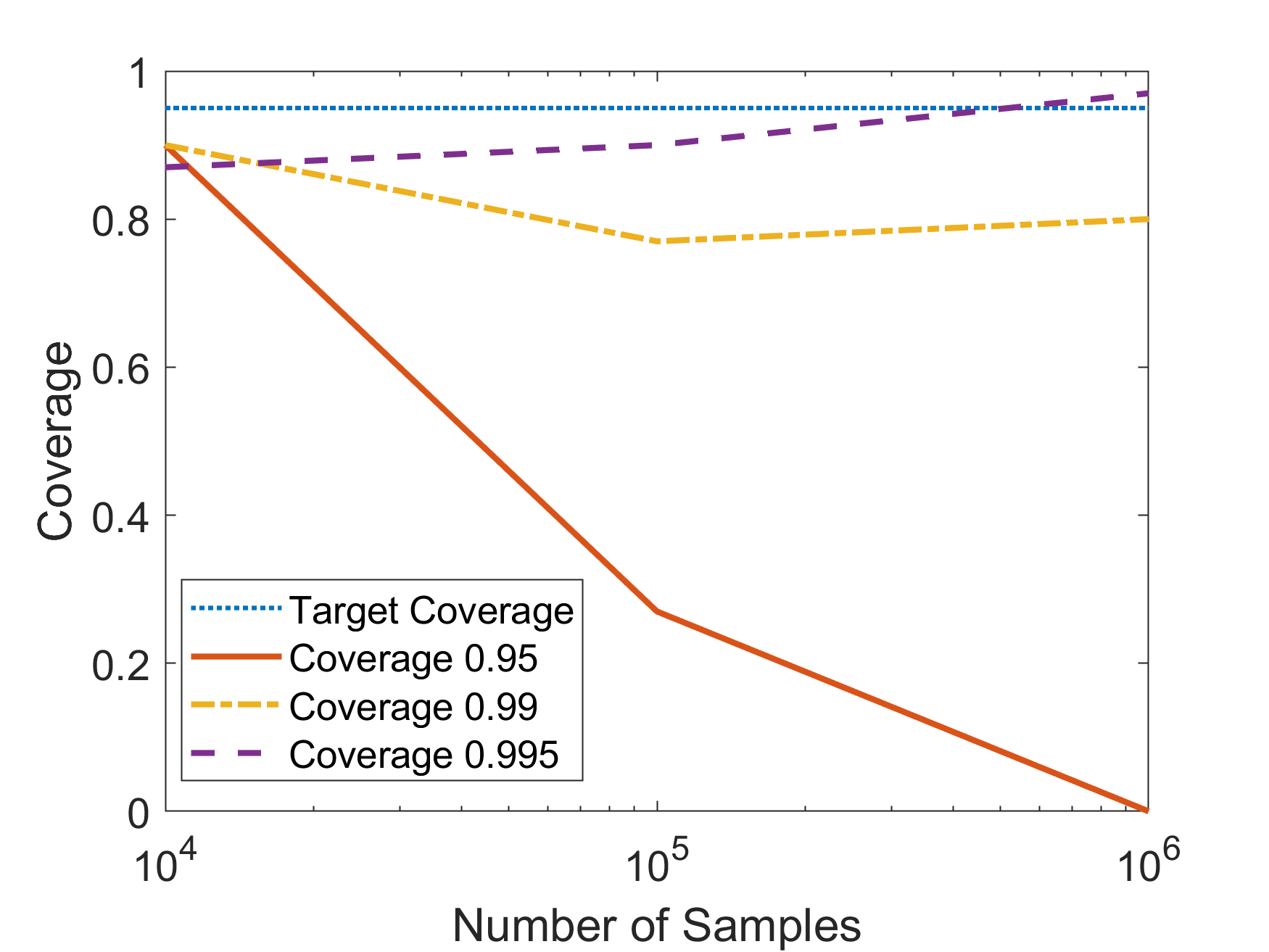}} \quad 
     \subfloat[ \label{fig5}]{\includegraphics[width=0.45\textwidth]{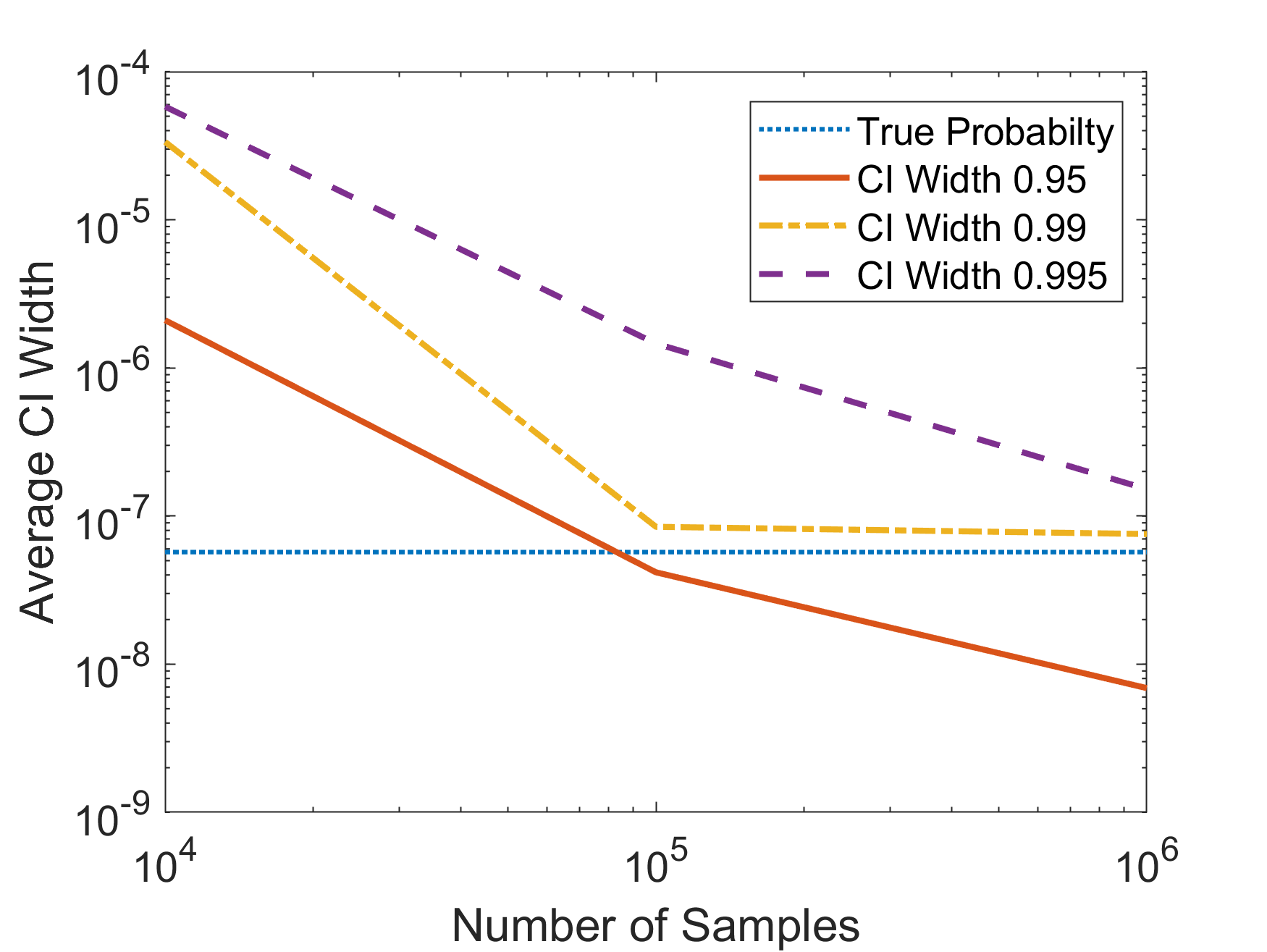}}
     \caption{CI coverages and widths of the generalized Pareto tail bootstrap scheme (fitted using MLE) on the problem with $p=5.7095\times 10^{-8}$, with truncation points in fitting the generalized Pareto set as the 0.05, 0.01, and 0.005 empirical tail quantiles of data. (a) Coverage. (b) Average CI width.}\label{fig:figure_3_set}
 \end{figure}


The experiment results are presented in Figure~\ref{fig:figure_3_set} and Table \ref{table:GPD_boot}. 
Figure~\ref{fig4} shows that in most scenarios, there is an improvement in coverage compared to the standard bootstrap, as indicated by the provided coverages being close to the target level. On the other hand, the coverage provided by the method is notably affected by the chosen truncation point. As depicted in Figure~\ref{fig4}, coverage can decrease significantly, even dropping to zero, with an increase in the sample size. We note that the true distribution in our case is a $t$-distribution, which implies a mismatch in the tail fit due to model misspecification. As the interval width shrinks with an increasing number of samples, the model biasedness starts to surface.


\begin{table}[!t]
\caption{CI coverage and width from the bootstrap using generalized Pareto distribution. ``\# Spl'' represents the sample size and ``Tail Qtl'' represents the tail quantile of the truncation point. The target rare-event probability is $p=5.7095\times 10^{-8}$.}
\label{table:GPD_boot}
\centering
\resizebox{0.9\textwidth}{!}{
\begin{tabular}{llllllll}
\hline
\multicolumn{1}{|l|}{}                                                       &          \multicolumn{1}{l|}{\# Spl }                                                           & \multicolumn{2}{l|}{$10^4$ }                                   & \multicolumn{2}{l|}{$10^5$}                                   & \multicolumn{2}{l|}{$10^6$}                                   \\ \hline
 \multicolumn{1}{|l|}{Tail Qtl}    & \multicolumn{1}{l|}{Method} & \multicolumn{1}{l|}{Coverage} & \multicolumn{1}{l|}{CI Width} & \multicolumn{1}{l|}{Coverage} & \multicolumn{1}{l|}{CI Width} & \multicolumn{1}{l|}{Coverage} & \multicolumn{1}{l|}{CI Width} \\ \hline
 \multicolumn{1}{|l|}{\multirow{3}{*}{0.05}}  & \multicolumn{1}{l|}{MLE}    &   0.9                          & \multicolumn{1}{l|}{$2.10 \times 10^{-6}$}         &         0.27       & \multicolumn{1}{l|}{$4.17 \times 10^{-8}$}         &       0          & \multicolumn{1}{l|}{$6.89 \times 10^{-9}$}         \\ 
\multicolumn{1}{|l|}{}                   & \multicolumn{1}{l|}{MOM}    &                  0.67      & \multicolumn{1}{l|}{$1.23 \times 10^{-6}$}         &   0.53                      & \multicolumn{1}{l|}{$4.89 \times 10^{-7}$}         &     0.30                   & \multicolumn{1}{l|}{$3.19 \times 10^{-7}$}         \\ 
\multicolumn{1}{|l|}{}                      & \multicolumn{1}{l|}{PWM}    &                0.87    & \multicolumn{1}{l|}{$1.81 \times 10^{-6}$}         &    0.30       & \multicolumn{1}{l|}{$4.57\times 10^{-8}$}         &     0          & \multicolumn{1}{l|}{$7.04 \times 10^{-9}$}         \\ \hline 
 \multicolumn{1}{|l|}{\multirow{3}{*}{0.01}}  & \multicolumn{1}{l|}{MLE}    &               0.90   & \multicolumn{1}{l|}{$3.36 \times 10^{-5}$}         &                   0.77   & \multicolumn{1}{l|}{$8.46 \times 10^{-7}$}         &     0.80     & \multicolumn{1}{l|}{$7.56 \times 10^{-8}$}         \\
\multicolumn{1}{|l|}{}                          & \multicolumn{1}{l|}{MOM}    &                  0.60     & \multicolumn{1}{l|}{$1.26 \times 10^{-6}$}         &                       0.70 & \multicolumn{1}{l|}{$8.10 \times 10^{-7}$}         &                     0.77   & \multicolumn{1}{l|}{$5.61 \times 10^{-7}$}         \\ 
\multicolumn{1}{|l|}{}                             & \multicolumn{1}{l|}{PWM}    &                0.90     & \multicolumn{1}{l|}{$1.09 \times 10^{-5}$}         &                    0.77   & \multicolumn{1}{l|}{$8.13 \times 10^{-7}$}         &                   0.77     & \multicolumn{1}{l|}{$8.84 \times 10^{-8}$}         \\ \hline 
 \multicolumn{1}{|l|}{\multirow{3}{*}{0.005}}  & \multicolumn{1}{l|}{MLE}    &              0.87     & \multicolumn{1}{l|}{$5.81 \times 10^{-5}$}         &                   0.90     & \multicolumn{1}{l|}{$1.46 \times 10^{-6}$}         &                   0.97     & \multicolumn{1}{l|}{$1.42 \times 10^{-7}$}         \\ 
\multicolumn{1}{|l|}{}                             & \multicolumn{1}{l|}{MOM}    &                 0.63    & \multicolumn{1}{l|}{$1.17 \times 10^{-6}$}         &    0.67       & \multicolumn{1}{l|}{$8.43 \times 10^{-7}$}         &      0.87      & \multicolumn{1}{l|}{$6.10 \times 10^{-7}$}         \\ 
\multicolumn{1}{|l|}{}                           & \multicolumn{1}{l|}{PWM}    &                0.87        & \multicolumn{1}{l|}{$1.25 \times 10^{-5}$}         &    0.83   & \multicolumn{1}{l|}{$1.58 \times 10^{-6}$}         &      0.97       & \multicolumn{1}{l|}{$1.70 \times 10^{-7}$}         \\ \hline

\end{tabular}
}
\end{table}

Among the three approaches for fitting generalized Pareto distribution, MLE and PWM turn out to be more reliable than MOM. This is revealed in Table \ref{table:GPD_boot}, where the coverage of MOM is less than the other two approaches in most cases. The performance matches the documented fact that MOM is unreliable when the shape parameter $\xi > 0.2$ \cite{hosking1987parameter}. When the sample size is smaller ($10^4$ observations), PWM gives a smaller average interval width than MLE, while providing similar coverages (e.g., with 0.01 tail quantile the widths are $3.36 \times 10^{-5}$ for MLE and $1.09 \times 10^{-5}$ for PWM). It therefore suggests that PWM is more suitable for smaller samples. When the sample size is large ($10^6$ observations), MLE has an upper hand in terms of interval width (e.g., with 0.005 tail quantile the widths are $1.42 \times 10^{-7}$ for MLE and $1.70 \times 10^{-7}$ for PWM).

The experimental results show that although the overall performance of this enhanced bootstrap scheme is better than the standard bootstrap, the obtained CIs can still be misleading. The latter is caused by the model biasedness from the generalized Pareto distribution that the bootstrap cannot overcome. 

\subsection{Detection of Tail Heaviness.} \label{sec:tail_detection}
We have seen evidence from theories as well as experiments that heavy tail may cause severe under-estimation when using data to inform the input model in rare-event estimation. Moreover, reliably detecting such under-estimation, or in other words providing CIs that correctly capture the amount of uncertainty, could be challenging as well. To address these, our approach is to devise a procedure to detect the risk of uncertainty under-estimation, or in other words identify cases where uncertainty under-estimation, as well as probability under-estimation, is prone to occur with inadequate data. If the procedure detects this risk, then we can mitigate it by collecting data big enough so that its size is confidently larger than $n/p$, a magnitude explained in Section~\ref{sec:samplesize_heavy}. Conversely, if the procedure detects otherwise, then both the analyses in Section~\ref{sec:samplesize_light} and previous experiments have demonstrated that we can obtain reliable estimates with a moderate sample size, such as $10^3$ samples for the experiments discussed in Section~\ref{sec:tail_heaviness}. Note that in the first case, we would need to get a pilot estimate of $p$, which could be obtained from the collected data along with suitable conservative inflation, or through some prior knowledge on its magnitude. However, it could well be the case that we cannot even get the rough magnitude of $p$, or we do not have access to more data, either of which case means any decision-making should account for the under-estimation risk.

We recommend to use properly chosen tail index estimates to conduct such a detection.
In the extreme value theory literature, various graphical and estimation methods can be applied to obtain tail information from data, including quantile-quantile (QQ) plotting \cite{kratz1996qq}, mean excess plotting (e.g. Chapter 7.2.2 in \cite{embrechts2005quantitative}) and tail index estimation (for detailed discussion, see Chapter 4 in \cite{resnick2007heavy} and Chapter 1.3 in \cite{markovich2008nonparametric}) In most of these approaches, the assessment of tail heaviness is achieved by estimating the tail index, $\alpha$ in \eqref{rv}, in the form of slopes in plots or explicit estimators. 
Among various choices, we focus on estimators, denoted as $\hat{\xi}$, that are designed for the extreme value index (also known as the shape parameter) of distributions in the domain of attractions of the generalized Pareto distribution. Compared with classical tail index estimators such as the classical Hill estimator \cite{mason1982laws,de1998asymptotic}, extreme value index estimators have an advantage in dissecting light from heavy tails: A negative estimate ($\hat{\xi}<0$) suggests the samples are light-tailed. On the other hand, for distributions with power-law tails as \eqref{rv}, $1/\hat{\xi}$ is consistent in estimating $\alpha$ \cite{dekkers1989estimation}. In particular, we consider the Pickands estimator \cite{dekkers1989estimation,drees1995refined} and the moment estimator \cite{dekkers1989estimation,dekkers1989moment,dekkers1993optimal,haan2006extreme}.

It is important to note that all tail index estimators, including the extreme value index estimators, are specifically designed for Pareto tails. This means that the presence of a slowly varying function in the distribution could lead to model misspecification. To address this issue, there is a common practice of excluding data outside the tail portion to minimize the impact of the slowly varying function. The truncation point in the data is selected prior to implementing the estimation process. Since the estimators might be sensitive to the selection of the truncation point, a convention is to plot the $(k,\hat{\xi}_k)$ pairs \cite{resnick2007heavy}, where $k$ denotes the number of order statistics (in a decreasing order) of the tail data for truncation, and $\hat{\xi}_k$ represents the corresponding estimate.


\begin{figure}[t]
      \centering
     \subfloat[Pickands, $N=10^4$\label{fig:pick_p_4}]{\includegraphics[width=0.45\textwidth]{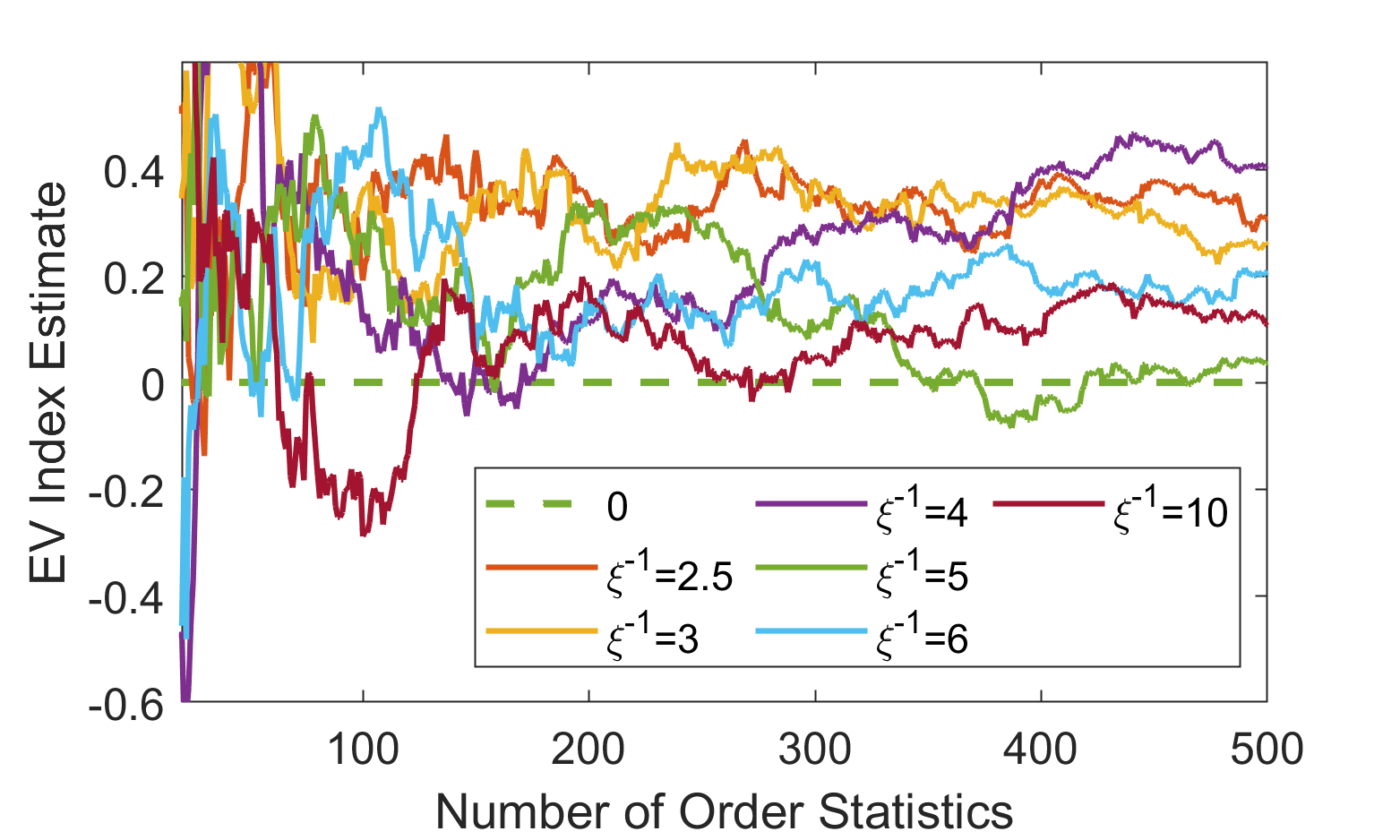}} \quad 
     \subfloat[Moment, $N=10^4$\label{fig:mome_p_4}]{\includegraphics[width=0.45\textwidth]{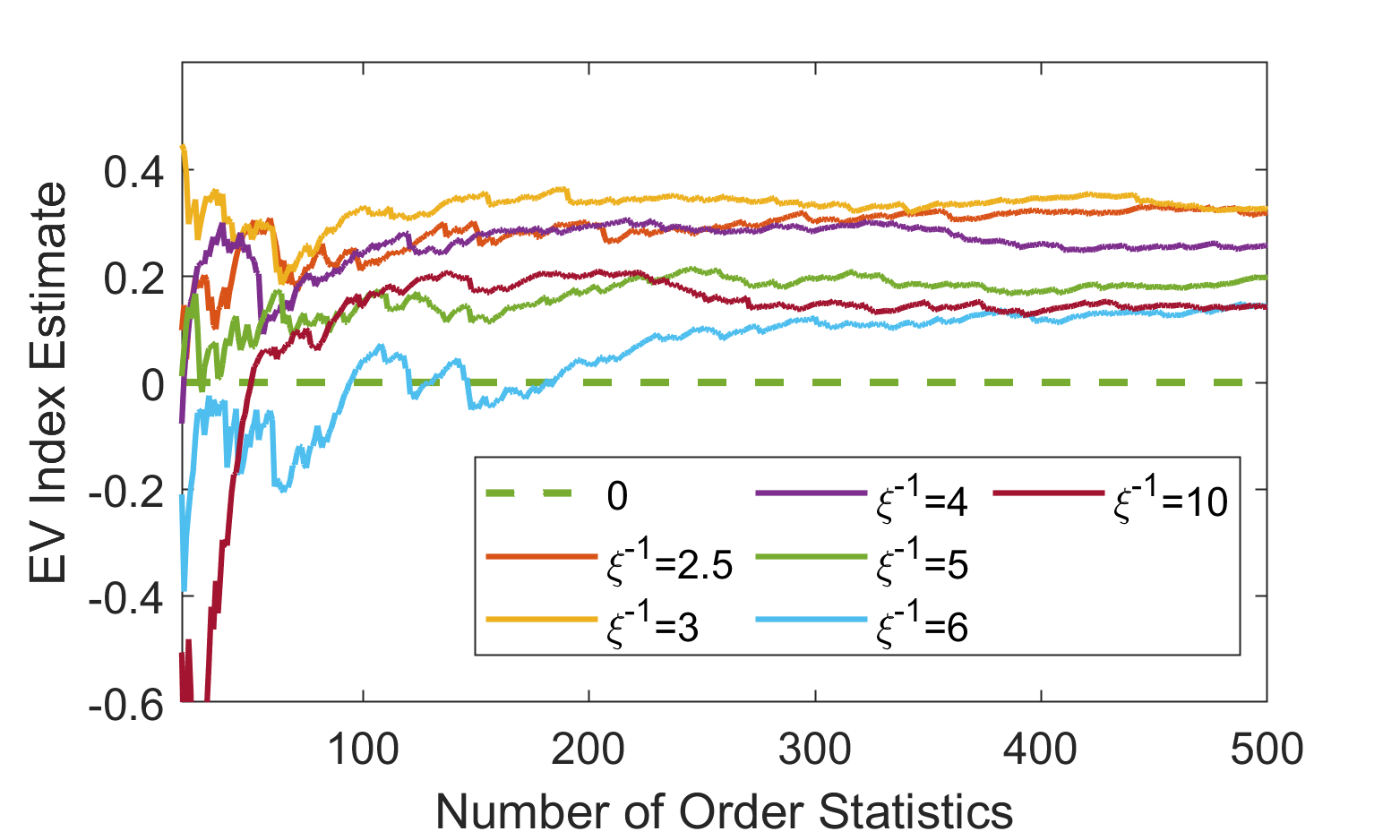}} \quad
     \subfloat[Pickands, $N=10^5$\label{fig:pick_p_5}]{\includegraphics[width=0.45\textwidth]{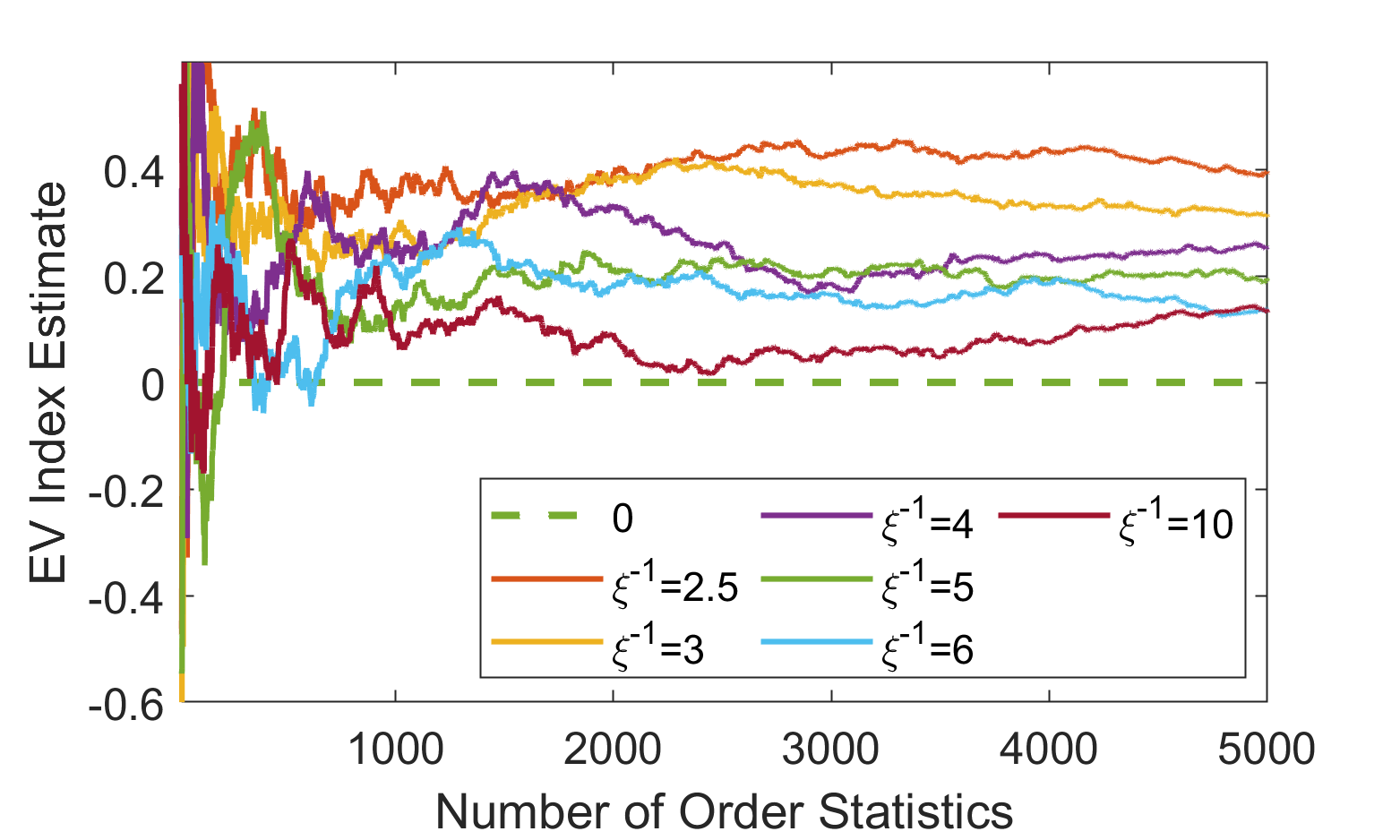}} \quad 
     \subfloat[Moment, $N=10^5$\label{fig:mome_p_5}]{\includegraphics[width=0.45\textwidth]{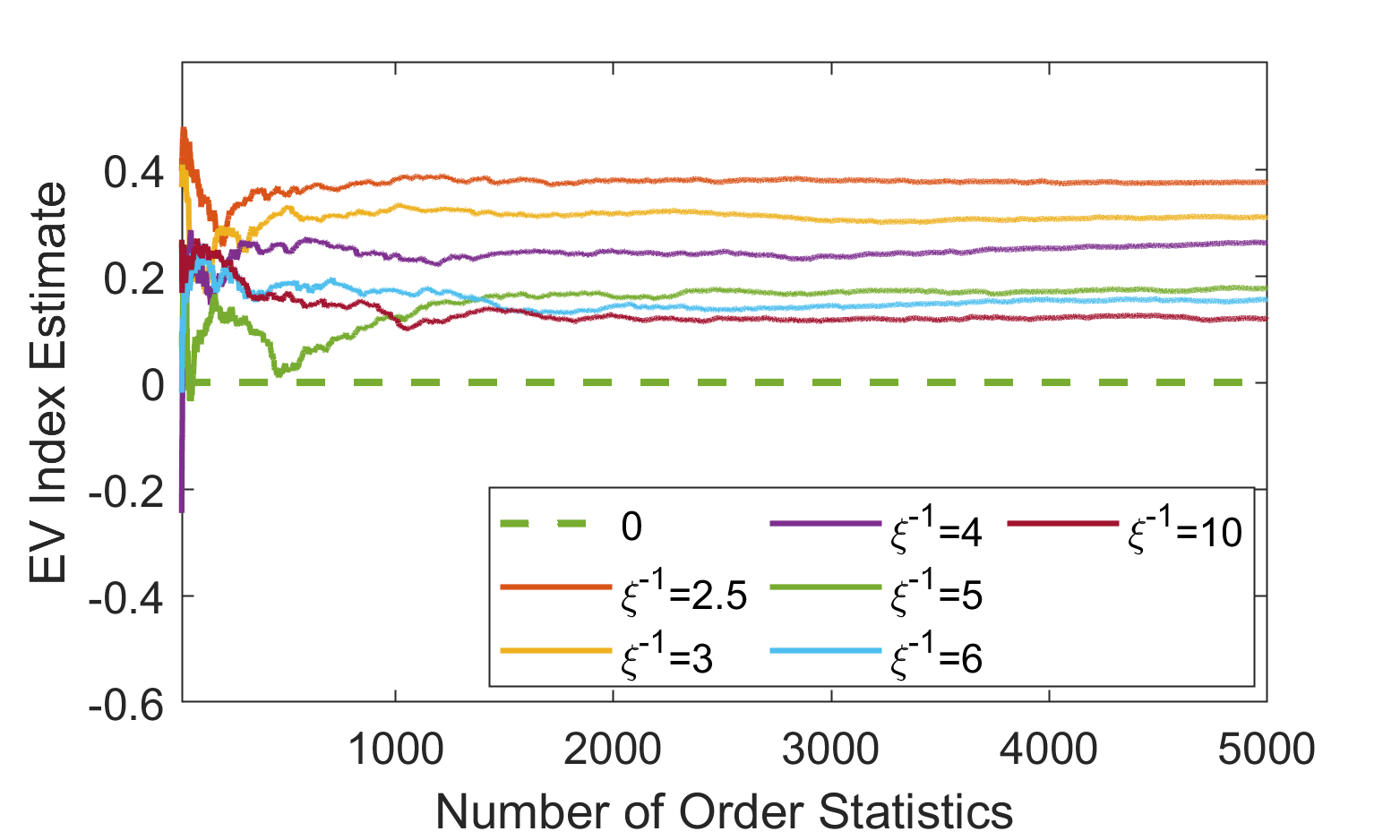}}
     \caption{Estimation results for the extreme value indices with generalized Pareto samples.}
     \label{fig:test_pareto}
 \end{figure}

In the following experiments, we collect data from each distribution in Section \ref{sec:tail_heaviness} and estimate their extreme value index using the Pickands estimator and the moment estimator. In Section \ref{sec:tail_heaviness}, we observed that one cannot obtain relatively accurate estimates with as small as $N=10^4$ samples in heavy-tailed problems. Here we want to check whether we are able to detect these cases with a similar amount of samples and hence to provide warnings of under-estimation caused by heavy tails. Our experimental results, including those for the Pareto distribution, $t$-distribution, and other distributions, are presented in Figures \ref{fig:test_pareto}, \ref{fig:test_t}, and \ref{fig:test_rest}, respectively.

Figure \ref{fig:test_pareto} shows the performance of the estimators with data generated from Pareto distributions. We classify the data as exhibiting ``heavy-tailed behaviors'' when the estimates with the majority of $k$ notably above 0. This characteristic is observed in all cases presented in the figure, showing that both estimators are capable of detecting the heavy tails.
For Pickands estimator with $N=10^4$ (in Figure \ref{fig:pick_p_4}), all the estimates are above 0 for most of the $k$ values, though the estimates are unstable when $k$ varies. For instance, for $\xi^{-1}=4$, the estimates fall below 0 with $k \in [50,120]$ (the purple solid line). However, this variation would not affect the conclusion since the rest of the values are above 0.
Compared to the Pickands estimator, the moment estimates are more stable with $k$, as the estimates are consistently above 0, especially when $k>100$. As we increase the number of samples to $N=10^5$ in Figures \ref{fig:pick_p_5} and \ref{fig:mome_p_5}, the estimates of both estimators are more clearly above 0, which is a stronger warning of heavy tails. In terms of estimating the extreme value index, the Pickands estimates (Figure \ref{fig:pick_p_5}) perform worse than the moment estimates (Figure \ref{fig:mome_p_5}), as they are less consistent with a varying $k$. Additionally, these estimates appear not accurate enough to correctly rank the tail heaviness and hence cannot reliably imply the severity of heavy-tailed behavior. For instance, $\xi^{-1}=6$ (the cyan solid line) consistently has smaller estimates than $\xi^{-1}=10$ (the dark red solid line), whereas the true indices are 1/6 and 1/10 respectively.

\begin{figure}[t]
      \centering
     \subfloat[Pickands, $N=10^4$\label{fig:pick_t_4}]{\includegraphics[width=0.45\textwidth]{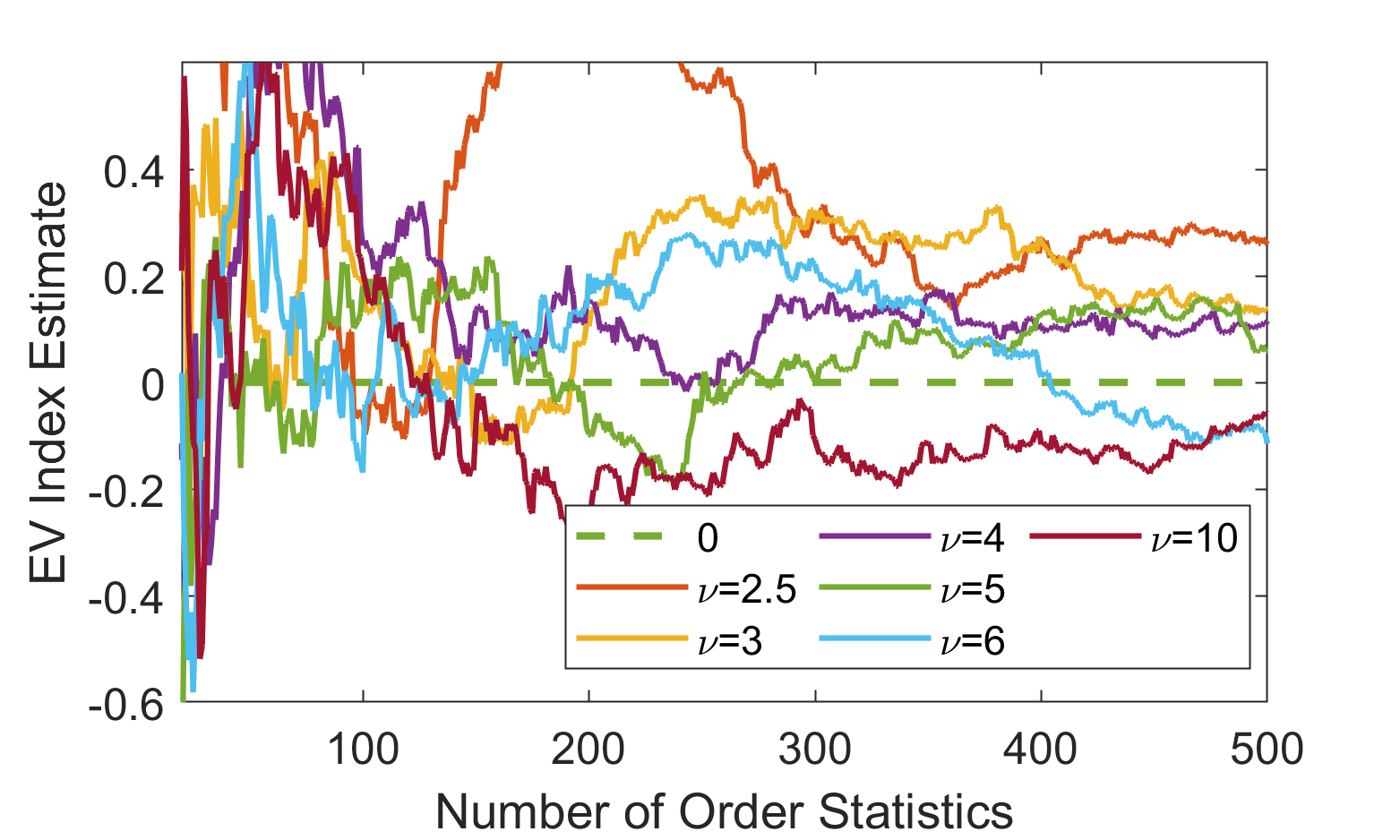}} \quad 
     \subfloat[Moment, $N=10^4$\label{fig:mome_t_4}]{\includegraphics[width=0.45\textwidth]{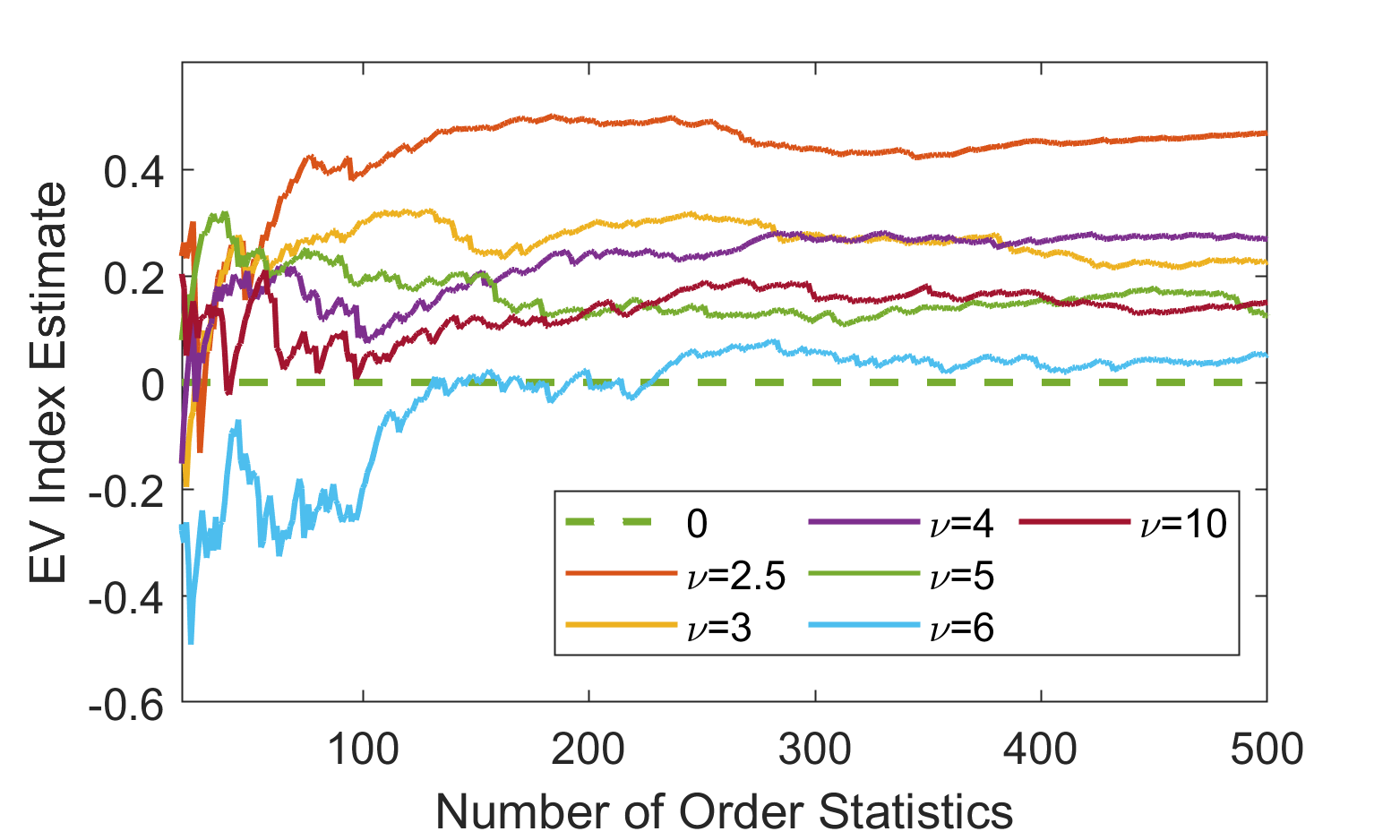}} \quad
     \subfloat[Pickands, $N=10^5$\label{fig:pick_t_5}]{\includegraphics[width=0.45\textwidth]{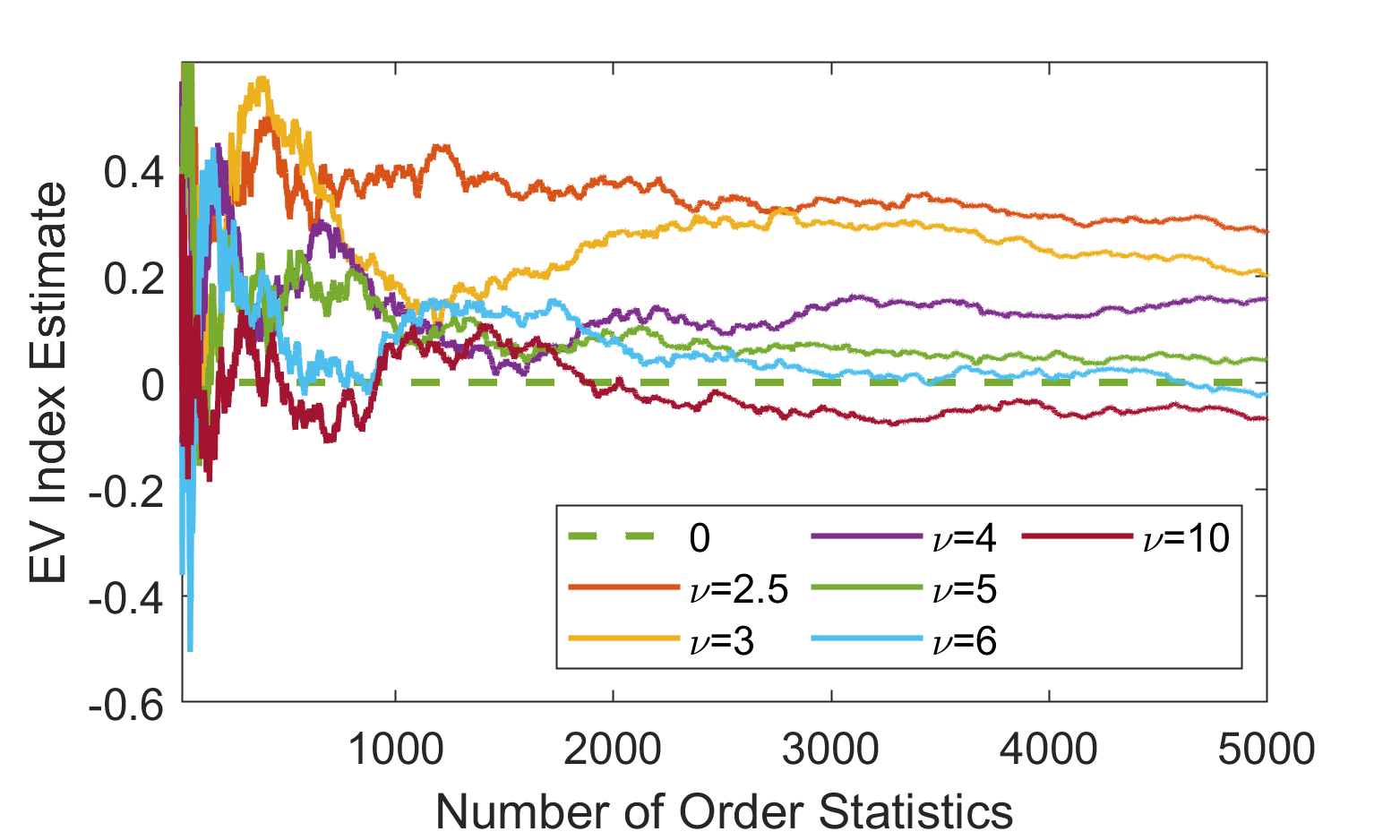}} \quad 
     \subfloat[Moment, $N=10^5$\label{fig:mome_t_5}]{\includegraphics[width=0.45\textwidth]{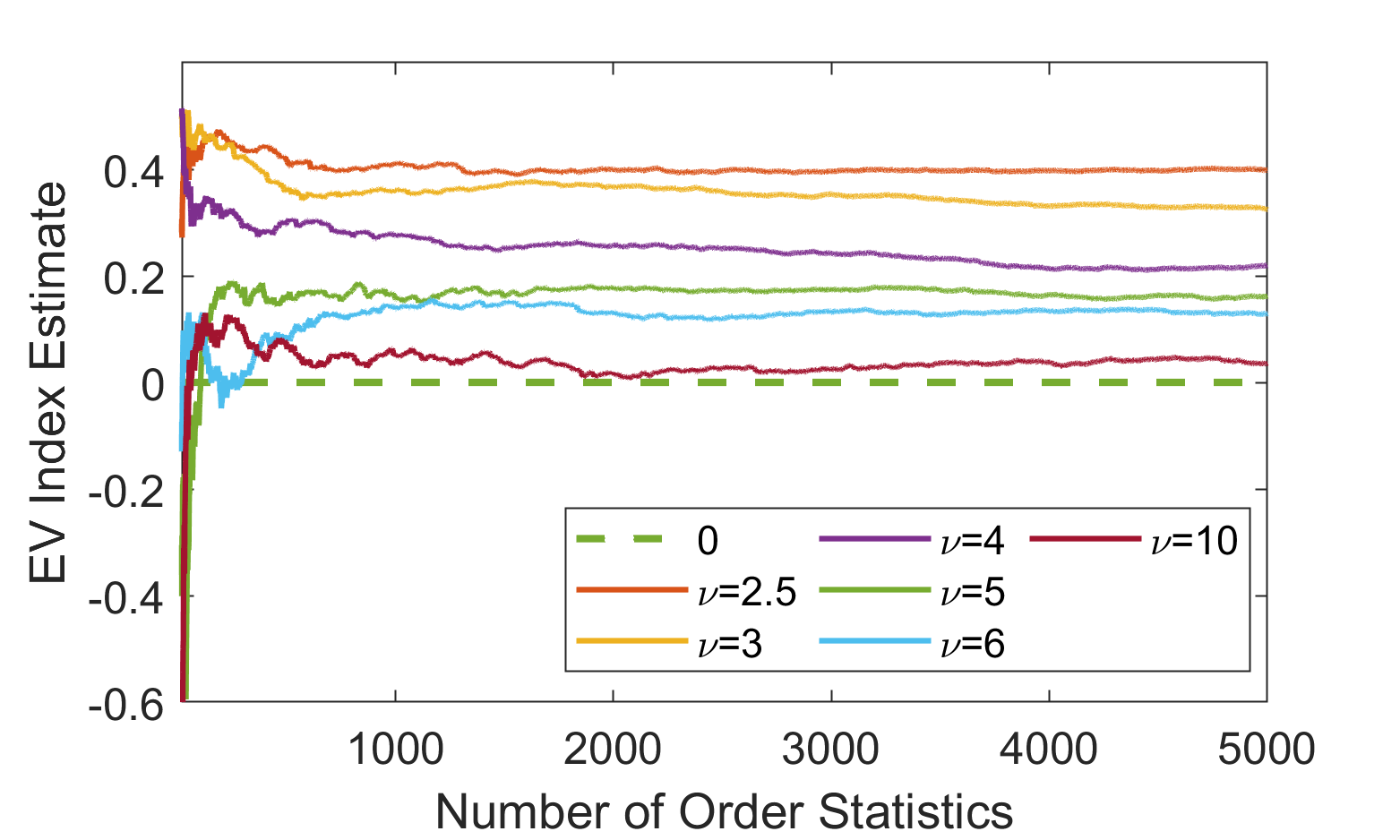}}
     \caption{Estimation results for the extreme value indices with Student's $t$ samples.}
     \label{fig:test_t}
 \end{figure}
 
Next, Figure \ref{fig:test_t} presents the results for data generated from Student's $t$-distribution. We observe that the Pickands estimator may lead to misclassification of heavy-tailed data as light-tailed by yielding estimates smaller than 0 for most $k$ values. This performance is less reliable compared to its performance in the Pareto case.
In particular, with $N=10^4$ samples in Figure \ref{fig:pick_t_4}, we observe that the estimates for $\nu=10$ (the dark red solid line) are consistently smaller than 0. The estimates for $\nu=5$ (the green solid line) and 6 (the cyan solid line) also cross 0 for some $k$ values (around 200 to 250 for $\nu=5$ and 400 to 500 for $\nu=6$). With more samples $N=10^5$ in Figure \ref{fig:pick_t_5}, the estimates for $\nu=10$ (the dark red solid line) are still below 0 for most of the $k$'s. On the other hand, the moment estimator is more consistent and has less variations with varying $k$'s. With $N=10^4$ or $N=10^5$ samples in Figures \ref{fig:mome_t_4} and \ref{fig:mome_t_5}, most of the estimates are positive. Similar to the Pareto cases, the moment estimates are capable of detecting heavy-tailed behaviors, but cannot rank the heaviness correctly when the sample size is small (e.g., with $N=10^4$ in Figure \ref{fig:mome_t_4}).

\begin{figure}[t]
      \centering
     \subfloat[Pickands, $N=10^4$\label{fig:pick_rest_4}]{\includegraphics[width=0.45\textwidth]{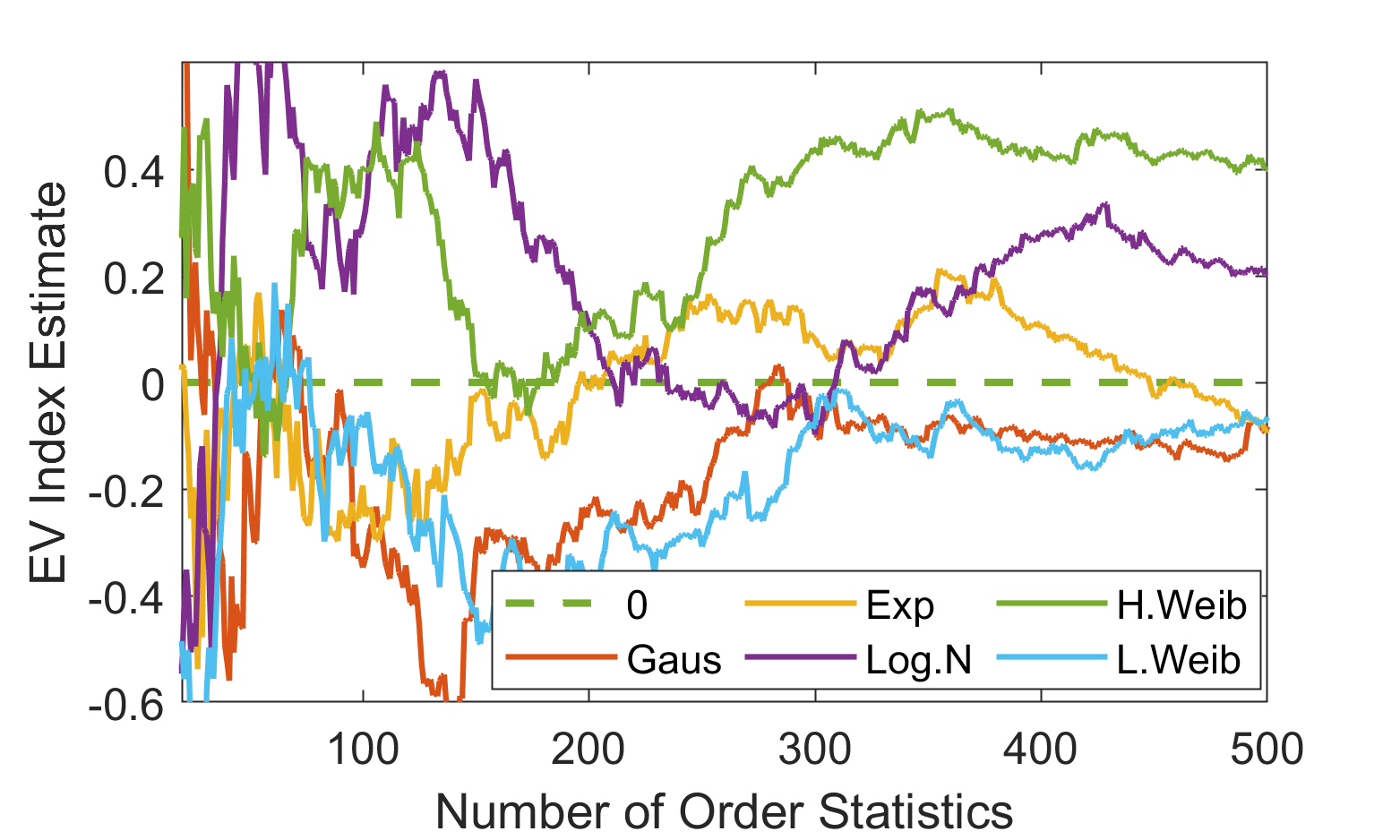}} \quad 
     \subfloat[Moment, $N=10^4$\label{fig:mome_rest_4}]{\includegraphics[width=0.45\textwidth]{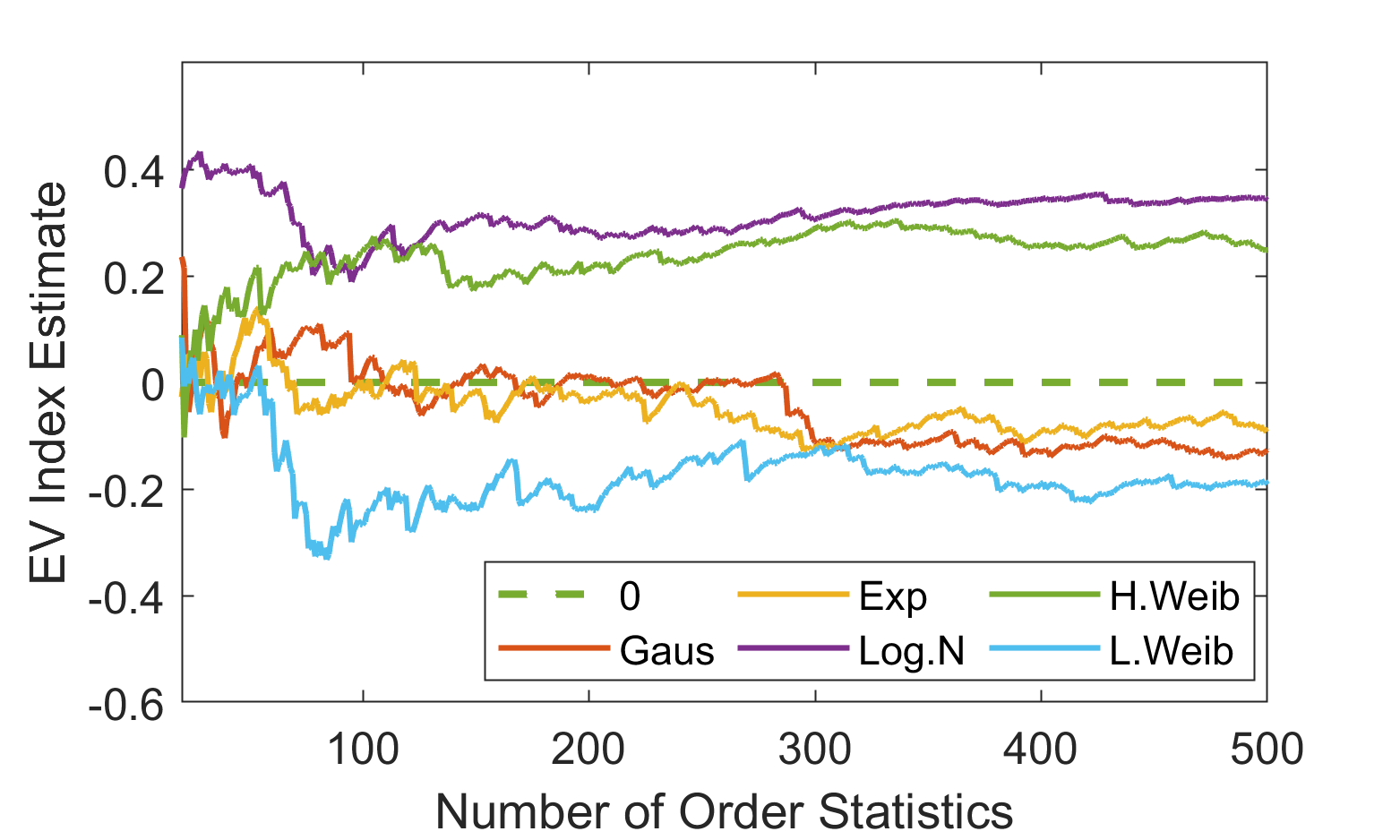}} \quad
     \subfloat[Pickands, $N=10^5$\label{fig:pick_rest_5}]{\includegraphics[width=0.45\textwidth]{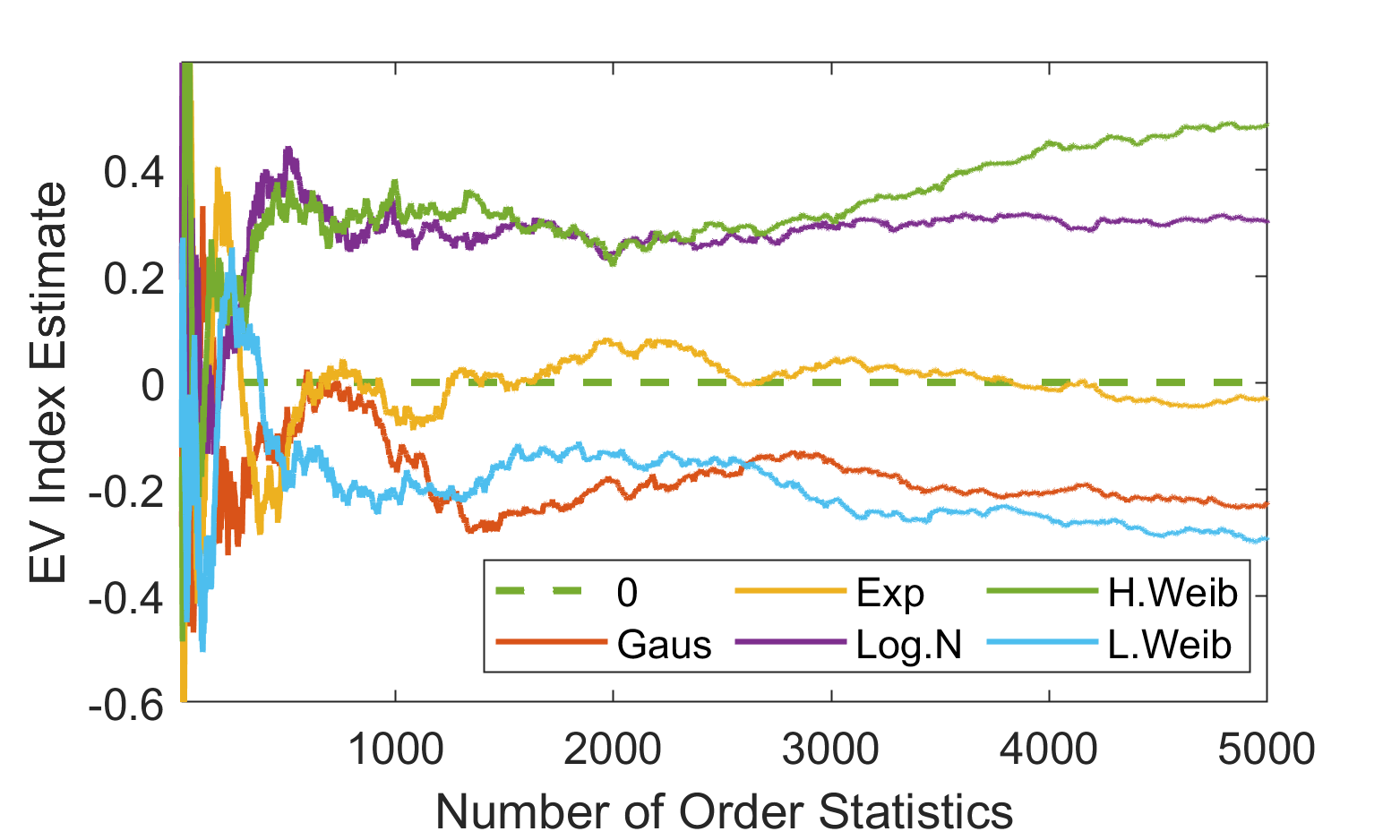}} \quad 
     \subfloat[Moment, $N=10^5$\label{fig:mome_rest_5}]{\includegraphics[width=0.45\textwidth]{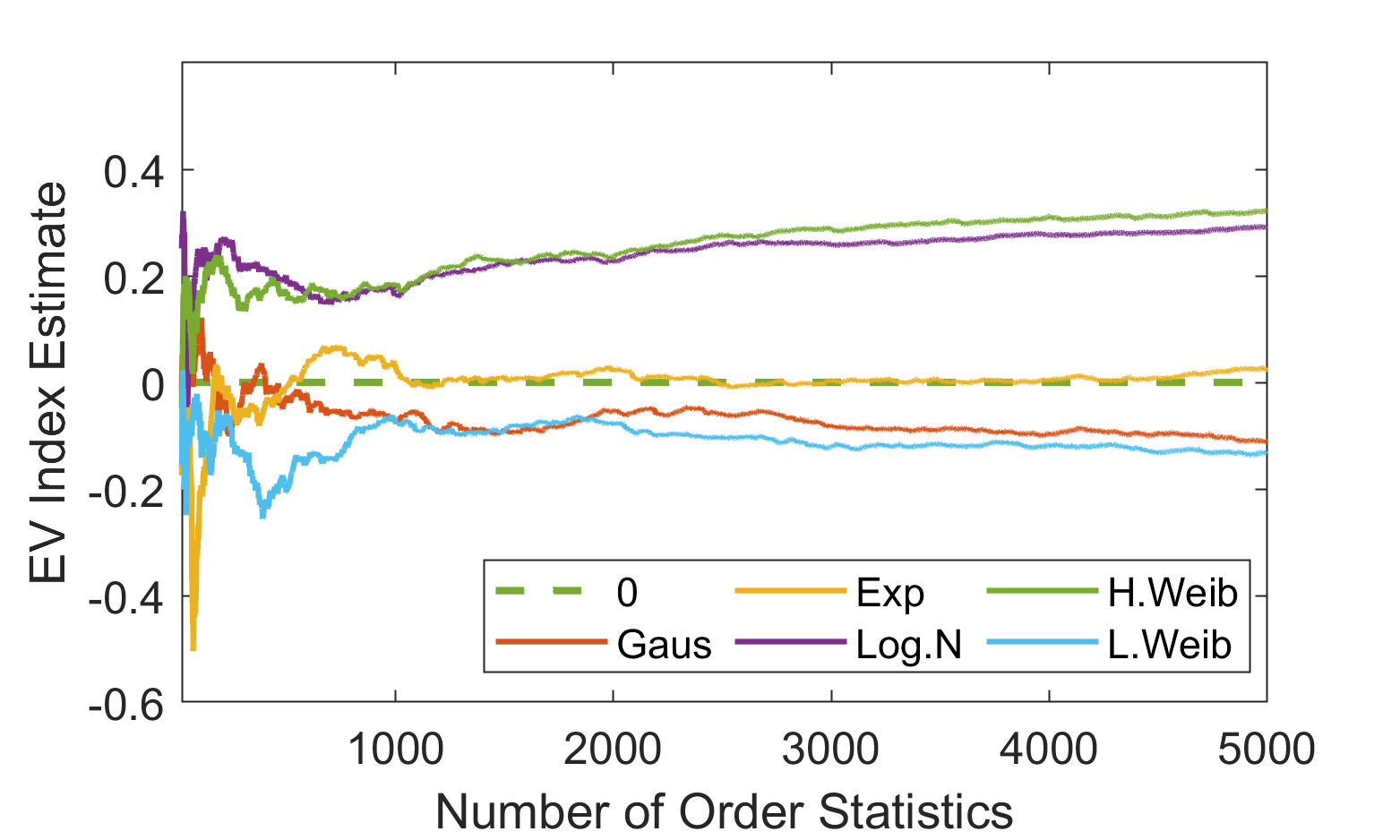}}
     \caption{Estimation results for the extreme value indices with non-power-law-tail samples. }
     \label{fig:test_rest}
 \end{figure}
 
Lastly, we consider samples from light-tailed distributions and subexponential distributions with non-power-law tails. We recall that in Section \ref{sec:exp_slow_varying}, log-normal and heavy-tailed Weibull distributions have similar performances as Pareto tail distributions in terms of under-estimations. Figure \ref{fig:test_rest} shows that extreme value index estimation can successfully detect the risk of under-estimations by classifying these non-power-law subexponential distributions as heavy tails. With these distributions, the extreme value index estimates are consistently positive and range mostly from 0.2 to 0.4. Once again, we find that the extreme value index estimator is unable to accurately gauge the severity of heavy-tailed behaviors. This is evident from the fact that the index estimates for both the log-normal and heavy-tailed Weibull distributions are not substantially different in Figures \ref{fig:pick_rest_5} and \ref{fig:mome_rest_5}. However, as discussed in Section \ref{sec:exp_slow_varying}, the log-normal distributions require a larger sample size than heavy-tailed Weibull distributions for obtaining reliable estimations.

In these experiments, it can be concluded that extreme value index estimators are effective in detecting the risk of under-estimations by classifying whether the data are heavy tailed. This is particularly prominent for the moment estimator which appear to give correct classifications across our considered settings. However, we also note that these estimators may not provide precise estimates for the tail indices and, as a result, may not accurately assess the severity of heavy-tailed behaviors.

\subsection{Summary of Experimental Takeaways.} 

The key findings from our experimental study can be summarized as follows:

\begin{itemize}
    \item Heavy-tailed problems are more vulnerable to the lack of data compared to those with light tails---for reliable estimation of rare-event probabilities in heavy-tailed scenarios, significantly larger sample sizes are required than those needed for light-tailed distributions. Moreover, assessing this uncertainty can be equally difficult.
    
    \item Subexponential distributions, even in the absence of a power-law tail, can be prone to under-estimation of both the point estimate and uncertainty. This is evident from the similar sample size requirements for reliable estimation of rare-event probabilities in subexponential distributions and heavy-tailed distribution problems.
    
    \item  Standard bootstrap methods do not effectively provide valid uncertainty quantification for heavy-tailed problems when the sample size is only moderate.
    
    \item While incorporating generalized Pareto tail extrapolation into the bootstrap improves the coverage of bootstrap CIs, the reliability of this coverage is affected by model misspecification.
    
    \item Estimators for the extreme index, especially the moment estimator, appear effective in detecting the risk of under-estimation of the point estimate and uncertainty by classifying whether a distribution is heavy-tailed.
    
    \item Our suggestion of a dependable approach to ensure reliable estimation is a two-step procedure: First detect whether the data is prone to the risk of under-estimations. If so, employ a sample size that is confidently larger than $n/p$ or, in cases where this is not possible, then any decision-making should account for the substantial under-estimation risk.
\end{itemize}

\begin{APPENDICES}

\section{Results for $P(S_{n}\geq\gamma)$.}\label{sec:another target}
We define the target probability as
\begin{equation}
\tilde{p}(G)=P_{G}(S_{n}\geq\gamma)\label{p_tilde}
\end{equation}
where $X_{i}$'s are governed by an arbitrary distribution $G$. Using $p(G)$ or $\tilde{p}(G)$ to estimate the overshoot will not make a difference in the heavy-tailed case. In fact, we have the following theorems:
\begin{theorem}[Unreliable estimation of $\tilde{p}(F)$ for the heavy-tailed distribution]
Under the same assumptions in Theorem \ref{heavy}, the discrepancy between using a truncated distribution $\tilde F_u$ and the original distribution $F$ in evaluating the probability $\tilde{p}(F)=P(S_n\ge\gamma)$ as $n\to\infty$ is given by
$$ \tilde{p}(\tilde{F}_{u})-\tilde{p}(F)=-\tilde{p}(F)(1+o(1)).$$
\label{heavy_p_tilde}
\end{theorem}

\begin{theorem}[Reliable estimation of $\tilde{p}(F)$ for the heavy-tailed distribution]
Under the same assumptions in Theorem \ref{heavy_large_u}, the discrepancy between using a truncated distribution $\tilde F_u$ and the original distribution $F$ in evaluating the probability $\tilde{p}(F)=P(S_n\ge\gamma)$ as $n\to\infty$ is asymptotically negligible, i.e.,
$$ \tilde{p}(\tilde{F}_{u})-\tilde{p}(F)=o(\tilde{p}(F)).$$
\label{heavy_large_u_p_tilde}
\end{theorem}

In the light-tailed case, the asymptotics of the probabilities $\tilde{p}(F)$ and $\tilde{p}(\tilde{F}_{u})$ will change in the lattice case. But since the changes of the asymptotics are the same for $\tilde{p}(F)$ and $\tilde{p}(\tilde{F}_{u})$, this will not affect the reliability result in Theorem \ref{light}. We summarize that as the following theorem:
\begin{theorem}[Reliable estimation of $\tilde{p}(F)$ for the light-tailed distribution]
Under the same assumptions in Theorem \ref{light}, we have%
\begin{equation}
\lim_{n\rightarrow\infty}\theta^*\sqrt{\psi^{\prime\prime}(\theta^*)2\pi n}e^{\theta^*nb-n\psi(\theta^*)}\tilde{p}(F)
=\lim_{n\rightarrow\infty}\theta^*\sqrt{\psi^{\prime\prime}(\theta^*)2\pi n}e^{\theta^*nb-n\psi(\theta^*)}\tilde{p}(\tilde{F}_{u})
=1
\end{equation}
in the non-lattice case, and
\begin{equation}
\lim_{n\rightarrow\infty}\theta^*\sqrt{\psi^{\prime\prime}(\theta^*)2\pi n}e^{\theta^*nb-n\psi(\theta^*)}\tilde{p}(F)
=\lim_{n\rightarrow\infty}\theta^*\sqrt{\psi^{\prime\prime}(\theta^*)2\pi n}e^{\theta^*nb-n\psi(\theta^*)}\tilde{p}(\tilde{F}_{u})
=\frac{\theta^*h}{1-e^{-\theta^*h}}
\end{equation}
in the lattice case. Moreover, in both cases, the discrepancy between using a truncated distribution $\tilde{F}_{u}$ and the original distribution $F$ in evaluating the probability $\tilde{p}(F)=P(S_{n}\ge\gamma)$ is asymptotically negligible, i.e.,
\begin{equation}
\tilde{p}(\tilde{F}_{u})-\tilde{p}(F)=o(\tilde{p}(F)).
\end{equation}
\label{light_p_tilde}
\end{theorem}

\section{Proofs.}
\label{sec:appendix}

\begin{proof}
{Proof of Theorem \ref{heavy}:}Recall that
\[
p(\tilde{F}_{u})-p(F)=\frac{P(S_{n}>\gamma,X_{i}\leq u\ \forall i=1,\ldots,n)}{F(u)^{n}}-P(S_{n}>\gamma).
\]
To prove the theorem, it suffices to show $P(S_{n}>\gamma,X_{i}\leq u\ \forall i=1,\ldots,n)/F(u)^{n}=o(P(S_{n}>\gamma)).$ Since $u\leq \mu+O((\gamma-n\mu)/\sqrt{\log n})$, there exists $\overline{M}\in (0,\infty)$ s.t. $u\leq \mu+\overline{M}(\gamma-n\mu)/\sqrt{\log n}$ for large $n$, We first let $u=\mu+\overline{M}(\gamma-n\mu)/\sqrt{\log n}$ and define $\sigma^{2}$ as the variance of $X$. We will show $P(S_{n}>\gamma,X_{i}\leq u\ \forall i=1,\ldots,n)=o(P(S_{n}>\gamma))$ and $F(u)^{n}\rightarrow1$. By normalizing $X_{i}$'s, we have%
\begin{align}
& P(S_{n}>\gamma,X_{i}\leq u\ \forall i=1,\ldots,n)\nonumber\\
=& P\left(  \frac{S_{n}-n\mu }{\sigma}>\frac{\gamma-n\mu}{\sigma},\frac{X_{i}-\mu}{\sigma}\leq\frac{\overline{M}(\gamma-n\mu)}{\sigma\sqrt{\log n}}\ \forall i=1,\ldots,n\right) \nonumber\\
\le &  P\left(  \frac{S_{n}-n\mu }{\sigma}\ge \frac{\gamma-n\mu}{\sigma},\frac{X_{i}-\mu}{\sigma}\leq\frac{\overline{M}(\gamma-n\mu)}{\sigma\sqrt{\log n}}\ \forall i=1,\ldots,n\right).\label{normalization_heavy}
\end{align}
Note that the tail distribution function of the normalized random variable $Y=(X-\mu)/\sigma$ is
\[
\bar{F}_{Y}(x)=\bar{F}(\mu+\sigma x)=L(\mu+\sigma x)\frac{(\mu+\sigma x)^{-\alpha}}{x^{-\alpha}}x^{-\alpha}.
\]
By Karamata representation theorem (\cite{bingham1989regular} Theorem 1.3.1), we can see $L(\mu+\sigma x)$ is slowly varying and by Proposition 1.3.6 (iii) in \cite{bingham1989regular}, we know that $L(\mu+\sigma x)(\mu+\sigma x)^{-\alpha}/x^{-\alpha}$ is a slowly varying function. By applying equation (1.45) in \cite{nagaev1979large} (this equation also holds when $y_n=\overline{M}x_n/\alpha_n$ where $x_n,y_n,\alpha_n$ are defined below equation (1.26) in \cite{nagaev1979large}) to (\ref{normalization_heavy}), we have%
\begin{equation}
P(S_{n}>\gamma,X_{i}\leq u\ \forall i=1,\ldots,n)=o\left(  nP\left(\frac{X-\mu}{\sigma}>\frac{\gamma-n\mu}{\sigma}\right)  \right)  =o(n\bar{F}(\gamma-(n-1)\mu)).\label{asymptotic_appro1_heavy}%
\end{equation}
Similarly, by applying equation (1.25b) in \cite{nagaev1979large} to $(X_{i}-\mu)/\sigma$, we have%
\begin{equation}
n\bar{F}(\gamma+1-(n-1)\mu)(1+o(1))=P(S_{n}\geq\gamma+1)\leq P(S_{n}>\gamma)\leq P(S_{n}\geq\gamma)=n\bar{F}(\gamma-(n-1)\mu)(1+o(1)).\label{squeeze_inequality}%
\end{equation}
By Karamata representation theorem (\cite{bingham1989regular} Theorem 1.3.1), we have $L(\gamma+1-(n-1)\mu)/L(\gamma-(n-1)\mu)\rightarrow1$ and thus $\bar{F}(\gamma+1-(n-1)\mu)/\bar{F}(\gamma-(n-1)\mu)\rightarrow1$. Therefore, it follows from (\ref{squeeze_inequality}) that
\begin{equation}
P(S_{n}>\gamma)=n\bar{F}(\gamma-(n-1)\mu)(1+o(1)).\label{asymptotic_appro2_heavy}%
\end{equation}
Comparing (\ref{asymptotic_appro1_heavy}) and (\ref{asymptotic_appro2_heavy}),
we have%
\begin{equation}
P(S_{n}>\gamma,X_{i}\leq u\ \forall i=1,\ldots,n)=o(P(S_{n}>\gamma))\label{asymptotic_appro3_heavy}%
\end{equation}
To show $F(u)^{n}\rightarrow1$, recall that $(\gamma-n\mu)/\sqrt{\log n}=\omega(\sqrt{n})$ and $x^{-\lambda}L(x)\rightarrow0$ as $x\rightarrow \infty$ (\cite{bingham1989regular} Proposition 1.3.6 (v)). We have%
\begin{align*}
n\bar{F}(u)  & =nL\left(  \mu+\frac{\overline{M}(\gamma-n\mu)}{\sqrt{\log n}}\right)\left(  \mu+\frac{\overline{M}(\gamma-n\mu)}{\sqrt{\log n}}\right)  ^{-\alpha}\\
& =\left[  n\left(  \mu+\frac{\overline{M}(\gamma-n\mu)}{\sqrt{\log n}}\right)^{-\frac{\alpha+2}{2}}\right]  \left[  L\left(  \mu+\frac{\overline{M}(\gamma-n\mu)}{\sqrt{\log n}}\right)  \left(  \mu+\frac{\overline{M}(\gamma-n\mu)}{\sqrt{\log n}}\right)^{-\frac{\alpha-2}{2}}\right]  \rightarrow0,
\end{align*}
which implies $F(u)^{n}=(1-\bar{F}(u))^{n}\rightarrow1$. Combining it with (\ref{asymptotic_appro3_heavy}), we get
\begin{equation}
\frac{P(S_{n}>\gamma,X_{i}\leq u\ \forall i=1,\ldots,n)}{F(u)^{n}}=o(P(S_{n}>\gamma))\label{asymptotic_appro4_heavy}%
\end{equation}
when $u=\mu+\overline{M}(\gamma-n\mu)/\sqrt{\log n}$.

For $u\le \mu+\overline{M}(\gamma-n\mu)/\sqrt{\log n}$, since the truncated distribution $\tilde{F}_u(\cdot)$ stochastically dominates $\tilde{F}_{u^{\prime}}(\cdot)$, i.e., $\bar{\tilde{F}}_u(\cdot) \ge \bar{\tilde{F}}_{u^{\prime}}(\cdot)$ for any $u>u^{\prime}$, we must have, for given $\gamma$, $P(S_{n}>\gamma,X_{i}\leq u\ \forall i=1,\ldots,n)/F(u)^{n}$ is non-decreasing in $u$ by the property of usual stochastic order (\cite{shaked2007stochastic} Theorem 1.A.3.(b)). Thus (\ref{asymptotic_appro4_heavy}) holds for any $u\leq \mu+\overline{M}(\gamma-n\mu)/\sqrt{\log n}$. This concludes our proof. $\square$
\end{proof}

\begin{proof}{Proof of Theorem \ref{heavy_large_u}:}
By the same argument in the proof of Theorem \ref{heavy}, we can show $F(u)^{n}\rightarrow1$. By the formula of the approximation error in (\ref{interim}), it suffices to show
\begin{equation}
P(S_{n}>\gamma,\text{ at least one }X_{i}>u)=o(P(S_{n}>\gamma)).\label{asymptotic_appro5_heavy}%
\end{equation}
For $P(S_{n}>\gamma,$ at least one $X_{i}>u)$, we have the inequality%
\[
P(S_{n}>\gamma,\text{ at least one }X_{i}>u)\leq P(\text{at least one }X_{i}>u)\leq n\bar{F}(u).
\]
By the asymptotic property of $P(S_{n}>\gamma)$ in (\ref{asymptotic_appro2_heavy}), to prove (\ref{asymptotic_appro5_heavy}), it suffices to show%
\[
\frac{n\bar{F}(u)}{n\bar{F}(\gamma-(n-1)\mu)}=\frac{\bar{F}(u)}{\bar{F}(\gamma-(n-1)\mu)}\rightarrow0.
\]
Recall that $\bar{F}(\cdot)$ has the form (\ref{rv}). By the choice of $u=\omega((\gamma-n\mu)^{\beta}),\beta>1$ and the property $x^{\lambda}L(x)\rightarrow\infty,x^{-\lambda}L(x)\rightarrow0$ as $x\rightarrow\infty$ for any $\lambda>0$ (\cite{bingham1989regular} Proposition 1.3.6 (v)), we have%
\[
\frac{\bar{F}(u)}{\bar{F}(\gamma-(n-1)\mu)}=\left(  \frac{\gamma-(n-1)\mu}{u}\right)  ^{\alpha}\frac{L(u)}{L(\gamma-(n-1)\mu)}\rightarrow0,
\]
which proves (\ref{asymptotic_appro5_heavy}).$\square$
\end{proof}

\begin{proof}{Proofs of Theorem \ref{heavy_p_tilde} and Theorem \ref{heavy_large_u_p_tilde}:}
By the same arguments in the proof of Theorem \ref{heavy}, we can see (\ref{asymptotic_appro1_heavy}) and (\ref{asymptotic_appro2_heavy}) also hold when the probabilities are replaced by $P(S_{n}\ge\gamma,X_{i}\leq u\ \forall i=1,\ldots,n)$ and $P(S_{n}\ge\gamma)$ respectively. Therefore, the following analog of (\ref{asymptotic_appro3_heavy})
\begin{equation*}
P(S_{n}\ge\gamma,X_{i}\leq u\ \forall i=1,\ldots,n)=o(P(S_{n}\ge\gamma))
\end{equation*}
holds. All the remaining arguments in the proof of Theorem \ref{heavy} can be applied here, which proves Theorem \ref{heavy_p_tilde}. Theorem \ref{heavy_large_u_p_tilde} can be proved similarly. $\square$
\end{proof}

Now let us turn to the proofs of the light-tailed case. As we explained below Lemma \ref{asymptotic_truncated}, the proof is quite lengthy so we would like to provide a roadmap before we go into the details. Recall that the choice of $u$ depends on $n$. Therefore the truncated random variables $X_1,\ldots,X_n\overset{i.i.d}{\sim} \tilde{F}_u$ are actually a triangular array when $n$ varies. Lemma \ref{asymptotic_truncated} essentially establishes the exact asymptotics for the triangular array (in the special case of truncated distribution with varying truncation levels). Theorem 3.7.4 in \cite{dembozeitouni1998} gives the exact asymptotics in the case of i.i.d. random variables. We will use the same techniques there to prove Lemma \ref{asymptotic_truncated}. A key technique in their proof is the Berry-Ess\'{e}en expansion up to order $\sqrt{n}$ which is well-known in the i.i.d. case and we need to make much more efforts to establish a similar expansion for the truncated distributions with varying truncation levels (i.e., Lemma \ref{edgeworth_expansion_tri_array}). Since the proofs of the light-tailed case require more notations, we use $u_{n}$ instead of $u$ to make our proofs more understandable. For convenience, we abbreviate the log moment generating function $\psi_{u_{n}}$ of $\tilde{F}_{u_{n}}$ to $\psi_{n}$.

Before proving Lemma \ref{asymptotic_truncated}, we need an ancillary lemma about the convergence rates of $\psi_{n}^{\prime}$ and $\psi_{n}$ to $\psi^{\prime}$ and $\psi$ at the point $\theta^*$.

\begin{lemma}[Convergence rates of $\psi_{n}^{\prime}$ and $\psi_{n}$]
\label{implication_of_un} If the truncation level $u_{n}$ satisfies condition (\ref{restrictions_on_un}), then we have
\[
\lim_{n\rightarrow\infty}n(\psi_{n}^{\prime}(\theta^*)-\psi^{\prime}(\theta^*))=0,~\lim_{n\rightarrow\infty}n(\psi_{n}(\theta^*)-\psi(\theta^*))=0.
\]
\end{lemma}

\begin{proof}{Proof of Lemma \ref{implication_of_un}:}
First, let us prove
\[
\lim_{n\rightarrow\infty}n(\psi_{n}^{\prime}(\theta^*)-\psi^{\prime}(\theta^*))=0.
\]
By direct calculation, we have%
\begin{align*}
n(\psi_{n}^{\prime}(\theta^*)-\psi^{\prime}(\theta^*)) &=n\left(  \frac{E[Xe^{\theta^*X}|X\leq u_{n}]}{E[e^{\theta^*X}|X\leq u_{n}]}-\frac{E[Xe^{\theta^*X}]}{E[e^{\theta^*X}]}\right)  \\
&  =n\left(  \frac{E[Xe^{\theta^*X}I(X\leq u_{n})]}{E[e^{\theta^*X}I(X\leq u_{n})]}-\frac{E[Xe^{\theta^*X}]}{E[e^{\theta^*X}]}\right)  \\
&  =n\frac{E[Xe^{\theta^*X}]E[e^{\theta^*X}I(X>u_{n})]-E[e^{\theta^*X}]E[Xe^{\theta^*X}I(X>u_{n})]}{E[e^{\theta^*X}I(X\leq u_{n})]E[e^{\theta^*X}]}.
\end{align*}
Note that%
\[
0\leq nE[e^{\theta^*X}I(X>u_{n})]\leq\frac{n}{e^{\theta^{\prime}u_{n}}}E[e^{(\theta^*+\theta^{\prime})X}I(X>u_{n})]\leq\frac{n}{e^{\theta^{\prime}u_{n}}}
E[e^{(\theta^*+\theta^{\prime})X}],
\]
and%
\[
0\leq nE[Xe^{\theta^*X}I(X>u_{n})]\leq\frac{n}{e^{\theta^{\prime}u_{n}}}E[Xe^{(\theta^*+\theta^{\prime})X}I(X>u_{n})]\leq\frac{n}{e^{\theta^{\prime}u_{n}}}E[Xe^{(\theta^*+\theta^{\prime})X}I(X>0)],
\]
when $n$ is large enough such that $u_{n}>0$. Since $\theta^*+\theta^{\prime}\in\mathcal{D}^{\circ}$, we have
\[
E[e^{(\theta^*+\theta^{\prime})X}]<\infty,\text{ }E[Xe^{(\theta^*+\theta^{\prime})X}I(X>0)]<\infty.
\]
Thus condition (\ref{restrictions_on_un}) implies
\begin{equation}
\lim_{n\rightarrow\infty}nE[e^{\theta^*X}I(X>u_{n})]=\lim_{n\rightarrow\infty}nE[Xe^{\theta^*X}I(X>u_{n})]=0,\label{limit_of_tail_expectation_of_MGF}%
\end{equation}
which further implies
\[
\lim_{n\rightarrow\infty}n(\psi_{n}^{\prime}(\theta^*)-\psi^{\prime}(\theta^*))=0.
\]

Then let us prove
\[
\lim_{n\rightarrow\infty}n(\psi_{n}(\theta^*)-\psi(\theta^*))=0.
\]
By direct calculation, we have%
\begin{align*}
n(\psi_{n}(\theta^*)-\psi(\theta^*)) &  =n(\log E[e^{\theta^*X}|X\leq u_{n}]-\log E[e^{\theta^*X}])\\
&  =n(\log E[e^{\theta^*X}I(X\leq u_{n})]-\log P(X\leq u_{n})-\log E[e^{\theta^*X}])\\
&  =n\log\left(  \frac{E[e^{\theta^*X}I(X\leq u_{n})]}{E[e^{\theta^*X}]}\right)  -n\log P(X\leq u_{n})\\
&  =n\log\left(  1-\frac{E[e^{\theta^*X}I(X>u_{n})]}{E[e^{\theta^*}X]}\right)  -n\log(1-P(X>u_{n})).
\end{align*}
Notice that $\theta^*+\theta^{\prime}\in\mathcal{D}^{\circ}$ and $\theta^*>0$ implies $\theta^{\prime}\in\mathcal{D}^{\circ}$, i.e., $Ee^{\theta^{\prime}X}<\infty$. Then by the similar argument for (\ref{limit_of_tail_expectation_of_MGF}), we have
\begin{equation}
\lim_{n\rightarrow\infty}nP(X>u_{n})=0.\label{limit_of_tail_probability}%
\end{equation}
Thus (\ref{limit_of_tail_expectation_of_MGF}) and (\ref{limit_of_tail_probability}) implies%
\begin{align*}
\lim_{n\rightarrow\infty}n(\psi_{n}(\theta^*)-\psi(\theta^*)) & =\lim_{n\rightarrow\infty}n\log\left(  1-\frac{E[e^{\theta^*X}I(X>u_{n})]}{E[e^{\theta^*X}]}\right)  -\lim_{n\rightarrow\infty}n\log(1-P(X>u_{n}))\\
&  =-\lim_{n\rightarrow\infty}n\frac{E[e^{\theta^*X}I(X>u_{n})]}{E[e^{\theta^*X}]}+\lim_{n\rightarrow\infty}nP(X>u_{n})\\
&  =0.
\end{align*}
$\square$
\end{proof}

\begin{proof}{Proof of Lemma \ref{asymptotic_truncated}:}
By a change of measure, we define a new distribution $F_{n}^{(\theta^*)}$ by%
\begin{equation}
dF_{n}^{(\theta^*)}(x)=e^{\theta^*x-\psi_{n}(\theta^*)}d\tilde{F}_{u_{n}}(x). \label{change_of_measure}%
\end{equation}
Let $\{X_{n,m}^{(\theta^*)},n\geq1,1\leq m\leq n\}$ be a triangular sequence of independent random variables where $X_{n,m}^{(\theta^*)}$ has the distribution $F_{n}^{(\theta^*)}$and then we define a centered random variable $Y_{n,m}=(X_{n,m}^{(\theta^*)}-b-\mu_{n})/\sqrt{\psi^{\prime\prime}(\theta^*)}$ where%
\begin{equation}
\mu_{n}=E[X_{n,1}^{(\theta^*)}-b]=\psi_{n}^{\prime}(\theta^*)-b=\psi_{n}^{\prime}(\theta^*)-\psi^{\prime}(\theta^*).\label{mu_n}
\end{equation}
Let $H_{n}$ denote the distribution function of $W_{n}=\sum_{m=1}^{n}Y_{n,m}/\sqrt{n}$. Then it follows that%
\begin{align*}
p(\tilde{F}_{u_{n}})  &  =P(S_{n}>bn|X_{i}\leq u_{n},i=1,\ldots,n)\\
&  =\int_{\mathbb{R}^{n}}I\left(  \sum_{i=1}^{n}x_{i}>bn\right)  d\tilde{F}_{u_{n}}(x_{1})d\tilde{F}_{u_{n}}(x_{2})\cdots d\tilde{F}_{u_{n}}(x_{n})\\
&  =\int_{\mathbb{R}^{n}}I\left(  \sum_{i=1}^{n}x_{i}>bn\right)e^{n\psi_{n}(\theta^*)-\theta^*\sum_{i=1}^{n}x_{i}}dF_{n}^{(\theta^*)}(x_{1})dF_{n}^{(\theta^*)}(x_{2})\cdots dF_{n}^{(\theta^*)}(x_{n})\\
&  =E\left[  e^{n\psi_{n}(\theta^*)-\theta^*\sum_{i=1}^{n}X_{n,i}^{(\theta^*)}}I\left(  \sum_{i=1}^{n}X_{n,i}^{(\theta^*)}>bn\right)  \right] \\
&  =E[e^{n\psi_{n}(\theta^*)-\theta^*(\sqrt{\psi^{\prime\prime}(\theta^*)n}W_{n}+n(b+\mu_{n}))}I(\sqrt{\psi^{\prime\prime}(\theta^*)n}W_{n}+n(b+\mu_{n})>bn)]\\
&  =e^{n\psi_{n}(\theta^*)-\theta^*n(b+\mu_{n})}\int_{(-\sqrt{n}\mu_{n}/\sqrt{\psi^{\prime\prime}(\theta^*)},\infty)}e^{-\theta^*\sqrt{\psi^{\prime\prime}(\theta^*)n}x}dH_{n}(x).
\end{align*}
For simplicity, we define $\lambda_{n}=\theta^*\sqrt{\psi^{\prime\prime}(\theta^*)n}$ and $a_{n}=-\sqrt{n}\mu_{n}/\sqrt{\psi^{\prime\prime}(\theta^*)}$. By the integration by parts, we obtain
\begin{align}
\theta^*\sqrt{\psi^{\prime\prime}(\theta^*)2\pi n}e^{\theta^*n(b+\mu_{n})-n\psi_{n}(\theta^*)}p(\tilde{F}_{u_{n}})  &=\sqrt{2\pi}\lambda_{n}\int_{(a_{n},\infty)}e^{-\lambda_{n}x}dH_{n}(x) \nonumber\\
&  =\sqrt{2\pi}\lambda_{n}^{2}\int_{(a_{n},\infty)}(H_{n}(x)-H_{n}(a_{n}))e^{-\lambda_{n}x}dx \nonumber\\
&  =\sqrt{2\pi}\lambda_{n}\int_{(-\theta^*n\mu_{n},\infty)}\left(H_{n}\left(  \frac{x}{\lambda_{n}}\right)  -H_{n}(a_{n})\right)  e^{-x}dx.\label{integration_by_parts}
\end{align}

(1): When $X$ is non-lattice, we can get the following Berry-Ess\'{e}en expansion of $H_{n}(x)$
\begin{equation}
\lim_{n\rightarrow\infty}\left(  \sqrt{n}\sup_{x}\left\vert H_{n}(x)-\Phi\left(  \frac{x}{\sigma_{n}}\right)  -\frac{m_{3,n}}{6\sigma_{n}^{3}\sqrt{n}}\left(  1-\left(  \frac{x}{\sigma_{n}}\right)  ^{2}\right)\phi\left(  \frac{x}{\sigma_{n}}\right)  \right\vert \right)  =0,
\label{equ_edgeworth_expansion_tri_array}%
\end{equation}
where
\[
\sigma_{n}^{2}=EY_{n,1}^{2}=\frac{\psi_{n}^{\prime\prime}(\theta^*)}{\psi^{\prime\prime}(\theta^*)},
\]%
\[
m_{3,n}=EY_{n,1}^{3}=\frac{\psi_{n}^{\left(  3\right)  }(\theta^*)}{\sqrt{\psi^{\prime\prime}(\theta^*)}^{3}},
\]
and $\phi\left(  x\right)  ,\Phi\left(  x\right)  $ are the density and distribution function of the standard normal distribution respectively. The proof of (\ref{equ_edgeworth_expansion_tri_array}) is deferred to Lemma \ref{edgeworth_expansion_tri_array}.

Define%
\begin{align*}
&  c_{n}=\sqrt{2\pi}\lambda_{n}\int_{(-\theta^*n\mu_{n},\infty)}\left[\Phi\left(  \frac{x}{\sigma_{n}\lambda_{n}}\right)  +\frac{m_{3,n}}{6\sigma_{n}^{3}\sqrt{n}}\left(  1-\left(  \frac{x}{\sigma_{n}\lambda_{n}}\right)  ^{2}\right)  \phi\left(  \frac{x}{\sigma_{n}\lambda_{n}}\right)\right. \\
&  \left.  -\Phi\left(  \frac{a_{n}}{\sigma_{n}}\right)  -\frac{m_{3,n}}{6\sigma_{n}^{3}\sqrt{n}}\left(  1-\left(  \frac{a_{n}}{\sigma_{n}}\right)^{2}\right)  \phi\left(  \frac{a_{n}}{\sigma_{n}}\right)  \right]  e^{-x}dx.
\end{align*}
By Lemma \ref{implication_of_un} and formula of $\mu_n$ in (\ref{mu_n}), we have 
\begin{equation}
\lim_{n\rightarrow\infty}n\mu_{n}=\lim_{n\rightarrow\infty}n(\psi_{n}^{\prime}(\theta^*)-\psi^{\prime}(\theta^*))=0. \label{limit_of_mu_n}%
\end{equation}
By (\ref{limit_of_mu_n}) and (\ref{equ_edgeworth_expansion_tri_array}), we can see%
\begin{equation}
\lim_{n\rightarrow\infty}\left\vert \theta^*\sqrt{\psi^{\prime\prime}(\theta^*)2\pi n}e^{\theta^*n(b+\mu_{n})-n\psi_{n}(\theta^*)}p(\tilde{F}_{u_{n}})-c_{n}\right\vert =0.
\label{limit_relation_between_truncated_probability_and_c_n}%
\end{equation}
Therefore the asymptotic property of $p(\tilde{F}_{u_{n}})$ can be obtained via the limit of $c_{n}$.

To get the limit of $c_{n}$, we divide it into two parts:%
\begin{align*}
&  c_{n}=\sqrt{2\pi}\lambda_{n}\int_{(-\theta^*n\mu_{n},\infty)}\frac{m_{3,n}}{6\sigma_{n}^{3}\sqrt{n}}\left[\left(  1-\left(  \frac{x}{\sigma_{n}\lambda_{n}}\right)  ^{2}\right)  \phi\left(  \frac{x}{\sigma_{n}\lambda_{n}}\right)  -\left(  1-\left(\frac{a_{n}}{\sigma_{n}}\right)  ^{2}\right)  \phi\left(  \frac{a_{n}}{\sigma_{n}}\right)  \right]  e^{-x}dx\\
&  +\sqrt{2\pi}\lambda_{n}\int_{(-\theta^*n\mu_{n},\infty)}\left[\Phi\left(  \frac{x}{\sigma_{n}\lambda_{n}}\right)  -\Phi\left(  \frac{a_{n}}{\sigma_{n}}\right)  \right]  e^{-x}dx.
\end{align*}
For the first term, we know that $m_{3,n}\rightarrow\psi^{\left(  3\right)}(\theta^*)/\sqrt{\psi^{\prime\prime}(\theta^*)}^{3}$, $\sigma_{n}^{2}\rightarrow1$, $n\mu_{n}\rightarrow0$ and $\sup_{x}|  \left(1-x^{2}\right)  \phi\left(  x\right)|  <\infty$. Besides, recall that $\lambda_{n}=\theta^*\sqrt{\psi^{\prime\prime}(\theta^*)n}\rightarrow\infty$ and $a_{n}=-\sqrt{n}\mu_{n}/\sqrt{\psi^{\prime\prime}(\theta^*)}\rightarrow0$. So by dominated convergence theorem, we have%
\begin{align*}
&  \sqrt{2\pi}\lambda_{n}\int_{(-\theta^*n\mu_{n},\infty)}\frac{m_{3,n}}{6\sigma_{n}^{3}\sqrt{n}}\left[\left(  1-\left(  \frac{x}{\sigma_{n}\lambda_{n}}\right)  ^{2}\right)  \phi\left(  \frac{x}{\sigma_{n}\lambda_{n}}\right)  -\left(  1-\left(\frac{a_{n}}{\sigma_{n}}\right)  ^{2}\right)  \phi\left(  \frac{a_{n}}{\sigma_{n}}\right)  \right]  e^{-x}dx\\
&  \rightarrow\sqrt{2\pi}\theta^*\int_{(0,\infty)}\left[  \frac{\psi^{\left(  3\right)  }(\theta^*)}{6\psi^{\prime\prime}(\theta^*)}\phi(0)-\frac{\psi^{\left(  3\right)  }(\theta^*)}{6\psi^{\prime\prime}(\theta^*)}\phi(0)\right]  e^{-x}dx=0
\end{align*}
as $n\rightarrow\infty$. For the second term, since $\sigma_{n}^{2}\rightarrow1$, $n\mu_{n}\rightarrow0$, $\lambda_{n}=\theta^*\sqrt{\psi^{\prime\prime}(\theta^*)n}\rightarrow\infty$, $\sqrt{n}a_{n}=-n\mu_{n}/\sqrt{\psi^{\prime\prime}(\theta^*)}\rightarrow0$ and $\sup_{x}|  \phi^{\prime}\left(  x\right) |  <\infty$, it follows by the second-order Taylor expansion of $\Phi(\cdot)$ and the dominated convergence theorem that
\begin{align*}
&  \lim_{n\rightarrow\infty}\left\{  \sqrt{2\pi}\lambda_{n}\int_{(-\theta^*n\mu_{n},\infty)}\left[  \Phi\left(  \frac{x}{\sigma_{n}\lambda_{n}}\right)  -\Phi\left(  \frac{a_{n}}{\sigma_{n}}\right)  \right]e^{-x}dx\right\} \\
&  =\lim_{n\rightarrow\infty}\left\{  \sqrt{2\pi}\lambda_{n}\int_{(-\theta^*n\mu_{n},\infty)}\left[  \phi\left(  \frac{a_{n}}{\sigma_{n}}\right)  \left(  \frac{x}{\sigma_{n}\lambda_{n}}-\frac{a_{n}}{\sigma_{n}}\right)  +\frac{1}{2}\phi^{\prime}\left(  \eta\right)  \left(  \frac{x}{\sigma_{n}\lambda_{n}}-\frac{a_{n}}{\sigma_{n}}\right)  ^{2}\right]e^{-x}dx\right\} \\
&  =\sqrt{2\pi}\phi\left(  0\right)  \int_{0}^{\infty}xe^{-x}dx=1,
\end{align*}
where $\eta$ lies between $x/\sigma_{n}\lambda_{n}$ and $a_{n}/\sigma_{n}$. The limits of the two terms of $c_{n}$ imply that $\lim_{n\rightarrow\infty}c_{n}=1$. Then by (\ref{limit_relation_between_truncated_probability_and_c_n}), we obtain%
\[
\lim_{n\rightarrow\infty}\theta^*\sqrt{\psi^{\prime\prime}(\theta^*)2\pi n}e^{\theta^*n(b+\mu_{n})-n\psi_{n}(\theta^*)}p(\tilde{F}_{u_{n}})=1.
\]
Finally by $n\mu_{n}\rightarrow0$ and $n(\psi_{n}(\theta^*)-\psi(\theta^*))\rightarrow0$ in Lemma \ref{implication_of_un}, we obtain%
\[
\lim_{n\rightarrow\infty}\theta^*\sqrt{\psi^{\prime\prime}(\theta^*)2\pi n}e^{\theta^*nb-n\psi(\theta^*)}p(\tilde{F}_{u_{n}})=1.
\]

(2): When $X$ has a lattice distribution, we can get the following analog of Berry-Ess\'{e}en expansion of $H_{n}(x)$%
\begin{align}
&  \lim_{n\rightarrow\infty}\left(  \sqrt{n}\sup_{x}\left\vert H_{n}(x)-\Phi\left(  \frac{x}{\sigma_{n}}\right)  -\frac{m_{3,n}}{6\sigma_{n}^{3}\sqrt{n}}\left(  1-\left(  \frac{x}{\sigma_{n}}\right)  ^{2}\right)\phi\left(  \frac{x}{\sigma_{n}}\right)  \right.  \right. \nonumber\\
&  \left.  \left.  -\frac{h}{\sqrt{\psi^{\prime\prime}(\theta^*)}\sqrt{n}\sigma_{n}}\phi\left(  \frac{x}{\sigma_{n}}\right)  S\left(\frac{(x\sqrt{\psi^{\prime\prime}(\theta^*)}+\mu_{n}\sqrt{n})}{h}\sqrt{n}\right)  \right\vert \right)=0,\label{equ_edgeworth_expansion_tri_array_lattice}%
\end{align}
where $S(x)$ is a right-continuous function with period $1$ defined by%
\[
S(x)=-x+\frac{1}{2},\quad0\leq x<1,
\]%
\[
\sigma_{n}^{2}=EY_{n,1}^{2}=\frac{\psi_{n}^{\prime\prime}(\theta^*)}{\psi^{\prime\prime}(\theta^*)},
\]%
\[
m_{3,n}=EY_{n,1}^{3}=\frac{\psi_{n}^{\left(  3\right)  }(\theta^*)}{\sqrt{\psi^{\prime\prime}(\theta^*)}^{3}},
\]
and $\phi\left(  x\right)  ,\Phi\left(  x\right)  $ are the density and distribution function of the standard normal distribution respectively. The proof of (\ref{equ_edgeworth_expansion_tri_array_lattice}) is deferred to Lemma \ref{edgeworth_expansion_tri_array}.

Thus, by the preceding argument for the non-lattice case, we have%
\begin{align}
&  \lim_{n\rightarrow\infty}\theta^*\sqrt{\psi^{\prime\prime}(\theta^*)2\pi n}e^{\theta^*nb-n\psi(\theta^*)}p(\tilde{F}_{u_{n}})\nonumber\\
&  =1+\lim_{n\rightarrow\infty}\sqrt{2\pi}\lambda_{n}\int_{(-\theta^*n\mu_{n},\infty)}\frac{h}{\sqrt{\psi^{\prime\prime}(\theta^*)}\sqrt{n}\sigma_{n}}\left[  \phi\left(  \frac{x}{\lambda_{n}\sigma_{n}}\right)S\left(  \frac{(x\sqrt{\psi^{\prime\prime}(\theta^*)}/\lambda_{n}+\mu_{n}\sqrt{n})}{h}\sqrt{n}\right)  \right. \nonumber\\
&  \left.  -\phi\left(  \frac{a_{n}}{\sigma_{n}}\right)  S\left(  \frac{(a_{n}\sqrt{\psi^{\prime\prime}(\theta^*)}+\mu_{n}\sqrt{n})}{h}\sqrt{n}\right)  \right]  e^{-x}dx,
\label{limit_relation_between_truncated_probability_and_c_n_lattice}%
\end{align}
where the additional term in (\ref{limit_relation_between_truncated_probability_and_c_n_lattice}) is due to the additional term in the Berry-Ess\'{e}en expansion (\ref{equ_edgeworth_expansion_tri_array_lattice}). Recall that $\lambda_{n}=\theta^*\sqrt{\psi^{\prime\prime}(\theta^*)n}$ and $a_{n}=-\sqrt{n}\mu_{n}/\sqrt{\psi^{\prime\prime}(\theta^*)}$. It follows that%
\begin{align*}
&  \lim_{n\rightarrow\infty}\sqrt{2\pi}\lambda_{n}\int_{(-\theta^*n\mu_{n},\infty)}\frac{h}{\sqrt{\psi^{\prime\prime}(\theta^*)}\sqrt{n}\sigma_{n}}\left[  \phi\left(  \frac{x}{\lambda_{n}\sigma_{n}}\right)S\left(  \frac{(x\sqrt{\psi^{\prime\prime}(\theta^*)}/\lambda_{n}+\mu_{n}\sqrt{n})}{h}\sqrt{n}\right)  \right. \\
&  \left.  -\phi\left(  \frac{a_{n}}{\sigma_{n}}\right)  S\left(  \frac{(a_{n}\sqrt{\psi^{\prime\prime}(\theta^*)}+\mu_{n}\sqrt{n})}{h}\sqrt{n}\right)  \right]  e^{-x}dx\\
&  =\lim_{n\rightarrow\infty}\sqrt{2\pi}\theta^*\int_{(-\theta^*n\mu_{n},\infty)}\frac{h}{\sigma_{n}}\left[  \phi\left(  \frac{x}{\lambda_{n}\sigma_{n}}\right)  S\left(  \frac{(x\sqrt{\psi^{\prime\prime}(\theta^*)}/\lambda_{n}+\mu_{n}\sqrt{n})}{h}\sqrt{n}\right)-\phi\left(  \frac{a_{n}}{\sigma_{n}}\right)  S\left(0\right)  \right]  e^{-x}dx.
\end{align*}
By dominated convergence theorem, $n\mu_{n}\rightarrow0$, $\sigma_{n}\rightarrow1$, $\lambda_{n}\rightarrow\infty$, and%
\[
S\left(  \frac{(x\sqrt{\psi^{\prime\prime}(\theta^*)}/\lambda_{n}+\mu_{n}\sqrt{n})}{h}\sqrt{n}\right)=S\left(  \frac{x}{\theta^*h}+\frac{n\mu_n}{h}\right)  \rightarrow S\left(  \frac{x}{\theta^*h}\right)  \text{, a.e.,}%
\]
we can get%
\begin{align*}
&  \lim_{n\rightarrow\infty}\sqrt{2\pi}\theta^*\int_{(-\theta^*n\mu_{n},\infty)}\frac{h}{\sigma_{n}}\left[  \phi\left(  \frac{x}{\lambda_{n}\sigma_{n}}\right)  S\left(  \frac{(x\sqrt{\psi^{\prime\prime}(\theta^*)}/\lambda_{n}+\mu_{n}\sqrt{n})}{h}\sqrt{n}\right)-\phi\left(  \frac{a_{n}}{\sigma_{n}}\right)  S\left(0\right)  \right]  e^{-x}dx\\
&  =\sqrt{2\pi}h\theta^*\int_{0}^{\infty}\left[  \phi\left(  0\right)S\left(  \frac{x}{\theta^*h}\right)  -\phi\left(  0\right)  S\left(0\right)  \right]  e^{-x}dx\\
&  =h\theta^*\sum_{j=0}^{\infty}\int_{j\theta^*h}^{(j+1)\theta^*h}\left[  S\left(  \frac{x}{\theta^*h}\right)  -S\left(0\right)  \right]  e^{-x}dx\\
&  =-\sum_{j=0}^{\infty}e^{-j\theta^*h}\int_{0}^{\theta^*h}e^{-x}xdx=\frac{\theta^*he^{-\theta^*h}}{1-e^{-\theta^*h}}-1.
\end{align*}
Plugging it into (\ref{limit_relation_between_truncated_probability_and_c_n_lattice}), we can get%
\[
\lim_{n\rightarrow\infty}\theta^*\sqrt{\psi^{\prime\prime}(\theta^*)2\pi n}e^{\theta^*nb-n\psi(\theta^*)}p(\tilde{F}_{u_{n}})=\frac{\theta^*he^{-\theta^*h}}{1-e^{-\theta^*h}}.
\]
$\square$
\end{proof}

\begin{lemma}[Berry-Ess\'{e}en expansion for the truncated distribution]
\label{edgeworth_expansion_tri_array} Following the notations in the proof of Lemma \ref{asymptotic_truncated}, the following hold.

\begin{description}
\item[(1)] For the non-lattice case, we have%
\[
\lim_{n\rightarrow\infty}\left(  \sqrt{n}\sup_{x}\left\vert H_{n}(x)-\Phi\left(  \frac{x}{\sigma_{n}}\right)  -\frac{m_{3,n}}{6\sigma_{n}^{3}\sqrt{n}}\left(  1-\left(  \frac{x}{\sigma_{n}}\right)  ^{2}\right)\phi\left(  \frac{x}{\sigma_{n}}\right)  \right\vert \right)  =0.
\]

\item[(2)] For the lattice case, we have%
\begin{align*}
&  \lim_{n\rightarrow\infty}\left(  \sqrt{n}\sup_{x}\left\vert H_{n}(x)-\Phi\left(  \frac{x}{\sigma_{n}}\right)  -\frac{m_{3,n}}{6\sigma_{n}^{3}\sqrt{n}}\left(  1-\left(  \frac{x}{\sigma_{n}}\right)  ^{2}\right)\phi\left(  \frac{x}{\sigma_{n}}\right)  \right.  \right. \\
&  \left.  \left.  -\frac{h}{\sqrt{\psi^{\prime\prime}(\theta^*)}\sqrt{n}\sigma_{n}}\phi\left(  \frac{x}{\sigma_{n}}\right)  S\left(\frac{(x\sqrt{\psi^{\prime\prime}(\theta^*)}+\mu_{n}\sqrt{n})}{h}\sqrt{n}\right)  \right\vert \right)  =0.
\end{align*}

\end{description}
\end{lemma}

\begin{proof}{Proof of Lemma \ref{edgeworth_expansion_tri_array}:} (1): This proof is a modified version of the proof of Theorem 1 in Section XVI.4 of \cite{feller2008introduction} which proves the Berry-Ess\'{e}en expansion for the i.i.d. non-lattice distribution. We will prove it for triangular arrays (in the special case of truncated distributions with varying truncation levels).

Recall that $H_{n}$ is the distribution function of $W_{n}=\sum_{m=1}^{n}Y_{n,m}/\sqrt{n}$ and the first three moments of $Y_{n,m}$ are $EY_{n,1}=0$, $\sigma_{n}^{2}=EY_{n,1}^{2}=\psi_{n}^{\prime\prime}(\theta^*)/\psi^{\prime\prime}(\theta^*)$ and $m_{3,n}=EY_{n,1}^{3}=\psi_{n}^{\left(  3\right)  }(\theta^*)/\sqrt{\psi^{\prime\prime}(\theta^*)}^{3}$. The existence of the limits $\lim_{n\rightarrow\infty}\sigma_{n}^{2}=1$ and $\lim_{n\rightarrow\infty}m_{3,n}=\psi^{\left(  3\right)  }(\theta^*)/\sqrt{\psi^{\prime\prime}(\theta^*)}^{3}$ implies that when $n$ is large enough, we will have
\begin{equation}
\sigma_{n}^{2}\geq\frac{3}{4}, \label{bound_of_sigma_n}%
\end{equation}
and
\begin{equation}
\left\vert m_{3,n}\right\vert \leq\left\vert \frac{\psi^{\left(  3\right)}(\theta^*)}{\sqrt{\psi^{\prime\prime}(\theta^*)}^{3}}\right\vert+1. \label{bound_of_m_3n}%
\end{equation}
Without loss of generality, we assume that (\ref{bound_of_sigma_n}) and (\ref{bound_of_m_3n}) hold for all $n\in\mathbb{N}$.

Put
\[
G_{n}(x)=\Phi\left(  \frac{x}{\sigma_{n}}\right)  +\frac{m_{3,n}}{6\sigma_{n}^{3}\sqrt{n}}\left(  1-\left(  \frac{x}{\sigma_{n}}\right)  ^{2}\right)\phi\left(  \frac{x}{\sigma_{n}}\right)  .
\]
Then $G_{n}\left(  x\right)  $ satisfies the conditions of Lemma 2 in Section XVI.3 of \cite{feller2008introduction} with%
\[
G_{n}^{\prime}(x)=\frac{1}{\sigma_{n}}\phi\left(  \frac{x}{\sigma_{n}}\right)+\frac{m_{3,n}}{6\sigma_{n}^{3}\sqrt{n}}\left[  \phi\left(  \frac{x}{\sigma_{n}}\right)  \left(  \frac{1}{\sigma_{n}}\left(  \frac{x}{\sigma_{n}}\right)  ^{3}-\frac{3}{\sigma_{n}}\frac{x}{\sigma_{n}}\right)  \right]  ,
\]
and%
\begin{equation}
\gamma_{n}(\xi)=e^{-\frac{1}{2}\sigma_{n}^{2}\xi^{2}}\left(  1+\frac{m_{3,n}}{6\sqrt{n}}(i\xi)^{3}\right)  . \label{gamma_n}%
\end{equation}
(\ref{bound_of_sigma_n}) and (\ref{bound_of_m_3n}) imply that%
\[
\left\vert G_{n}^{\prime}(x)\right\vert \leq\frac{2}{\sqrt{6\pi}}+\frac{4}{9\sqrt{3}}\left(  \left\vert \frac{\psi^{\left(  3\right)  }(\theta^*)}{\sqrt{\psi^{\prime\prime}(\theta^*)}^{3}}\right\vert+1\right)  \left[  \frac{2}{\sqrt{3}}\sup_{x}(|x|^{3}\phi(x))+2\sqrt{3}\sup_{x}(|x|\phi(x))\right]  <\infty.
\]

For any $\epsilon>0$, we use the inequality (3.13) in Section XVI.3 of \cite{feller2008introduction} with $T=a\sqrt{n}$ where the constant $a$ is chosen so large that
\[
\frac{2}{\sqrt{6\pi}}+\frac{4}{9\sqrt{3}}\left(  \left\vert \frac{\psi^{\left(  3\right)  }(\theta^*)}{\sqrt{\psi^{\prime\prime}(\theta^*)}^{3}}\right\vert +1\right)  \left[  \frac{2}{\sqrt{3}}\sup_{x}(|x|^{3}\phi(x))+2\sqrt{3}\sup_{x}(|x|\phi(x))\right]  <\frac{\pi}{24}\epsilon a.
\]
Then we will get%
\begin{equation}
|H_{n}(x)-G_{n}(x)|\leq\frac{1}{\pi}\int_{-a\sqrt{n}}^{a\sqrt{n}}\frac{1}{|\xi|}\left\vert \varphi_{n}^{n}\left(  \frac{\xi}{\sqrt{n}}\right)-\gamma_{n}(\xi)\right\vert d\xi+\frac{\epsilon}{\sqrt{n}},
\label{difference_of_Hn_and_Gn}%
\end{equation}
where $\varphi_{n}(\xi)=Ee^{i\xi Y_{n,1}}$ is the characteristic function of $Y_{n,1}$. Next, we will show
\begin{equation}
\frac{1}{\pi}\int_{-a\sqrt{n}}^{a\sqrt{n}}\frac{1}{|\xi|}\left\vert\varphi_{n}^{n}\left(  \frac{\xi}{\sqrt{n}}\right)  -\gamma_{n}(\xi)\right\vert d\xi\leq\frac{100\epsilon}{\sqrt{n}} \label{estimation_integral}%
\end{equation}
for $n$ large enough. To achieve this, we need some precise estimations for $\varphi_{n}^{n}(\xi/\sqrt{n})$, which will be divided into two parts.

Part I: In this part, we will show for any $\delta>0$, $\exists$ $N(\delta)\in\mathbb{N}$ s.t.
\begin{equation}
q:=\sup_{n\geq N(\delta),\delta\leq|\xi|\leq a}|\varphi_{n}(\xi)|<1.\label{bound_away_from_1}%
\end{equation}

On the one hand, recall that $Y_{n,1}=(X_{n,1}^{(\theta^*)}-b-\mu_{n})/\sqrt{\psi^{\prime\prime}(\theta^*)}$ and $X_{n,1}^{(\theta^*)}$ has the distribution $F_{n}^{(\theta^*)}$ defined in (\ref{change_of_measure}). Thus%
\begin{align}
\varphi_{n}(\xi)  &  =Ee^{i\xi\frac{X_{n,1}^{(\theta^*)}-b-\mu_{n}}{\sqrt{\psi^{\prime\prime}(\theta^*)}}}=e^{-i\xi\frac{b+\mu_{n}}{\sqrt{\psi^{\prime\prime}(\theta^*)}}}\int e^{i\xi\frac{x}{\sqrt{\psi^{\prime\prime}(\theta^*)}}}dF_{n}^{(\theta^*)}(x)=e^{-i\xi\frac{b+\mu_{n}}{\sqrt{\psi^{\prime\prime}(\theta^*)}}}\int e^{i\xi\frac{x}{\sqrt{\psi^{\prime\prime}(\theta^*)}}}e^{\theta^*x-\psi_{n}(\theta^*)}d\tilde{F}_{u_{n}}\left(  x\right) \nonumber\\
&  =e^{-i\xi\frac{b+\mu_{n}}{\sqrt{\psi^{\prime\prime}(\theta^*)}}}\frac{E\left[  \left.  \exp\left(  \left(  \theta^*+\frac{i\xi}{\sqrt{\psi^{\prime\prime}(\theta^*)}}\right)  X\right)  \right\vert X\leq u_{n}\right]  }{E[e^{\theta^*X}|X\leq u_{n}]}\nonumber\\
&  =e^{-i\xi\frac{b+\mu_{n}}{\sqrt{\psi^{\prime\prime}(\theta^*)}}}\frac{E\left[  \exp\left(  \left(  \theta^*+\frac{i\xi}{\sqrt{\psi^{\prime\prime}(\theta^*)}}\right)  X\right)  I(X\leq u_{n})\right]  }{E[e^{\theta^*X}I(X\leq u_{n})]} \label{varphi_n}%
\end{align}
On the other hand, recall that the distribution $F$ of $X$ is non-lattice. We introduce a new distribution%
\[
dF^{(\theta^*)}(x)=\frac{e^{\theta^*x}}{E[e^{\theta^*X}]}dF(x)
\]
and we can see $F^{(\theta^*)}$ and $F$ have the same support, which means $F^{(\theta^*)}$ is also a non-lattice distribution. Let $X_{\infty}^{(\theta^*)}$ be a random variable with distribution $F^{(\theta^*)}$. The characteristic function of $X_{\infty}^{(\theta^*)}/\sqrt{\psi^{\prime\prime}\left(  \theta^*\right)  }$ is given by%
\begin{equation}
\varphi_{\infty}(\xi)=Ee^{i\xi\frac{X_{\infty}^{(\theta^*)}}{\sqrt{\psi^{\prime\prime}(\theta^*)}}}=\int e^{i\xi\frac{x}{\sqrt{\psi^{\prime\prime}(\theta^*)}}}\frac{e^{\theta^*x}}{E[e^{\theta^*X}]}dF(x)=\frac{E\left[  \exp\left(  \left(  \theta^*+\frac{i\xi}{\sqrt{\psi^{\prime\prime}(\theta^*)}}\right)  X\right)  \right]}{E[e^{\theta^*X}]}. \label{varphi_infinity}%
\end{equation}
Since $F^{(  \theta^*)  }$ is non-lattice, we have $\sup_{\delta\leq|\xi|\leq a}|\varphi_{\infty}(\xi)|<1$ for any $\delta>0$. Comparing (\ref{varphi_n}) and (\ref{varphi_infinity}), we have%
\begin{align}
&  \left\vert \varphi_{n}(\xi)-e^{-i\xi\frac{b+\mu_{n}}{\sqrt{\psi^{\prime\prime}(\theta^*)}}}\varphi_{\infty}(\xi)\right\vert \nonumber\\
&  =\left\vert \frac{E\left[  \exp\left(  \left(  \theta^*+\frac{i\xi}{\sqrt{\psi^{\prime\prime}(\theta^*)}}\right)  X\right)  \right]E[e^{\theta^*X}I(X>u_{n})]-E[e^{\theta^*X}]E\left[  \exp\left(\left(  \theta^*+\frac{i\xi}{\sqrt{\psi^{\prime\prime}(\theta^*)}}\right)  X\right)  I(X>u_{n})\right]  }{E[e^{\theta^*X}I(X\leq u_{n})]E[e^{\theta^*X}]}\right\vert \nonumber\\
&  \leq\frac{2E[e^{\theta^*X}]E[e^{\theta^*X}I(X>u_{n})]}{E[e^{\theta^*X}I(X\leq u_{n})]E[e^{\theta^*X}]}=\frac{2E[e^{\theta^*X}I(X>u_{n})]}{E[e^{\theta^*X}I(X\leq u_{n})]}\rightarrow0 \label{uniform_convergence}%
\end{align}
uniformly as $n\rightarrow\infty$. Therefore, $\exists$ $N(\delta)\in\mathbb{N}$, s.t.
\[
\left\vert \varphi_{n}(\xi)-e^{-i\xi\frac{b+\mu_{n}}{\sqrt{\psi^{\prime\prime}(\theta^*)}}}\varphi_{\infty}(\xi)\right\vert <\frac{1-\sup_{\delta\leq|\xi|\leq a}|\varphi_{\infty}(\xi)|}{2}%
\]
for any $n\geq N\left(  \delta\right)  $ and $\xi\in\mathbb{R}$, which implies%
\[
\sup_{n\geq N(\delta),\delta\leq|\xi|\leq a}|\varphi_{n}(\xi)|\leq\sup_{\delta\leq|\xi|\leq a}|\varphi_{\infty}(\xi)|+\frac{1-\sup_{\delta\leq|\xi|\leq a}|\varphi_{\infty}(\xi)|}{2}=\frac{1+\sup_{\delta\leq|\xi|\leq a}|\varphi_{\infty}(\xi)|}{2}<1.
\]

Part II: In this part, we will show there exists $\delta>0$ and $N\in\mathbb{N}$ s.t.%
\begin{equation}
\frac{1}{|\xi|}\left\vert \varphi_{n}^{n}\left(  \frac{\xi}{\sqrt{n}}\right)-\gamma_{n}(\xi)\right\vert \leq e^{-\frac{1}{8}\xi^{2}}\left(  \frac{\epsilon}{\sqrt{n}}|\xi|^{2}+\frac{1}{72n}\left(  \left\vert \frac{\psi^{\left(  3\right)  }(\theta^*)}{\sqrt{\psi^{\prime\prime}(\theta^*)}^{3}}\right\vert +1\right)  ^{2}|\xi|^{5}\right),n\geq N,|\xi|\leq\delta\sqrt{n}.
\label{estimation_of_integrand}%
\end{equation}
Define the function\footnote{All logarithms of complex numbers used in the sequel are defined by the Taylor series $\log(1+z)=\sum_{n=1}^{\infty}\left(  -1\right)  ^{n+1}z^{n}/n$ which is valid for $\left\vert z\right\vert <1$.}%
\[
h_{n}(\xi)=\log\varphi_{n}(\xi)+\frac{1}{2}\sigma_{n}^{2}\xi^{2}.
\]
To ensure $h_{n}(\xi)$ is well defined, we note that
\begin{align}
|\varphi_{n}(\xi)-1|  & \leq\left\vert \varphi_{n}(\xi)-e^{-i\xi\frac{b+\mu_{n}}{\sqrt{\psi^{\prime\prime}(\theta^{\ast})}}}\varphi_{\infty}(\xi)\right\vert +\left\vert e^{-i\xi\frac{b+\mu_{n}}{\sqrt{\psi^{\prime\prime}(\theta^{\ast})}}}\varphi_{\infty}(\xi)-e^{-i\xi\frac{b}{\sqrt{\psi^{\prime\prime}(\theta^{\ast})}}}\varphi_{\infty}(\xi)\right\vert \nonumber\\
& +\left\vert e^{-i\xi\frac{b}{\sqrt{\psi^{\prime\prime}(\theta^{\ast})}}}\varphi_{\infty}(\xi)-1\right\vert \nonumber\\
& \leq\left\vert \varphi_{n}(\xi)-e^{-i\xi\frac{b+\mu_{n}}{\sqrt{\psi^{\prime\prime}(\theta^{\ast})}}}\varphi_{\infty}(\xi)\right\vert +\left\vert e^{-i\xi\frac{\mu_{n}}{\sqrt{\psi^{\prime\prime}(\theta^{\ast})}}}-1\right\vert +\left\vert e^{-i\xi\frac{b}{\sqrt{\psi^{\prime\prime}(\theta^{\ast})}}}\varphi_{\infty}(\xi)-1\right\vert \label{bound_phi_n_around_0}
\end{align}
The uniform convergence (\ref{uniform_convergence}) ensures the first term of (\ref{bound_phi_n_around_0}) converges to $0$ uniformly for $\xi\in\mathbb{R}$. By the convergence of $\mu_n$ to $0$ stated in (\ref{limit_of_mu_n}), the second term of (\ref{bound_phi_n_around_0}) converges to $0$ uniformly for bounded $\xi$. The last term of (\ref{bound_phi_n_around_0}) is continuous at $0$. Therefore, there exists $\delta>0$ and $N\in\mathbb{N}$ s.t. $|\varphi_{n}(\xi)-1|\leq1/2$ for $|\xi|\leq\delta$ and $n\geq N$. Hence $h_{n}(\xi)$ is well defined for $|\xi|\leq\delta$ and $n\geq N$. By means of $h_{n}(\xi)$, we can see%
\[
\frac{1}{|\xi|}\left\vert \varphi_{n}^{n}\left(  \frac{\xi}{\sqrt{n}}\right)-\gamma_{n}(\xi)\right\vert =e^{-\frac{1}{2}\sigma_{n}^{2}\xi^{2}}\frac{1}{|\xi|}\left\vert \exp\left(  nh_{n}\left(  \frac{\xi}{\sqrt{n}}\right)\right)  -1-\frac{m_{3,n}}{6\sqrt{n}}(i\xi)^{3}\right\vert .
\]
By direct calculation, we have $h_{n}(0)=h_{n}^{\prime}(0)=h_{n}^{\prime\prime}(0)=0$, $h_{n}^{\left(  3\right)  }(0)=i^{3}m_{3,n}$ and%
\begin{align*}
h_{n}^{\left(  4\right)  }(\xi)  &  =\frac{\varphi_{n}^{\left(  4\right)}(\xi)}{\varphi_{n}(\xi)}-\frac{\varphi_{n}^{\left(  3\right)  }(\xi)\varphi_{n}^{\prime}(\xi)}{(\varphi_{n}(\xi))^{2}}-3\frac{\varphi_{n}^{\left(  3\right)  }(\xi)\varphi_{n}^{\prime}(\xi)+(\varphi_{n}^{\prime\prime}(\xi))^{2}}{(\varphi_{n}(\xi))^{2}}\\
&  +6\frac{\varphi_{n}^{\prime\prime}(\xi)(\varphi_{n}^{\prime}(\xi))^{2}}{(\varphi_{n}(\xi))^{3}}+6\left(  \frac{\varphi_{n}^{\prime}(\xi)}{\varphi_{n}(\xi)}\right)  ^{2}\frac{\varphi_{n}^{\prime\prime}(\xi)\varphi_{n}(\xi)-(\varphi_{n}^{\prime}(\xi))^{2}}{(\varphi_{n}(\xi))^{2}}.
\end{align*}
Let us show the boundness of $h_{n}^{\left(  4\right)  }(\xi)$. For the denominators, we already know that $\left\vert \varphi_{n}\left(  \xi\right)\right\vert \geq1/2$ for $\left\vert \xi\right\vert \leq\delta$ and $n\geq N$. For the numerators, notice that for $k\in\mathbb{N}$,
\begin{align}
\left\vert \varphi_{n}^{(k)}(\xi)\right\vert  &  =|E[(iY_{n,1})^{k}e^{i\xi Y_{n,1}}]|\leq E[|Y_{n,1}|^{k}]=\int\left\vert \frac{x-b-\mu_{n}}{\sqrt{\psi^{\prime\prime}(\theta^*)}}\right\vert ^{k}dF_{n}^{(\theta^*)}(x)\nonumber\\
&  =\int\left\vert \frac{x-b-\mu_{n}}{\sqrt{\psi^{\prime\prime}(\theta^*)}}\right\vert ^{k}e^{\theta^*x-\psi_{n}(\theta^*)}d\tilde{F}_{u_{n}}\left(  x\right)  =\frac{E[\left.  |X-b-\mu_{n}|^{k}e^{\theta^*X}\right\vert X\leq u_{n}]}{\sqrt{\psi^{\prime\prime}(\theta^*)}^{k}E[\left.  e^{\theta^*X}\right\vert X\leq u_{n}]}.
\label{bound_of_varphi^(k)_n}%
\end{align}
Dominated convergence theorem shows that
\[
\lim_{n\rightarrow\infty}\frac{E[\left.  |X-b-\mu_{n}|^{k}e^{\theta^*X}\right\vert X\leq u_{n}]}{\sqrt{\psi^{\prime\prime}(\theta^*)}^{k}E[\left.  e^{\theta^*X}\right\vert X\leq u_{n}]}=\frac{E[|X-b|^{k}e^{\theta^*X}]}{\sqrt{\psi^{\prime\prime}(\theta^*)}^{k}E[e^{\theta^*X}]},
\]
so the sequence%
\[
\left\{  \frac{E[\left.  |X-b-\mu_{n}|^{k}e^{\theta^*X}\right\vert X\leq u_{n}]}{\sqrt{\psi^{\prime\prime}(\theta^*)}^{k}E[\left.  e^{\theta^*X}\right\vert X\leq u_{n}]},n\geq1\right\}
\]
is bounded. Therefore, for each $k$, (\ref{bound_of_varphi^(k)_n}) implies that $\varphi_{n}^{(k)}(\xi)$ is bounded for all $n$ and all $\xi\in\mathbb{R}$. Then the estimations for both denominators and numerators imply that $h_{n}^{\left(  4\right)  }(\xi)$ is bounded when $|\xi|\leq\delta$ and $n\geq N$. By the fourth-order Taylor expansion with the Lagrange remainder, we conclude that $\exists$ $\delta>0$ and $N\in\mathbb{N}$ s.t.
\begin{equation}
\left\vert h_{n}\left(  \xi\right)  -\frac{1}{6}i^{3}m_{3,n}\xi^{3}\right\vert=\left\vert \frac{1}{24}h_{n}^{\left(  4\right)  }\left(  \hat{\xi}\right)\xi^{4}\right\vert \leq\epsilon\left\vert \xi\right\vert ^{3}
\label{Taylor_approximation_of_h_n}%
\end{equation}
for $n\geq N$ and $|\xi|\leq\delta$. Besides, by (\ref{Taylor_approximation_of_h_n}) and the bound (\ref{bound_of_m_3n}) of $m_{3,n}$, we can also see when $\delta$ is chosen small enough (without loss of generality, assume our choice of $\delta$ is already small enough), the following estimations hold%
\[
|h_{n}(\xi)|\leq\frac{1}{4}\xi^{2},\text{ }\left\vert \frac{1}{6}i^{3}m_{3,n}\xi^{3}\right\vert \leq\frac{1}{4}\xi^{2}%
\]
for $n\geq N$ and $|\xi|\leq\delta$.

Applying these estimations, we can see when $n\geq N$ and $|\xi|\leq\delta\sqrt{n}$, the following hold
\[
\left\vert nh_{n}\left(  \frac{\xi}{\sqrt{n}}\right)  -\frac{m_{3,n}}{6\sqrt{n}}(i\xi)^{3}\right\vert =n\left\vert h_{n}\left(  \frac{\xi}{\sqrt{n}}\right)  -\frac{m_{3,n}}{6}\left(  i\frac{\xi}{\sqrt{n}}\right)^{3}\right\vert \leq n\epsilon\left\vert \frac{\xi}{\sqrt{n}}\right\vert^{3}=\frac{\epsilon}{\sqrt{n}}|\xi|^{3},
\]%
\[
\left\vert nh_{n}\left(  \frac{\xi}{\sqrt{n}}\right)  \right\vert \leq n\frac{1}{4}\left(  \frac{\xi}{\sqrt{n}}\right)  ^{2}=\frac{\xi^{2}}{4},
\]
and%
\[
\left\vert \frac{m_{3,n}}{6\sqrt{n}}(i\xi)^{3}\right\vert =n\left\vert\frac{m_{3,n}}{6}\left(  i\frac{\xi}{\sqrt{n}}\right)  ^{3}\right\vert \leq n\frac{1}{4}\left(  \frac{\xi}{\sqrt{n}}\right)  ^{2}=\frac{\xi^{2}}{4}.
\]
Then by (2.8) in Section XVI.2 of \cite{feller2008introduction}, the bound (\ref{bound_of_sigma_n}) of $\sigma_{n}^{2}$ and the bound (\ref{bound_of_m_3n}) of $m_{3,n}$, when $n\geq N$ and $\left\vert\xi\right\vert \leq\delta\sqrt{n}$, we have%
\begin{align*}
\frac{1}{|\xi|}\left\vert \varphi_{n}^{n}\left(  \frac{\xi}{\sqrt{n}}\right)-\gamma_{n}(\xi)\right\vert  &  =e^{-\frac{1}{2}\sigma_{n}^{2}\xi^{2}}\frac{1}{|\xi|}\left\vert \exp\left(  nh_{n}\left(  \frac{\xi}{\sqrt{n}}\right)\right)  -1-\frac{m_{3,n}}{6\sqrt{n}}(i\xi)^{3}\right\vert \\
&  \leq e^{-\frac{1}{2}\frac{3}{4}\xi^{2}}\frac{1}{|\xi|}\left(\frac{\epsilon}{\sqrt{n}}|\xi|^{3}+\frac{1}{2}\left\vert \frac{m_{3,n}}{6\sqrt{n}}(i\xi)^{3}\right\vert ^{2}\right)  e^{\frac{\xi^{2}}{4}}\\
&  =e^{-\frac{1}{8}\xi^{2}}\left(  \frac{\epsilon}{\sqrt{n}}|\xi|^{2}+\frac{1}{72n}\left\vert m_{3,n}\right\vert ^{2}|\xi|^{5}\right) \\
&  =e^{-\frac{1}{8}\xi^{2}}\left(  \frac{\epsilon}{\sqrt{n}}|\xi|^{2}+\frac{1}{72n}\left(  \left\vert \frac{\psi^{\left(  3\right)  }(\theta^*)}{\sqrt{\psi^{\prime\prime}(\theta^*)}^{3}}\right\vert +1\right)^{2}|\xi|^{5}\right) ,
\end{align*}
which proves (\ref{estimation_of_integrand}).

By now all the preparations have been finished. Let us prove (\ref{estimation_integral}), i.e.%
\[
\frac{1}{\pi}\int_{-a\sqrt{n}}^{a\sqrt{n}}\frac{1}{|\xi|}\left\vert\varphi_{n}^{n}\left(  \frac{\xi}{\sqrt{n}}\right)  -\gamma_{n}(\xi)\right\vert d\xi\leq\frac{100\epsilon}{\sqrt{n}}.
\]
for $n$ large enough. By Part I and II, we can choose $\delta>0$ and $N=N(\delta)\in\mathbb{N}$ s.t. (\ref{bound_away_from_1}) and (\ref{estimation_of_integrand}) hold. Without loss of generality, we can also assume that $N$ is chosen so large that when $n\geq N$%
\begin{equation}
\frac{1}{72\sqrt{n}}\left(  \left\vert \frac{\psi^{\left(  3\right)  }\left(\theta^*\right)  }{\left(  \sqrt{\psi^{\prime\prime}\left(\theta^*\right)  }\right)  ^{3}}\right\vert +1\right)  ^{2}\int_{\mathbb{R}}e^{-\frac{1}{8}\xi^{2}}\left\vert \xi\right\vert ^{5}d\xi<25\epsilon.
\label{estimation_of_O(1/sqrt(n))_term}%
\end{equation}
Then we divide the left hand side of (\ref{estimation_integral}) into two terms:%
\begin{align*}
&  \frac{1}{\pi}\int_{-a\sqrt{n}}^{a\sqrt{n}}\frac{1}{|\xi|}\left\vert\varphi_{n}^{n}\left(  \frac{\xi}{\sqrt{n}}\right)  -\gamma_{n}(\xi)\right\vert d\xi\\
&  =\frac{1}{\pi}\int_{-\delta\sqrt{n}}^{\delta\sqrt{n}}\frac{1}{|\xi|}\left\vert \varphi_{n}^{n}\left(  \frac{\xi}{\sqrt{n}}\right)  -\gamma_{n}(\xi)\right\vert d\xi+\frac{1}{\pi}\int_{\delta\sqrt{n}\leq|\xi|\leq a\sqrt{n}}\frac{1}{|\xi|}\left\vert \varphi_{n}^{n}\left(  \frac{\xi}{\sqrt{n}}\right)  -\gamma_{n}(\xi)\right\vert d\xi.
\end{align*}
For the first term, it follows by (\ref{estimation_of_integrand}) and (\ref{estimation_of_O(1/sqrt(n))_term}) that
\begin{align}
\frac{1}{\pi}\int_{-\delta\sqrt{n}}^{\delta\sqrt{n}}\frac{1}{|\xi|}\left\vert\varphi_{n}^{n}\left(  \frac{\xi}{\sqrt{n}}\right)  -\gamma_{n}(\xi)\right\vert d\xi &  \leq\frac{1}{\pi}\int_{-\delta\sqrt{n}}^{\delta\sqrt{n}}e^{-\frac{1}{8}\xi^{2}}\left(  \frac{\epsilon}{\sqrt{n}}|\xi|^{2}+\frac{1}{72n}\left(  \left\vert \frac{\psi^{\left(  3\right)  }(\theta^*)}{\sqrt{\psi^{\prime\prime}(\theta^*)}^{3}}\right\vert +1\right)^{2}|\xi|^{5}\right)  d\xi\nonumber\\
&  \leq\frac{1}{\pi}\int_{\mathbb{R}}e^{-\frac{1}{8}\xi^{2}}\left(  \frac{\epsilon}{\sqrt{n}}|\xi|^{2}+\frac{1}{72n}\left(  \left\vert \frac{\psi^{\left(  3\right)  }(\theta^*)}{\sqrt{\psi^{\prime\prime}(\theta^*)}^{3}}\right\vert +1\right)^{2}|\xi|^{5}\right)  d\xi\nonumber\\
&  \leq\frac{\epsilon}{\sqrt{n}}\frac{4}{\pi\sqrt{8\pi}}+\frac{25\epsilon}{\pi\sqrt{n}}<\frac{50\epsilon}{\sqrt{n}}. \label{estimation_part1}%
\end{align}
For the second term, it follows by (\ref{gamma_n}), (\ref{bound_of_sigma_n}), (\ref{bound_of_m_3n}) and (\ref{bound_away_from_1}) that%
\begin{align}
&  \frac{1}{\pi}\int_{\delta\sqrt{n}\leq|\xi|\leq a\sqrt{n}}\frac{1}{|\xi|}\left\vert \varphi_{n}^{n}\left(  \frac{\xi}{\sqrt{n}}\right)  -\gamma_{n}(\xi)\right\vert d\xi\nonumber\\
&  \leq\frac{1}{\pi}\int_{\delta\sqrt{n}\leq|\xi|\leq a\sqrt{n}}\frac{1}{\delta\sqrt{n}}\left(  q^{n}+e^{-\frac{3}{8}\xi^{2}}\left(  1+\frac{|\xi|^{3}}{6\sqrt{n}}\left(  \left\vert \frac{\psi^{\left(  3\right)}(\theta^*)}{\sqrt{\psi^{\prime\prime}(\theta^*)}^{3}}\right\vert+1\right)  \right)  \right)  d\xi\nonumber\\
&  \leq\frac{2}{\pi\delta\sqrt{n}}\left(  q^{n}+e^{-\frac{3}{8}\delta^{2}n}\left(  1+\frac{(a\sqrt{n})^{3}}{6\sqrt{n}}\left(  \left\vert \frac{\psi^{\left(  3\right)  }(\theta^*)}{\sqrt{\psi^{\prime\prime}(\theta^*)}^{3}}\right\vert +1\right)  \right)  \right)  (a-\delta)\sqrt{n}\nonumber\\
&  =\frac{2\left(  a-\delta\right)  }{\pi\delta}\left(  q^{n}+e^{-\frac{3}{8}\delta^{2}n}\left(  1+\frac{a^{3}n}{6}\left(  \left\vert \frac{\psi^{\left(  3\right)  }(\theta^*)}{\sqrt{\psi^{\prime\prime}(\theta^*)}^{3}}\right\vert +1\right)  \right)  \right)  <\frac{50\epsilon}{\sqrt{n}} \label{estimation_part2}%
\end{align}
when $n$ is large enough. Combining the estimations (\ref{estimation_part1}) and (\ref{estimation_part2}), we have%
\[
\frac{1}{\pi}\int_{-a\sqrt{n}}^{a\sqrt{n}}\frac{1}{|\xi|}\left\vert\varphi_{n}^{n}\left(  \frac{\xi}{\sqrt{n}}\right)  -\gamma_{n}(\xi)\right\vert d\xi\leq\frac{100\epsilon}{\sqrt{n}}.
\]
So from (\ref{difference_of_Hn_and_Gn}), we can see%
\[
\left\vert H_{n}\left(  x\right)  -G_{n}\left(  x\right)  \right\vert\leq\frac{101\epsilon}{\sqrt{n}}.
\]
Since $\epsilon$ is chosen arbitrarily, we finally get%
\begin{align*}
&  \lim_{n\rightarrow\infty}\left(  \sqrt{n}\sup_{x}\left\vert H_{n}\left(x\right)  -\Phi\left(  \frac{x}{\sigma_{n}}\right)  -\frac{m_{3,n}}{6\sigma_{n}^{3}\sqrt{n}}\left(  1-\left(  \frac{x}{\sigma_{n}}\right)^{2}\right)  \phi\left(  \frac{x}{\sigma_{n}}\right)  \right\vert \right) \\
&  =\lim_{n\rightarrow\infty}\sqrt{n}\left\vert H_{n}\left(  x\right)-G_{n}\left(  x\right)  \right\vert =0.
\end{align*}

(2): This proof is a modified version of the proof of Theorem 1 in Section 43 of \cite{gnedenko1968limit} which proves the Berry-Ess\'{e}en expansion for the i.i.d. lattice distribution. We will prove it for triangular arrays (in the special case of truncated distributions with varying truncation levels).

Recall that $H_{n}$ is the distribution function of $W_{n}=\sum_{m=1}^{n}Y_{n,m}/\sqrt{n}$, $Y_{n,m}=(X_{n,m}^{(\theta^*)}-b-\mu_{n})/\sqrt{\psi^{\prime\prime}(\theta^*)}$ and $X_{n,m}^{(\theta^*)}$ has the distribution $F_{n}^{(\theta^*)}(x)$ defined by $dF_{n}^{(\theta^*)}(x)=e^{\theta^*x-\psi_{n}(\theta^*)}d\tilde{F}_{u_{n}}(x)$. Therefore, $X_{n,m}^{(\theta^*)}$ has the same support as $X|X\leq u_{n}$ (i.e., the distribution of $X$ conditional on $X\leq u_{n}$) which also has a lattice distribution. First, we need to show that $h$ is also the span of $X|X\leq u_{n}$ when $n$ is large enough. Suppose the support of $X$ is $\{c_{m},m\in\mathcal{M}\}$ with $c_m<c_{m+1}$ where $\mathcal{M}$ is a consecutive subset of $\mathbb{Z}$. Consider the distances between two adjacent points, i.e. $y_{m}=c_{m}-c_{m-1}$. By the definition of the span, $h$ is the largest number such that $y_{m}/h$ are integer for all $m$. Now suppose that $c_{m_{n}}$ is the largest number satisfying $c_{m_{n}}\leq u_{n}$. Then the support of $X|X\leq u_{n}$ is $\{c_{m},m\leq m_{n},m\in\mathcal{M}\}$. So the span of $X|X\leq u_{n}$ is the largest number $h_{n}$ such that $y_{m}/h_{n}$ are integer for all $m\leq m_{n}$. Clearly $h_{n}\geq h$ so we can write $h_{n}$ as $h_{n}=r_{n}h$ for some $r_{n}\geq1$. Note that $r_{n}$ must be rational otherwise $y_{m}/h_{n}=(y_{m}/h)/r_{n}$ will be irrational, which contradicts that $y_{m}/h_{n}$ are integer for all $m\leq m_{n}$. Further, we claim that $r_{n}$ must be integer. Otherwise, we can assume that $r_{n}=q_{1,n}/q_{2,n}$ for some $q_{1,n}\in\mathbb{N}$, $q_{2,n}\in\mathbb{N}$ with $(q_{1,n},q_{2,n})=1$ and $q_{2,n}>1$. Thus $y_{m}/h_{n}=q_{2,n}(y_{m}/h)/q_{1,n}$ are all integer for all $m\leq m_{n}$. By $(q_{1,n},q_{2,n})=1$, we know that $q_{1,n}|(y_{m}/h)$, i.e. $y_{m}/(q_{1,n}h)$ are all integer for all $m\leq m_{n}$. But $q_{1,n}h>h_{n}=p_{n}h$, which contradicts the definition of the span $h_{n}$. So $r_{n}$ must be integer. Besides, since $u_{n}\uparrow\infty$ and $h$ is the largest number such that $y_{m}/h$ are integer for all $m$, we must have $r_{1}\geq r_{2}\geq r_{3}\geq\cdots$ and $\lim_{n\rightarrow\infty}r_{n}=1$. Thus, when $n$ is large enough, $r_{n}\equiv1$, which means $h$ is the span of $X|X\leq u_{n}$ for $n$ large enough.

Without loss of generality, we assume that $h$ is the span of $X|X\leq u_{n}$ for all $n$, and thus the span of $X_{n,m}^{(\theta^*)}$. Since $b$ is in the lattice of $X$ because of $0<P(X=b)<1$, we can see $Y_{n,m}$ has the lattice distribution taking values in $\{-\mu_{n}/\sqrt{\psi^{\prime\prime}(\theta^*)}+mh/\sqrt{\psi^{\prime\prime}(\theta^*)},m\in\mathbb{Z}\}$ with the span $h/\sqrt{\psi^{\prime\prime}(\theta^*)}$. Besides, just as the non-lattice case, we assume that (\ref{bound_of_sigma_n}) and (\ref{bound_of_m_3n}) hold for all $n\in\mathbb{N}$.

Recall that $S(x)=-x+1/2$ is a function with period $1$. By Fourier expansion, it can be write as $S(x)=\sum_{j=1}^{\infty}\sin(2\pi jx)/j\pi$ for non-integer values. Then, the characteristic function (Fourier transform) of
\begin{equation}
D_{n}(x):=\frac{h}{\sqrt{\psi^{\prime\prime}(\theta^*)}\sqrt{n}\sigma_{n}}\phi\left(  \frac{x}{\sigma_{n}}\right)  S\left(  \frac{(x\sqrt{\psi^{\prime\prime}(\theta^*)}+\mu_{n}\sqrt{n})}{h}\sqrt{n}\right)  \label{D_n}%
\end{equation}
is given by%
\begin{align}
d_{n}(\xi)  &  =\int e^{i\xi x}dD_{n}(x)=-i\xi\int e^{i\xi x}D_{n}(x)dx\nonumber\\
&  =-i\xi\int\left[  e^{i\xi x}\frac{h}{\sqrt{\psi^{\prime\prime}(\theta^*)}\sqrt{n}\sigma_{n}}\phi\left(  \frac{x}{\sigma_{n}}\right)\sum_{j=1}^{\infty}\frac{\sin\left(  \tau j\sqrt{n}\left(  x+\frac{\mu_{n}\sqrt{n}}{\sqrt{\psi^{\prime\prime}(\theta^*)}}\right)  \right)}{\pi j}\right]  dx\nonumber\\
&  =\frac{-2i\xi}{\tau\sqrt{n}\sigma_{n}}\sum_{j=1}^{\infty}\frac{1}{j}\int\left[  e^{i\xi x}\phi\left(  \frac{x}{\sigma_{n}}\right)  \sin\left(\tau j\sqrt{n}\left(  x+\frac{\mu_{n}\sqrt{n}}{\sqrt{\psi^{\prime\prime}(\theta^*)}}\right)  \right)  \right]  dx\nonumber\\
&  =\frac{-\xi}{\tau\sqrt{n}\sigma_{n}}\sum_{j\in\mathbb{Z},j\neq0}\frac{1}{j}\int\left[  \exp\left(  i\left(  \xi x+\tau j\sqrt{n}\left(  x+\frac{\mu_{n}\sqrt{n}}{\sqrt{\psi^{\prime\prime}(\theta^*)}}\right)  \right)  \right)  \phi\left(  \frac{x}{\sigma_{n}}\right)\right]  dx\nonumber\\
&  =\frac{-\xi}{\tau\sqrt{n}}\sum_{j\in\mathbb{Z},j\neq0}\frac{1}{j}\exp\left(  \frac{i\tau jn\mu_{n}}{\sqrt{\psi^{\prime\prime}(\theta^*)}}\right)  \exp\left(  -\frac{1}{2}\sigma_{n}^{2}(\xi+\tau j\sqrt{n})^{2}\right)  , \label{d_n}%
\end{align}
where $\tau=2\pi\sqrt{\psi^{\prime\prime}(\theta^*)}/h$ and the third equality can be justified by the $L^p$ convergence of Fourier series in any period and the fast convergence rate of $\phi(\cdot)$:
\[
\sum_{j=-\infty}^{\infty}\phi(ja_{1}+a_{2})<\infty,\forall a_{1}\in\mathbb{R},a_{2}\in\mathbb{R},a_{1}\neq0.
\]

Define%
\[
G_{n}(x):=\Phi\left(  \frac{x}{\sigma_{n}}\right)  +\frac{m_{3,n}}{6\sigma_{n}^{3}\sqrt{n}}\left(  1-\left(  \frac{x}{\sigma_{n}}\right)  ^{2}\right)\phi\left(  \frac{x}{\sigma_{n}}\right)  +\frac{h}{\sqrt{\psi^{\prime\prime}(\theta^*)}\sqrt{n}\sigma_{n}}\phi\left(  \frac{x}{\sigma_{n}}\right)S\left(  \frac{(x\sqrt{\psi^{\prime\prime}(\theta^*)}+\mu_{n}\sqrt{n})}{h}\sqrt{n}\right)  .
\]
Recall that $H_{n}$ is the distribution function of $W_{n}=\sum_{m=1}^{n}Y_{n,m}/\sqrt{n}$, where $Y_{n,m}$ has the lattice distribution taking values in $\{-\mu_{n}/\sqrt{\psi^{\prime\prime}(\theta^*)}+mh/\sqrt{\psi^{\prime\prime}(\theta^*)},m\in\mathbb{Z}\}$. We can see both $H_{n}(x)$ and $G_{n}(x)$ have discontinuities only at $\{-\mu_{n}\sqrt{n}/\sqrt{\psi^{\prime\prime}(\theta^*)}+mh/\sqrt {n\psi^{\prime\prime}(\theta^*)},m\in\mathbb{Z}\}$. At the continuous points, $G_{n}(x)$ is also differentiable. By (\ref{bound_of_sigma_n}) and (\ref{bound_of_m_3n}), the derivative can be bounded by
\begin{align*}
|G_{n}^{\prime}(x)|  &  =\left\vert \frac{1}{\sigma_{n}}\phi\left(  \frac{x}{\sigma_{n}}\right)  +\frac{m_{3,n}}{6\sigma_{n}^{3}\sqrt{n}}\left[\phi\left(  \frac{x}{\sigma_{n}}\right)  \left(  \frac{1}{\sigma_{n}}\left(\frac{x}{\sigma_{n}}\right)  ^{3}-\frac{3}{\sigma_{n}}\frac{x}{\sigma_{n}}\right)  \right]  \right. \\
&  -\frac{h}{\sqrt{\psi^{\prime\prime}(\theta^*)}\sqrt{n}\sigma_{n}^{2}}\frac{x}{\sigma_{n}}\phi\left(  \frac{x}{\sigma_{n}}\right)  S\left(\frac{(x\sqrt{\psi^{\prime\prime}(\theta^*)}+\mu_{n}\sqrt{n})}{h}\sqrt{n}\right) \\
&  \left.  +\frac{1}{\sigma_{n}}\phi\left(  \frac{x}{\sigma_{n}}\right)S^{\prime}\left(  \frac{(x\sqrt{\psi^{\prime\prime}(\theta^*)}+\mu_{n}\sqrt{n})}{h}\sqrt{n}\right)  \right\vert \\
&  \leq A,
\end{align*}
for any $n\in\mathbb{N}$ and all continuity points of $G_{n}(x)$, where%
\begin{align*}
A  &  =\frac{2}{\sqrt{6\pi}}+\frac{4}{9\sqrt{3}}\left(  \left\vert \frac{\psi^{\left(  3\right)  }(\theta^*)}{\sqrt{\psi^{\prime\prime}(\theta^*)}^{3}}\right\vert +1\right)  \left[  \frac{2}{\sqrt{3}}\sup_{x}(|x|^{3}\phi(x))+2\sqrt{3}\sup_{x}(|x|\phi(x))\right] \\
&  +\frac{2h}{3\sqrt{\psi^{\prime\prime}(\theta^*)}}\sup_{x}(|x|\phi(x))+\frac{2}{\sqrt{6\pi}}.
\end{align*}
Fix $\epsilon>0$. We now apply Theorem 2 in Section 39 of \cite{gnedenko1968limit}, putting there $F(x)=H_{n}(x)$, $G(x)=G_{n}(x)$, $l=h/\sqrt{n\psi^{\prime\prime}(\theta^*)}$ and $T=(2K+1)\tau\sqrt{n}/2$ for $K\in\mathbb{N}$ large enough s.t. $Tl=(2K+1)\tau h/(2\sqrt{\psi^{\prime\prime}(\theta^*)})\geq c_{2}(2)$ and $2c_{1}(2)A/((2K+1)\tau)<\epsilon$. Condition 2 in this theorem can be verified by Theorem 6 in Section 3, chapter VI of \cite{petrov2012sums}, which holds for any fixed $n$. Then we can get%
\begin{equation}
|H_{n}(x)-G_{n}(x)|\leq\frac{1}{\pi}\int_{-\frac{(2K+1)\tau}{2}\sqrt{n}}^{\frac{(2K+1)\tau}{2}\sqrt{n}}\frac{1}{|\xi|}\left\vert \varphi_{n}^{n}\left(  \frac{\xi}{\sqrt{n}}\right)  -g_{n}(\xi)\right\vert d\xi+\frac{\epsilon}{\sqrt{n}}, \label{difference_of_Hn_and_Gn_lattice}%
\end{equation}
where $\varphi_{n}(\xi)=Ee^{i\xi Y_{n,1}}$ is the characteristic function of $Y_{n,1}$ and $g_{n}(\xi)=\gamma_{n}(\xi)+d_{n}(\xi)$ is the characteristic function of $G_{n}(x)$ with $\gamma_{n}(\xi)$ and $d_{n}(\xi)$ defined in (\ref{gamma_n}) and (\ref{d_n}), respectively. Next, we will show
\begin{equation}
\frac{1}{\pi}\int_{-\frac{(2K+1)\tau}{2}\sqrt{n}}^{\frac{(2K+1)\tau}{2}\sqrt{n}}\frac{1}{|\xi|}\left\vert \varphi_{n}^{n}\left(  \frac{\xi}{\sqrt{n}}\right)  -g_{n}(\xi)\right\vert d\xi\leq\frac{100\epsilon}{\sqrt{n}}
\label{estimation_integral_lattice}%
\end{equation}
for $n$ large enough. Then, since $\epsilon$ is arbitrary, we will get the desired conclusion, i.e.,
\[
\sup_{x\in\mathbb{R}}\sqrt{n}|H_{n}(x)-G_{n}(x)|\rightarrow0.
\]
The proof of (\ref{estimation_integral_lattice}) will be divided into three parts based on the division of the integral interval.

Part I: We notice that the estimation in Part II of the non-lattice case also holds in the lattice case. So $\exists$ $\delta>0$ and $N\in\mathbb{N}$ s.t. the (\ref{estimation_part1}) holds for $n>N$. Without loss of generality, we can also assume $\delta\leq\tau/2$. Thus, we have $|\xi+\tau j\sqrt{n}|\geq|\tau j\sqrt{n}|-|\xi|\geq(|j|-1/2)\tau\sqrt{n}$ for $|\xi|\leq\delta\sqrt{n}$. Combining the bound (\ref{bound_of_sigma_n}) for $\sigma_{n}^{2}$, we have the following estimation for $|d_{n}(\xi)/\xi|$ when $|\xi|\leq\delta\sqrt{n}$
\begin{align}
\left\vert \frac{d_{n}(\xi)}{\xi}\right\vert  &  =\left\vert \frac{-1}{\tau\sqrt{n}}\sum_{j\in\mathbb{Z},j\neq0}\frac{1}{j}\exp\left(  \frac{i\tau jn\mu_{n}}{\sqrt{\psi^{\prime\prime}(\theta^*)}}\right)  \exp\left(  -\frac{1}{2}\sigma_{n}^{2}(\xi+\tau j\sqrt{n})^{2}\right)  \right\vert \leq\frac{1}{\tau\sqrt{n}}\sum_{j\in\mathbb{Z},j\neq0}\exp\left(  -\frac{3}{8}\frac{(2|j|-1)^{2}}{4}\tau^{2}n\right)\nonumber\\
&  \leq\frac{2}{\tau\sqrt{n}}\sum_{j=1}^{\infty}\exp\left(  -\frac{3}{8}\frac{j}{4}\tau^{2}n\right)  =\frac{2\exp\left(  -\frac{3}{32}\tau^{2}n\right)  }{\tau\sqrt{n}\left(  1-\exp\left(  -\frac{3}{32}\tau^{2}n\right)  \right)  }. \label{estimation_of_d_n/xi}%
\end{align}
Thus, together with (\ref{estimation_part1}), we have%
\begin{align}
\frac{1}{\pi}\int_{-\delta\sqrt{n}}^{\delta\sqrt{n}}\frac{1}{|\xi|}\left\vert\varphi_{n}^{n}\left(  \frac{\xi}{\sqrt{n}}\right)  -g_{n}(\xi)\right\vert d\xi &  \leq\frac{1}{\pi}\int_{-\delta\sqrt{n}}^{\delta\sqrt{n}}\frac{1}{|\xi|}\left\vert \varphi_{n}^{n}\left(  \frac{\xi}{\sqrt{n}}\right)-\gamma_{n}(\xi)\right\vert d\xi+\frac{1}{\pi}\int_{-\delta\sqrt{n}}^{\delta\sqrt{n}}\frac{\left\vert d_{n}(\xi)\right\vert }{|\xi|}d\xi\nonumber\\
&  \leq\frac{50\epsilon}{\sqrt{n}}+\frac{1}{\pi}\frac{2\exp\left(  -\frac{3}{32}\tau^{2}n\right)  }{\tau\sqrt{n}\left(  1-\exp\left(  -\frac{3}{32}\tau^{2}n\right)  \right)  }2\delta\sqrt{n}<\frac{60\epsilon}{\sqrt{n}},
\label{lattice_result_part1}%
\end{align}
when $n$ is large enough.

Part II: We notice that the uniform convergence (\ref{uniform_convergence}) also holds in the lattice case, where $\varphi_{\infty}(\xi)$ is the characteristic function of a lattice random variable with the span $h/\sqrt{\psi^{\prime\prime}(\theta^*)}=2\pi/\tau$. By Lemma 4 in Section XV.1 of \cite{feller2008introduction}, we can get $\sup_{\delta\leq|\xi|\leq\tau/2}|\varphi_{\infty}(\xi)|<1$, where $\delta$ is chosen in Part I. Then by the uniform convergence (\ref{uniform_convergence}), $\exists$ $N\in\mathbb{N}$, s.t.
\[
\left\vert \varphi_{n}(\xi)-e^{-i\xi\frac{b+\mu_{n}}{\sqrt{\psi^{\prime\prime}(\theta^*)}}}\varphi_{\infty}(\xi)\right\vert <\frac{1-\sup_{\delta\leq|\xi|\leq\tau/2}|\varphi_{\infty}(\xi)|}{2}%
\]
for any $n\geq N$ and $\xi\in\mathbb{R}$, which implies%
\[
\sup_{n\geq N,\delta\leq|\xi|\leq\tau/2}|\varphi_{n}(\xi)|\leq q:=\sup_{\delta\leq|\xi|\leq\tau/2}|\varphi_{\infty}(\xi)|+\frac{1-\sup_{\delta\leq|\xi|\leq\tau/2}|\varphi_{\infty}(\xi)|}{2}=\frac{1+\sup_{\delta\leq|\xi|\leq\tau/2}|\varphi_{\infty}(\xi)|}{2}<1.
\]
Note that when $\delta\sqrt{n}\leq|\xi|\leq\tau\sqrt{n}/2$, the estimation (\ref{estimation_of_d_n/xi}) still holds. Finally, recall that the definition of $\gamma_{n}$ in (\ref{gamma_n}), the bound (\ref{bound_of_sigma_n}) of $\sigma_{n}^{2}$, and the bound (\ref{bound_of_m_3n}) of $m_{3,n}$. It follows that%
\begin{align}
&  \frac{1}{\pi}\int_{\delta\sqrt{n}\leq|\xi|\leq\tau\sqrt{n}/2}\frac{1}{|\xi|}\left\vert \varphi_{n}^{n}\left(  \frac{\xi}{\sqrt{n}}\right)-g_{n}(\xi)\right\vert d\xi\nonumber\\
&  \frac{1}{\pi}\int_{\delta\sqrt{n}\leq|\xi|\leq\tau\sqrt{n}/2}\left[  \frac{1}{|\xi|}\left\vert \varphi_{n}^{n}\left(  \frac{\xi}{\sqrt{n}}\right)\right\vert +\frac{1}{|\xi|}|\gamma_{n}(\xi)|+\frac{1}{|\xi|}|d_{n}(\xi)|\right]  d\xi \nonumber\\
&  \leq\frac{1}{\pi}\int_{\delta\sqrt{n}\leq|\xi|\leq\tau\sqrt{n}/2}\left[\frac{1}{\delta\sqrt{n}}\left(  q^{n}+e^{-\frac{3}{8}\xi^{2}}\left(1+\frac{|\xi|^{3}}{6\sqrt{n}}\left(  \left\vert \frac{\psi^{\left(  3\right)}(\theta^*)}{\sqrt{\psi^{\prime\prime}(\theta^*)}^{3}}\right\vert+1\right)  \right)  \right)  +\frac{2\exp\left(  -\frac{3}{32}\tau^{2}n\right)  }{\tau\sqrt{n}\left(  1-\exp\left(  -\frac{3}{32}\tau^{2}n\right)  \right)  }\right]  d\xi\nonumber\\
&  \leq\frac{1}{\pi}\int_{\delta\sqrt{n}\leq|\xi|\leq\tau\sqrt{n}/2}\left[\frac{1}{\delta\sqrt{n}}\left(  q^{n}+e^{-\frac{3}{8}\delta^{2}n}\left(1+\frac{\tau^{3}n}{48}\left(  \left\vert \frac{\psi^{\left(  3\right)}(\theta^*)}{\sqrt{\psi^{\prime\prime}(\theta^*)}^{3}}\right\vert+1\right)  \right)  \right)  +\frac{2\exp\left(  -\frac{3}{32}\tau^{2}n\right)  }{\tau\sqrt{n}\left(  1-\exp\left(  -\frac{3}{32}\tau^{2}n\right)  \right)  }\right]  d\xi\nonumber\\
&  \leq\frac{2\sqrt{n}}{\pi}\left(  \frac{\tau}{2}-\delta\right)  \left[\frac{1}{\delta\sqrt{n}}\left(  q^{n}+e^{-\frac{3}{8}\delta^{2}n}\left(1+\frac{\tau^{3}n}{48}\left(  \left\vert \frac{\psi^{\left(  3\right)}(\theta^*)}{\sqrt{\psi^{\prime\prime}(\theta^*)}^{3}}\right\vert+1\right)  \right)  \right)  +\frac{2\exp\left(  -\frac{3}{32}\tau^{2}n\right)  }{\tau\sqrt{n}\left(  1-\exp\left(  -\frac{3}{32}\tau^{2}n\right)  \right)  }\right] \nonumber\\
&  <\frac{20\epsilon}{\sqrt{n}} \label{lattice_result_part2}%
\end{align}
when $n$ is large enough.

Part III: The last part we need to estimate is
\begin{align*}
&  \frac{1}{\pi}\int_{\frac{\tau}{2}\sqrt{n}\leq|\xi|\leq\frac{(2K+1)\tau}{2}\sqrt{n}}\frac{1}{|\xi|}\left\vert \varphi_{n}^{n}\left(  \frac{\xi}{\sqrt{n}}\right)  -g_{n}(\xi)\right\vert d\xi\\
&  \leq\frac{1}{\pi}\int_{\frac{\tau}{2}\sqrt{n}\leq|\xi|\leq\frac{(2K+1)\tau}{2}\sqrt{n}}\frac{1}{|\xi|}\left\vert \varphi_{n}^{n}\left(  \frac{\xi}{\sqrt{n}}\right)  -d_{n}(\xi)\right\vert d\xi+\frac{1}{\pi}\int_{\frac{\tau}{2}\sqrt{n}\leq|\xi|\leq\frac{(2K+1)\tau}{2}\sqrt{n}}\frac{1}{|\xi|}\left\vert \gamma_{n}(\xi)\right\vert d\xi.
\end{align*}
Just as what we did in Part II, we can also show
\[
\frac{1}{\pi}\int_{\frac{\tau}{2}\sqrt{n}\leq|\xi|\leq\frac{(2K+1)\tau}{2}\sqrt{n}}\frac{1}{|\xi|}\left\vert \gamma_{n}(\xi)\right\vert d\xi=o\left(\frac{1}{\sqrt{n}}\right)  .
\]
So in this part, we aim at estimating
\begin{align*}
\frac{1}{\pi}\int_{\frac{\tau}{2}\sqrt{n}\leq|\xi|\leq\frac{(2K+1)\tau}{2}\sqrt{n}}\frac{1}{|\xi|}\left\vert \varphi_{n}^{n}\left(  \frac{\xi}{\sqrt{n}}\right)  -d_{n}(\xi)\right\vert d\xi & =\frac{1}{\pi}\sum_{k=1}^{K}\int_{\frac{(2k-1)\tau}{2}\sqrt{n}\leq|\xi|\leq\frac{(2k+1)\tau}{2}\sqrt{n}}\frac{1}{|\xi|}\left\vert \varphi_{n}^{n}\left(  \frac{\xi}{\sqrt{n}}\right)  -d_{n}(\xi)\right\vert d\xi\\
&  =\frac{1}{\pi}\sum_{k=1}^{K}\int_{\frac{(2k-1)\tau}{2}\leq|\xi|\leq\frac{(2k+1)\tau}{2}}\frac{1}{|\xi|}\left\vert \varphi_{n}^{n}(\xi)-d_{n}(\sqrt{n}\xi)\right\vert d\xi.
\end{align*}
For simplicity, we only show how to estimate%
\[
\frac{1}{\pi}\sum_{k=1}^{K}\int_{\frac{(2k-1)\tau}{2}\leq\xi\leq\frac{(2k+1)\tau}{2}}\frac{1}{|\xi|}\left\vert \varphi_{n}^{n}(\xi)-d_{n}(\sqrt{n}\xi)\right\vert d\xi.
\]
The other side
\[
\frac{1}{\pi}\sum_{k=1}^{K}\int_{\frac{(2k-1)\tau}{2}\leq-\xi\leq\frac{(2k+1)\tau}{2}}\frac{1}{|\xi|}\left\vert \varphi_{n}^{n}(\xi)-d_{n}(\sqrt{n}\xi)\right\vert d\xi
\]
can be estimated by the same method.

Put%
\[
I_{k}=\frac{1}{\pi}\int_{\frac{(2k-1)\tau}{2}\leq\xi\leq\frac{(2k+1)\tau}{2}}\frac{1}{|\xi|}\left\vert \varphi_{n}^{n}(\xi)-d_{n}(\sqrt{n}\xi)\right\vert d\xi.
\]
If we can show%
\begin{equation}
I_{k}=o\left(  \frac{1}{\sqrt{n}}\right)  \label{I_k=o(1/sqrt(n))}%
\end{equation}
for each $k$, we will have%
\[
\frac{1}{\pi}\int_{\frac{\tau}{2}\sqrt{n}\leq|\xi|\leq\frac{(2K+1)\tau}{2}\sqrt{n}}\frac{1}{|\xi|}\left\vert \varphi_{n}^{n}\left(  \frac{\xi}{\sqrt{n}}\right)  -g_{n}(\xi)\right\vert d\xi=o\left(  \frac{1}{\sqrt{n}}\right)  .
\]
We make the change of variables $\xi=z+k\tau$ for $I_{k}$. Here we recall that $\tau=2\pi\sqrt{\psi^{\prime\prime}(\theta^*)}/h$ and%
\[
\varphi_{n}(\xi)=Ee^{i\xi Y_{n,1}}=\sum_{m=-\infty}^{\infty}\exp\left(i\xi\left(  -\frac{\mu_{n}}{\sqrt{\psi^{\prime\prime}(\theta^*)}}+\frac{mh}{\sqrt{\psi^{\prime\prime}(\theta^*)}}\right)  \right)P\left(  Y_{n,1}=-\frac{\mu_{n}}{\sqrt{\psi^{\prime\prime}(\theta^*)}}+\frac{mh}{\sqrt{\psi^{\prime\prime}(\theta^*)}}\right)  ,
\]
and consequently%
\[
\varphi_{n}^{n}(z+k\tau)=\exp\left(  -\frac{ik\tau n\mu_{n}}{\sqrt{\psi^{\prime\prime}(\theta^*)}}\right)  \varphi_{n}^{n}(z).
\]
Therefore,
\begin{align*}
 I_{k}& =\frac{1}{\pi}\int_{-\frac{\tau}{2}}^{\frac{\tau}{2}}\frac{1}{|z+k\tau|}\left\vert \exp\left(  -\frac{ik\tau n\mu_{n}}{\sqrt{\psi^{\prime\prime}(\theta^*)}}\right)  \varphi_{n}^{n}(z)-d_{n}(\sqrt{n}(z+k\tau))\right\vert dz\\
&  =\frac{1}{\pi}\int_{-\frac{\tau}{2}}^{\frac{\tau}{2}}\frac{1}{|z+k\tau|}\left\vert \exp\left(  -\frac{ik\tau n\mu_{n}}{\sqrt{\psi^{\prime\prime}(\theta^*)}}\right)  \varphi_{n}^{n}(z)\right.  \\
&  \left.  +\frac{z+k\tau}{\tau}\sum_{j\in\mathbb{Z},j\neq0}\frac{1}{j}\exp\left(  \frac{i\tau jn\mu_{n}}{\sqrt{\psi^{\prime\prime}(\theta^*)}}\right)  \exp\left(  -\frac{1}{2}\sigma_{n}^{2}n(z+k\tau+\tau j)^{2}\right)  \right\vert dz.
\end{align*}
For $j\neq-k$, by the bound (\ref{bound_of_sigma_n}) of $\sigma_{n}^{2}$, we have%
\[
\exp\left(  -\frac{1}{2}\sigma_{n}^{2}n(z+k\tau+\tau j)^{2}\right)  \leq\exp\left(  -\frac{1}{2}\frac{3}{4}n\tau^{2}\left(  |k+j|-\frac{1}{2}\right)^{2}\right)  \leq\exp\left(  -\frac{3\tau^{2}n}{32}(2|k+j|-1)\right)  ,
\]
which implies the following estimation:%
\begin{align*}
&  \left\vert \sum_{j\in\mathbb{Z},j\neq0,-k}\frac{1}{j}\exp\left(  \frac{i\tau jn\mu_{n}}{\sqrt{\psi^{\prime\prime}(\theta^*)}}\right)  \exp\left(  -\frac{1}{2}\sigma_{n}^{2}n(z+k\tau+\tau j)^{2}\right)  \right\vert \\
&  \leq\sum_{j\in\mathbb{Z},j\neq0,-k}\exp\left(  -\frac{3\tau^{2}n}{32}(2|k+j|-1)\right)  \leq2\sum_{j=1}^{\infty}\exp\left(  -\frac{3\tau^{2}n}{32}j\right)  =\frac{2\exp\left(  -\frac{3\tau^{2}n}{32}\right)  }{1-\exp\left(  -\frac{3\tau^{2}n}{32}\right)  }.
\end{align*}
Hence,
\begin{align}
I_{k} &  \leq\frac{1}{\pi}\int_{-\frac{\tau}{2}}^{\frac{\tau}{2}}\frac{1}{|z+k\tau|}\left\vert \exp\left(  -\frac{ik\tau n\mu_{n}}{\sqrt{\psi^{\prime\prime}(\theta^*)}}\right)  \varphi_{n}^{n}(z)-\frac{z+k\tau}{\tau k}\exp\left(  -\frac{ik\tau n\mu_{n}}{\sqrt{\psi^{\prime\prime}(\theta^*)}}\right)  \exp\left(  -\frac{1}{2}\sigma_{n}^{2}nz^{2}\right)  \right\vert dz\nonumber\\
&  +\frac{2\exp\left(  -\frac{3\tau^{2}n}{32}\right)  }{\pi\left(1-\exp\left(  -\frac{3\tau^{2}n}{32}\right)  \right)  }\nonumber\\
&  =\frac{1}{\pi}\int_{-\frac{\tau}{2}}^{\frac{\tau}{2}}\frac{1}{|z+k\tau|}\left\vert \varphi_{n}^{n}(z)-\exp\left(  -\frac{1}{2}\sigma_{n}^{2}nz^{2}\right)  -\frac{z}{\tau k}\exp\left(  -\frac{1}{2}\sigma_{n}^{2}nz^{2}\right)  \right\vert dz+\frac{2\exp\left(  -\frac{3\tau^{2}n}{32}\right)  }{\pi\left(  1-\exp\left(  -\frac{3\tau^{2}n}{32}\right)\right)  }\nonumber\\
&  \leq\frac{2}{k\tau\pi}\int_{-\frac{\tau}{2}}^{\frac{\tau}{2}}\left\vert\varphi_{n}^{n}(z)-\exp\left(  -\frac{1}{2}\sigma_{n}^{2}nz^{2}\right)\right\vert dz+\frac{2}{\pi\tau^{2}k^{2}}\int_{-\frac{\tau}{2}}^{\frac{\tau}{2}}|z|\exp\left(  -\frac{1}{2}\sigma_{n}^{2}nz^{2}\right)  dz+\frac{2\exp\left(  -\frac{3\tau^{2}n}{32}\right)  }{\pi\left(  1-\exp\left(-\frac{3\tau^{2}n}{32}\right)  \right)  }.\label{I_k}%
\end{align}
For the first term, by a change of variable $\xi=\sqrt{n}\sigma_{n}z$, we have%
\begin{equation}
\frac{2}{k\tau\pi}\int_{-\frac{\tau}{2}}^{\frac{\tau}{2}}\left\vert\varphi_{n}^{n}(z)-\exp\left(  -\frac{1}{2}\sigma_{n}^{2}nz^{2}\right)\right\vert dz=\frac{2}{k\tau\pi\sqrt{n}\sigma_{n}}\int_{-\frac{\sigma_{n}\tau}{2}\sqrt{n}}^{\frac{\sigma_{n}\tau}{2}\sqrt{n}}\left\vert \varphi_{n}^{n}\left(  \frac{\xi}{\sqrt{n}\sigma_{n}}\right)  -\exp\left(  -\frac{1}{2}\xi^{2}\right)  \right\vert d\xi.\label{target_lattice_part3}%
\end{equation}
By Theorem 2 in Section 40 of \cite{gnedenko1968limit}, we have that
\[
\left\vert \varphi_{n}^{n}\left(  \frac{\xi}{\sqrt{n}\sigma_{n}}\right)-\exp\left(  -\frac{1}{2}\xi^{2}\right)  \right\vert \leq\frac{7}{6}\frac{|\xi|^{3}E|Y_{n,1}|^{3}}{\sigma_{n}^{3}\sqrt{n}}\exp\left(  -\frac{\xi^{2}}{4}\right)
\]
for $|\xi|\leq\sigma_{n}^{3}\sqrt{n}/(5E|Y_{n,1}|^{3})$. Since by dominated convergence theorem, we have%
\[
E|Y_{n,1}|^{3}=E\left\vert \frac{X_{n,m}^{(\theta^*)}-b-\mu_{n}}{\sqrt{\psi^{\prime\prime}(\theta^*)}}\right\vert ^{3}\rightarrow\frac{E[|X-b|^{3}e^{\theta^*X}]}{\sqrt{\psi^{\prime\prime}(\theta^*)}^{3}E[e^{\theta^*X}]}.
\]
Thus, without generality, we can assume
\[
E|Y_{n,1}|^{3}\leq\frac{2E[|X-b|^{3}e^{\theta^*X}]}{\sqrt{\psi^{\prime\prime}(\theta^*)}^{3}E[e^{\theta^*X}]}%
\]
for all $n\in\mathbb{N}$. Besides, recall the bound (\ref{bound_of_sigma_n}) of $\sigma_{n}^{2}$. So we can see
\[
\frac{\sigma_{n}^{3}\sqrt{n}}{5E|Y_{n,1}|^{3}}\geq\frac{3}{40}\frac{\sigma_{n}\sqrt{n}\sqrt{\psi^{\prime\prime}(\theta^*)}^{3}E[e^{\theta^*X}]}{E[|X-b|^{3}e^{\theta^*X}]}.
\]
Thus, when
\[
|\xi|\leq\frac{3}{40}\frac{\sigma_{n}\sqrt{n}\sqrt{\psi^{\prime\prime}(\theta^*)}^{3}E[e^{\theta^*X}]}{E[|X-b|^{3}e^{\theta^*X}]},
\]
we have%
\begin{equation}
\left\vert \varphi_{n}^{n}\left(  \frac{\xi}{\sqrt{n}\sigma_{n}}\right)-\exp\left(  -\frac{1}{2}\xi^{2}\right)  \right\vert \leq\frac{7}{6}\frac{|\xi|^{3}E|Y_{n,1}|^{3}}{\sigma_{n}^{3}\sqrt{n}}\exp\left(  -\frac{\xi^{2}}{4}\right)  \leq\frac{56|\xi|^{3}}{9\sqrt{3}\sqrt{n}}\frac{E[|X-b|^{3}e^{\theta^*X}]}{\sqrt{\psi^{\prime\prime}(\theta^*)}^{3}E[e^{\theta^*X}]}\exp\left(  -\frac{\xi^{2}}{4}\right)
.\label{estimation_integrand_lattice_part3}%
\end{equation}
If
\[
\frac{3}{40}\frac{\sigma_{n}\sqrt{n}\sqrt{\psi^{\prime\prime}(\theta^*)}^{3}E[e^{\theta^*X}]}{E[|X-b|^{3}e^{\theta^*X}]}>\frac{\sigma_{n}\tau}{2}\sqrt{n},
\]
we can directly estimate (\ref{target_lattice_part3}) by (\ref{estimation_integrand_lattice_part3}). Otherwise, we should divide (\ref{target_lattice_part3}) into%
\begin{align*}
&  \frac{2}{k\tau\pi\sqrt{n}\sigma_{n}}\int_{-\frac{\sigma_{n}\tau}{2}\sqrt{n}}^{\frac{\sigma_{n}\tau}{2}\sqrt{n}}\left\vert \varphi_{n}^{n}\left(\frac{\xi}{\sqrt{n}\sigma_{n}}\right)  -\exp\left(  -\frac{1}{2}\xi^{2}\right)  \right\vert d\xi\\
&  =\frac{2}{k\tau\pi\sqrt{n}\sigma_{n}}\int_{|\xi|\leq\frac{3}{40}\frac{\sigma_{n}\sqrt{n}\sqrt{\psi^{\prime\prime}(\theta^*)}^{3}E[e^{\theta^*X}]}{E[|X-b|^{3}e^{\theta^*X}]}}\left\vert\varphi_{n}^{n}\left(  \frac{\xi}{\sqrt{n}\sigma_{n}}\right)  -\exp\left(-\frac{1}{2}\xi^{2}\right)  \right\vert d\xi\\
&  +\frac{2}{k\tau\pi\sqrt{n}\sigma_{n}}\int_{\frac{3}{40}\frac{\sigma_{n}\sqrt{n}\sqrt{\psi^{\prime\prime}(\theta^*)}^{3}E[e^{\theta^*X}]}{E[|X-b|^{3}e^{\theta^*X}]}\leq|\xi|\leq\frac{\sigma_{n}\tau}{2}\sqrt{n}}\left\vert \varphi_{n}^{n}\left(  \frac{\xi}{\sqrt{n}\sigma_{n}}\right)  -\exp\left(  -\frac{1}{2}\xi^{2}\right)  \right\vert d\xi\\
&  =\frac{2}{k\tau\pi\sqrt{n}\sigma_{n}}\int_{|\xi|\leq\frac{3}{40}\frac{\sigma_{n}\sqrt{n}\sqrt{\psi^{\prime\prime}(\theta^*)}^{3}E[e^{\theta^*X}]}{E[|X-b|^{3}e^{\theta^*X}]}}\left\vert\varphi_{n}^{n}\left(  \frac{\xi}{\sqrt{n}\sigma_{n}}\right)  -\exp\left(-\frac{1}{2}\xi^{2}\right)  \right\vert d\xi\\
&  +\frac{2}{k\tau\pi}\int_{\frac{3}{40}\frac{\sqrt{\psi^{\prime\prime}(\theta^*)}^{3}E[e^{\theta^*X}]}{E[|X-b|^{3}e^{\theta^*X}]}\leq|\xi|\leq\frac{\tau}{2}}\left\vert \varphi_{n}^{n}\left(  \xi\right)-\exp\left(  -\frac{1}{2}\sigma_{n}^{2}n\xi^{2}\right)  \right\vert d\xi,
\end{align*}
and then estimate the first term by (\ref{estimation_integrand_lattice_part3}) and estimate the second term by the technique in Part II. Anyway, in both cases, we will always get the following estimation for (\ref{target_lattice_part3}):
\[
\frac{2}{k\tau\pi}\int_{-\frac{\tau}{2}}^{\frac{\tau}{2}}\left\vert\varphi_{n}^{n}(z)-\exp\left(  -\frac{1}{2}\sigma_{n}^{2}nz^{2}\right)\right\vert dz=O\left(  \frac{1}{n}\right)  .
\]
For the second term of (\ref{I_k}), we have%
\[
\frac{2}{\pi\tau^{2}k^{2}}\int_{-\frac{\tau}{2}}^{\frac{\tau}{2}}|z|\exp\left(  -\frac{1}{2}\sigma_{n}^{2}nz^{2}\right)  dz=\frac{4}{\pi\tau^{2}k^{2}}\int_{0}^{\frac{\tau}{2}}z\exp\left(  -\frac{1}{2}\sigma_{n}^{2}nz^{2}\right)  dz\leq\frac{8}{\pi\tau^{2}k^{2}\sigma_{n}^{2}n}=O\left(\frac{1}{n}\right)  .
\]
Overall, for (\ref{I_k}), we have the estimation%
\[
I_{k}=O\left(  \frac{1}{n}\right)  +O\left(  \frac{1}{n}\right)  +\frac{2\exp\left(  -\frac{3\tau^{2}n}{32}\right)  }{\pi\left(  1-\exp\left(-\frac{3\tau^{2}n}{32}\right)  \right)  }=o\left(  \frac{1}{\sqrt{n}}\right),
\]
i.e., (\ref{I_k=o(1/sqrt(n))}) holds, and consequently,
\begin{equation}
\frac{1}{\pi}\int_{\frac{\tau}{2}\sqrt{n}\leq|\xi|\leq\frac{(2K+1)\tau}{2}\sqrt{n}}\frac{1}{|\xi|}\left\vert \varphi_{n}^{n}\left(  \frac{\xi}{\sqrt{n}}\right)  -g_{n}(\xi)\right\vert d\xi=o\left(  \frac{1}{\sqrt{n}}\right)  <\frac{20\epsilon}{\sqrt{n}}\label{lattice_result_part3}%
\end{equation}
for $n$ large enough.

Plugging (\ref{lattice_result_part1}), (\ref{lattice_result_part2}) and (\ref{lattice_result_part3}) into (\ref{difference_of_Hn_and_Gn_lattice}), we get%
\[
|H_{n}(x)-G_{n}(x)|\leq\frac{1}{\pi}\int_{-\frac{(2K+1)\tau}{2}\sqrt{n}}^{\frac{(2K+1)\tau}{2}\sqrt{n}}\frac{1}{|\xi|}\left\vert \varphi_{n}^{n}\left(  \frac{\xi}{\sqrt{n}}\right)  -g_{n}(\xi)\right\vert d\xi+\frac{\epsilon}{\sqrt{n}}\leq\frac{101\epsilon}{\sqrt{n}},
\]
which means%
\[
\lim_{n\rightarrow\infty}\sqrt{n}\sup_{x\in\mathbb{R}}|H_{n}(x)-G_{n}(x)|=0
\]
by the arbitrary choice of $\epsilon$. $\square$
\end{proof}

\begin{proof}{Proof of Theorem \ref{light}:}
Here we only consider the non-lattice case. The lattice case can be proved by the same argument. By Lemma \ref{asymptotic_truncated}, we have%
\[
\lim_{n\rightarrow\infty}\theta^*\sqrt{\psi^{\prime\prime}(\theta^*)2\pi n}e^{\theta^*nb-n\psi(\theta^*)}p(\tilde{F}_{u_{n}})=1,
\]
i.e.,
\begin{equation}
p(\tilde{F}_{u_{n}})=\frac{e^{n\psi(\theta^*)-\theta^*nb}}{\theta^*\sqrt{\psi^{\prime\prime}(\theta^*)2\pi n}}(1+o(1)).
\label{equ_asymptotic_truncated}%
\end{equation}
Comparing the asymptotic properties (\ref{asymptotic_original_nonlattice}) for the original distribution $F$ and (\ref{equ_asymptotic_truncated}) for the truncated distributions $\tilde{F}_{u_{n}}$, we have%
\[
\frac{p(\tilde{F}_{u_{n}})-p(F)}{p(F)}=\frac{\frac{e^{n\psi(\theta^*)-\theta^*nb}}{\theta^*\sqrt{\psi^{\prime\prime}(\theta^*)2\pi n}}(1+o(1))-\frac{e^{n\psi(\theta^*)-\theta^*nb}}{\theta^*\sqrt{\psi^{\prime\prime}(\theta^*)2\pi n}}(1+o(1))}{\frac{e^{n\psi(\theta^*)-\theta^*nb}}{\theta^*\sqrt{\psi^{\prime\prime}(\theta^*)2\pi n}}(1+o(1))}=\frac{o\left(1\right)  }{1+o\left(  1\right)  }\rightarrow0,
\]
which means%
\[
p(\tilde{F}_{u_{n}})-p(F)=o(p(F)).
\]
$\square$
\end{proof}

\begin{proof}{Proof of Theorem \ref{light_p_tilde}:} We will use the same notations in the proof of Lemma \ref{asymptotic_truncated}. For the original distribution, we can get the following asymptotics by Theorem 3.7.4 in \cite{dembozeitouni1998}:%
\[
\lim_{n\rightarrow\infty}\theta^*\sqrt{\psi^{\prime\prime}(\theta^*)2\pi n}e^{\theta^*nb-n\psi(\theta^*)}\tilde{p}(F)=1
\]
for the non-lattice case, and%
\[
\lim_{n\rightarrow\infty}\theta^*\sqrt{\psi^{\prime\prime}(\theta^*)2\pi n}e^{\theta^*nb-n\psi(\theta^*)}\tilde{p}(F)=\frac{\theta^*h}{1-e^{-\theta^*h}}%
\]
for the lattice case.

For the truncated distribution, by the same argument in the proof of Lemma \ref{asymptotic_truncated}, we have%
\[
\theta^*\sqrt{\psi^{\prime\prime}(\theta^*)2\pi n}e^{\theta^*n(b+\mu_{n})-n\psi_{n}(\theta^*)}\tilde{p}(\tilde{F}_{u_{n}})=\sqrt{2\pi}\lambda_{n}\int_{[-\theta^*n\mu_{n},\infty)}\left(H_{n}\left(  \frac{x}{\lambda_{n}}\right)  -H_{n}(a_{n}-)\right)  e^{-x}dx,
\]
where $H_{n}(x-)=\lim_{y\uparrow x}H_{n}(y)$.\footnote{In (3.7.7) of \cite{dembozeitouni1998}, $F_n(0)$ should be $F_n(0-)$.} Contrast to (\ref{integration_by_parts}), here we obtain $H_{n}(a_{n}-)$ instead of $H_{n}(a_{n})$. This is because the integral interval is closed in the left. So we will obtain $H_{n}(a_{n}-)$ using integration by parts.

For the non-lattice case, we have%
\[
\lim_{n\rightarrow\infty}\left(  \sqrt{n}\sup_{x}\left\vert H_{n}(x)-\Phi\left(  \frac{x}{\sigma_{n}}\right)  -\frac{m_{3,n}}{6\sigma_{n}^{3}\sqrt{n}}\left(  1-\left(  \frac{x}{\sigma_{n}}\right)  ^{2}\right)\phi\left(  \frac{x}{\sigma_{n}}\right)  \right\vert \right)  =0.
\]
By taking the left limit, we can get%
\[
\lim_{n\rightarrow\infty}\left(  \sqrt{n}\sup_{x}\left\vert H_{n}(x-)-\Phi\left(  \frac{x}{\sigma_{n}}\right)  -\frac{m_{3,n}}{6\sigma_{n}^{3}\sqrt{n}}\left(  1-\left(  \frac{x}{\sigma_{n}}\right)  ^{2}\right)\phi\left(  \frac{x}{\sigma_{n}}\right)  \right\vert \right)  =0,
\]
which means $H_{n}(x)$ and $H_{n}(x-)$ admit the same expansion. So the same argument in the proof of Lemma \ref{asymptotic_truncated} can be applied here. Finally we will get%
\[
\lim_{n\rightarrow\infty}\theta^*\sqrt{\psi^{\prime\prime}(\theta^*)2\pi n}e^{\theta^*nb-n\psi(\theta^*)}\tilde{p}(\tilde{F}_{u_{n}})=1,
\]
i.e. the same asymptotic of $p(\tilde{F}_{u_{n}})$.

For the lattice case, we have%
\begin{align*}
&  \lim_{n\rightarrow\infty}\left(  \sqrt{n}\sup_{x}\left\vert H_{n}(x)-\Phi\left(  \frac{x}{\sigma_{n}}\right)  -\frac{m_{3,n}}{6\sigma_{n}^{3}\sqrt{n}}\left(  1-\left(  \frac{x}{\sigma_{n}}\right)  ^{2}\right)\phi\left(  \frac{x}{\sigma_{n}}\right)  \right.  \right.  \\
&  \left.  \left.  -\frac{h}{\sqrt{\psi^{\prime\prime}(\theta^*)}\sqrt{n}\sigma_{n}}\phi\left(  \frac{x}{\sigma_{n}}\right)  S\left(\frac{(x\sqrt{\psi^{\prime\prime}(\theta^*)}+\mu_{n}\sqrt{n})}{h}\sqrt{n}\right)  \right\vert \right)  =0.
\end{align*}
By taking the left limit, we can get%
\begin{align*}
&  \lim_{n\rightarrow\infty}\left(  \sqrt{n}\sup_{x}\left\vert H_{n}(x-)-\Phi\left(  \frac{x}{\sigma_{n}}\right)  -\frac{m_{3,n}}{6\sigma_{n}^{3}\sqrt{n}}\left(  1-\left(  \frac{x}{\sigma_{n}}\right)  ^{2}\right)\phi\left(  \frac{x}{\sigma_{n}}\right)  \right.  \right.  \\
&  \left.  \left.  -\frac{h}{\sqrt{\psi^{\prime\prime}(\theta^*)}\sqrt{n}\sigma_{n}}\phi\left(  \frac{x}{\sigma_{n}}\right)  S\left(\frac{(x\sqrt{\psi^{\prime\prime}(\theta^*)}+\mu_{n}\sqrt{n})}{h}\sqrt{n}-\right)  \right\vert \right)  =0,
\end{align*}
which means $H_{n}(x)$ and $H_{n}(x-)$ admit different expansions because of the discontinuity of $S(x)$. Paraphrasing the argument in the proof of Lemma \ref{asymptotic_truncated}, we can obtain%
\begin{align*}
\lim_{n\rightarrow\infty}\theta^*\sqrt{\psi^{\prime\prime}(\theta^*)2\pi n}e^{\theta^*nb-n\psi(\theta^*)}\tilde{p}(\tilde{F}_{u_{n}})  & =1+\sqrt{2\pi}h\theta^*\int_{0}^{\infty}\left[\phi\left(  0\right)  S\left(  \frac{x}{\theta^*h}\right)  -\phi\left(0\right)  S\left(  0-\right)  \right]  e^{-x}dx\\
& =1+h\theta^*\int_{0}^{\infty}\left[  S\left(  \frac{x}{\theta^*h}\right)  +\frac{1}{2}\right]  e^{-x}dx\\
& =1+h\theta^*\sum_{j=0}^{\infty}\int_{j\theta^*h}^{(j+1)\theta^*h}\left[  S\left(  \frac{x}{\theta^*h}\right)  +\frac{1}{2}\right]  e^{-x}dx\\
& =1+\sum_{j=0}^{\infty}e^{-j\theta^*h}\int_{0}^{\theta^*h}e^{-x}(\theta^*h-x)dx=\frac{\theta^*h}{1-e^{-\theta^*h}}.
\end{align*}

Finally, by the same argument in Theorem \ref{light}, in both cases, we will get%
\[
\tilde{p}(\tilde{F}_{u_{n}})-\tilde{p}(F)=o(\tilde{p}(F)).
\]
$\square$
\end{proof}




\section{Bias of Conditional Monte Carlo for Discrete Distributions.}\label{sec:append_bias_cond_MC}
Suppose $X_{1},X_{2},\ldots,X_{n}$ are $n$ i.i.d. discrete random variables with equal mass on each point of $\{x_{1},x_{2},\ldots,x_{N}\}$ ($X_i$'s represent the samples generated from the empirical distribution in our experiments). We want to estimate the rare event probability $P(S_{n}>x)$ where $S_{n}=X_{1}+\cdots+X_{n}$. We use the estimator $Z$ proposed in (1.3) of \cite{asmussen2006improved} which is unbiased for estimating $nP(S_{n}>x,M_{n} \text{ is uniquely attained, } M_{n}=X_{n})$, where $M_{n}=\max_{1\leq i\leq n}X_{n}$. Note that if $X_i$'s were continuously distributed, we have
\begin{equation}
P(S_{n}>x)=nP(S_{n}>x,M_{n} \text{ is uniquely attained, } M_{n}=X_{n}) \label{validity_CMC},
\end{equation}
i.e., $Z$ is also unbiased for the target probability. However, when $X_i$'s follow a discrete distribution, (\ref{validity_CMC}) no longer holds and there will be a small bias if we still use such unbiased estimator of $nP(S_{n}>x,M_{n} \text{ is uniquely attained, } M_{n}=X_{n})$ to estimate $P(S_{n}>x)$. In the following, we analyze the magnitude of that bias and conclude that the bias is negligible compared to the magnitude of the target probability $P(S_{n}>x)$.

Note that
\begin{align*}
P(S_{n} >x) & =P(S_{n}>x,M_{n}\text{ is uniquely attained})+P(S_{n}>x,M_{n}\text{ is not uniquely attained})\\
& =\sum_{i=1}^{n}P(S_{n}>x,M_{n}\text{ is uniquely attained, }M_{n}=X_{i})+P(S_{n}>x,M_{n}\text{ is not uniquely attained})\\
& =nP(S_{n}>x,M_{n}\text{ is uniquely attained, }M_{n}=X_{n})+P(S_{n}>x,M_{n}\text{ is not uniquely attained}).
\end{align*}
So the bias is given by%
\[
P(S_{n}>x)-nP(S_{n}>x,M_{n}\text{ is uniquely attained, }M_{n}=X_{n})=P(S_{n}>x,M_{n}\text{ is not uniquely attained}).
\]
To analyze the bias, we first note that%
\[
\text{bias}=P(S_{n}>x|M_{n}\text{ is not uniquely attained})P(M_{n}\text{ is not uniquely attained}).
\]
By routine combinatorics and calculus, we can get%
\begin{align*}
P(M_{n}\text{ is not uniquely attained})  & =1-P(M_{n}\text{ is uniquely attained})=1-\sum_{k=2}^{N}\binom{n}{1}\frac{1}{N}\left[  \frac{(k-1)}{N}\right]  ^{n-1}\\
& =1-\frac{n}{N^{n}}\sum_{k=1}^{N-1}k^{n-1}\leq1-\frac{n}{N^{n}}\sum_{k=1}^{N-1}\int_{k-1}^{k}x^{n-1}dx\\
& =1-\frac{n}{N^{n}}\int_{0}^{N-1}x^{n-1}dx\\
& =1-\frac{(N-1)^{n}}{N^{n}}=1-\left(  1-\frac{1}{N}\right)  ^{n}.
\end{align*}
So the bias is bounded by%
\[
0\leq \text{bias}\leq P(S_{n}>x|M_{n}\text{ is not uniquely attained})\left[1-\left(  1-\frac{1}{N}\right)  ^{n}\right]  .
\]
When $n\le100$ and $N\geq10^{6}$, we can get%
\[
0\leq \text{bias}\leq10^{-4}P(S_{n}>x|M_{n}\text{ is not uniquely attained}).
\]
If we are fine with the approximation that $P(S_{n}>x|M_{n}$ is not uniquely attained$)\approx P(S_{n}>x)$ when $N\ge 10^6$ and relatively small $n\le100$, then we can see the bias is pretty small compared to the true probability $P(S_{n}>x)$.

\end{APPENDICES}

\section*{Acknowledgments.}
We gratefully acknowledge support from the InnoHK initiative, the Government of the HKSAR, and Laboratory for AI-Powered Financial Technologies, as well as the Columbia SEAS Innovation Hub grant. A preliminary conference version of this work has appeared in \cite{huang2019impacts}.


\bibliographystyle{informs2014} 
\bibliography{citation.bib} 


\end{document}